\documentclass[9pt]{article}

\makeatletter
\DeclareRobustCommand{\qed}{%
  \ifmmode 
  \else \leavevmode\unskip\penalty9999 \hbox{}\nobreak\hfill
  \fi
  \quad\hbox{\qedsymbol}}
\newcommand{\openbox}{\leavevmode
  \hbox to.77778em{%
  \hfil\vrule
  \vbox to.675em{\hrule width.6em\vfil\hrule}%
  \vrule\hfil}}
\newcommand{\qedsymbol}{\openbox}
\newenvironment{proof}[1][\proofname]{\par
  \normalfont
  \topsep6\p@\@plus6\p@ \trivlist
  \item[\hskip\labelsep\itshape
    #1.]\ignorespaces
}{%
  \qed\endtrivlist
}
\newcommand{\proofname}{Proof}
\makeatother

\usepackage[utf8]{inputenc}

\usepackage{bm}

\usepackage{color}
\usepackage{latexsym}
\usepackage{dsfont}
\usepackage{amssymb}
\usepackage{comment}
\usepackage{graphicx}
\usepackage{amsmath,amsfonts,amssymb,theorem,euscript,array,enumerate,amsfonts,mathrsfs}
\usepackage{appendix}
\usepackage[T1]{fontenc}
\usepackage{babel}
\numberwithin{equation}{section}
\usepackage{bbm}
\usepackage{subfigure}
\usepackage{color}
\usepackage{stmaryrd}

\usepackage[algo2e,ruled,vlined]{algorithm2e} 
 \usepackage{algorithm}
 \usepackage{algorithmicx}
\usepackage{algpseudocode}
\usepackage{ marvosym }
\usepackage{bm}
\usepackage{xcolor, soul}

\usepackage[hidelinks]{hyperref}

\hypersetup{
  pdftitle={Quantitative finite-population approximation of heterogeneous linear-quadratic mean-field control with common noise},
  pdfauthor={Aqib Ahmed},
  pdfsubject={Stochastic control, mean-field systems, and finite-population approximation},
  pdfkeywords={linear-quadratic control, mean-field control, non-exchangeable systems, stochastic Riccati equations, finite-population approximation}
}

\usepackage{etoolbox}

\makeatletter
\patchcmd{\thebibliography}
  {\sloppy}
  {\sloppy
   \setlength{\itemsep}{6pt}
   \setlength{\parskip}{0.2pt}}
  {}{}
\makeatother

\usepackage{mathtools}
\mathtoolsset{showonlyrefs}

\usepackage{comment}

\usepackage{tikz}
\usetikzlibrary{fit,matrix,chains,positioning,decorations.pathreplacing,arrows}

\usepackage{mathtools}

\usepackage{geometry}
\usepackage{algpseudocode}
\usepackage{algorithm}

\usepackage[normalem]{ulem}

\def \b1{\bf{1}}

\def \N{\mathbb{N}}
\def \R{\mathbb{R}}

\def \E{\mathbb{E}}
\def \F{\mathbb{F}}

\def \P{\mathbb{P}}

\def \S{\mathbb{S}}

\def \H{\mathbb{H}}

\def \d{\mathrm{d}}

\def\esssup_#1{\underset{#1}{\mathrm{ess\,sup\, }}}

\def\argmin_#1{\underset{#1}{\mathrm{argmin\, }}}
\def\argmax_#1{\underset{#1}{\mathrm{argmax\, }}}

\def\dm#1{\frac{\delta}{\delta m}}

\def \Ac{{\cal A}}

\def \Fc{{\cal F}}

\def \Ic{{\cal I}}

\def \Lc{{\cal L}}
\def \Pc{{\cal P}}

\def\bx{{\boldsymbol x}}

\def \d{\mathrm{d}}

\def\beqs{\begin{eqnarray*}}
\def\enqs{\end{eqnarray*}}
\def\beq{\begin{eqnarray}}
\def\enq{\end{eqnarray}}

\newtheorem{Theorem}{Theorem}[section] 
\newtheorem{Definition}[Theorem]{Definition} 
\newtheorem{Proposition}[Theorem]{Proposition}
\newtheorem{Assumption}[Theorem]{Assumption}
\newtheorem{Lemma}[Theorem]{Lemma}
\newtheorem{Corollary}[Theorem]{Corollary}
\newtheorem{Remark}[Theorem]{Remark}

\title{Quantitative finite-population approximation of heterogeneous linear--quadratic mean-field control with common noise}

\author{Aqib AHMED\footnote{Reykjavik University, Department of Engineering, \sf aqib24 at ru.is. This author is supported by the Sustainability Institute and Forum (SIF), grant number 224089 backed by Landsnet, and from the Energy Research fund of Landsvirkjun, in Iceland} }

\date{}

\begin{document}

\maketitle

\begin{abstract}
We study the finite-population approximation of a heterogeneous linear--quadratic mean-field control problem with common noise and random coefficients. Starting from the centralized social planner problem for $N$ non-exchangeable agents, we derive its exact finite-dimensional stochastic Riccati system and identify the scaling of its diagonal, off-diagonal, affine, and scalar components. We then compare this system with the Hilbert-space-valued Riccati system of the continuum model. Our main estimates give quantitative convergence of the complete backward system, including its common-noise martingale integrands. The error separates coefficient and kernel consistency, common-noise approximation, interface effects, exact diagonal corrections, and the mass of the limiting interaction kernel on the discrete diagonal band. We propagate these estimates to the feedback gains and, by a cellwise coupling, to the optimal states, controls, and initial social values. Finally, a cellwise projection of the limiting representative-agent feedback yields a decentralized finite-population strategy whose optimality gap vanishes. Under the stated regularity assumptions, the error bounds are expressed through explicit approximation and concentration moduli. Quantitative bounds on these moduli yield algebraic rates, including an $N^{-1/2}$ rate for the backward system in finite-type regimes; the rates for optimal trajectories additionally reflect the available conditional moment bounds.
\end{abstract}

\vspace{5mm}


\vspace{5mm}



\section{Introduction}

Linear--quadratic stochastic control is one of the few settings in which optimal feedbacks and value functions can be characterized explicitly. With random coefficients, the relevant objects are backward stochastic Riccati equations (BSREs), initiated by \cite{Bismut1976LQ} and developed, among others, in \cite{Tang2003GeneralLQ,HuZhou2003IndefiniteSRE,Tang2015LQ}. The
linear--quadratic structure is equally important in McKean--Vlasov control, where it provides tractable models with random coefficients and common noise; see \cite{Pham2016LQConditional,BaseiPham2019WeakMartingale} and the monographs \cite{CarmonaDelarue2018a,CarmonaDelarue2018b}.

Classical mean-field models rely on exchangeability. Many networked populations, however, contain agents with persistent labels and heterogeneous, possibly asymmetric, interaction intensities. Graphons and related kernel descriptions provide a natural continuum language for such systems; see \cite{Lovasz2012,BorgsChayesLovaszSosVesztergombi2008} and, for stochastic systems and their finite approximations, \cite{BayraktarChakrabortyWu2023,BayraktarWu2023Concentration, BayraktarKim2024Concentration,JabinPoyatoSoler2025}. Graphon games and their linear--quadratic specializations are studied in \cite{CainesHuang2021GraphonGames,AurellCarmonaLauriere2022, BayraktarWuZhang2023,LackerSoret2023}. Common-noise graphon particle systems have recently been investigated in \cite{BayraktarHeKim2026}.

For control rather than games, a general non-exchangeable mean-field framework, its law invariance property, and its dynamic programming equation were developed in \cite{DeCrescenzoFuhrmanKharroubiPham2026}. The linear--quadratic model without common noise was solved in \cite{DeCrescenzoDeFeoPham2026}, while a stochastic maximum principle was derived in \cite{KharroubiMekkaouiPham2025}. The continuum problem used in the present paper is directly based on \cite{DeFeoMekkaoui2025}, which treats heterogeneous linear--quadratic mean-field control with common noise, proves well-posedness of the associated infinite-dimensional stochastic Riccati system, and gives financial applications. Related recent directions include controlled interaction structures \cite{Djete2025} and quantitative non exchangeable mean-field Markov decision processes with common noise
\cite{MekkaouiPham2026}.

Finite-network control approximation has also been studied directly. Gao and Caines \cite{GaoCaines2020GraphonControl} develop graphon-based approximations for deterministic linear--quadratic network control, while Dunyak and Caines \cite{DunyakCaines2026GraphonQNoise} treat linear systems with additive Q-noise. Xu, Gou and Huang \cite{XuGouHuang2025SocialOptima} construct centralized controls and asymptotically socially optimal decentralized strategies for a linear--quadratic graphon model with correlated noises. Cao and Lauri\`ere \cite{CaoLauriere2025GraphonMFC} study graphon mean-field control through forward--backward equations and finite-agent approximation. Our emphasis is on quantitative comparison of the exact finite and continuum stochastic Riccati systems with common-noise-adapted random coefficients, including their martingale integrands and the diagonal corrections induced by the finite embedding.

The question addressed here is different from solving the continuum problem itself. We start from the centralized social planner problem for $N$ heterogeneous agents and ask whether its exact stochastic Riccati system, optimizer, and value are quantitatively approximated by their continuum counterparts. This is not a direct consequence of a state-level propagation-of-chaos estimate. The diagonal blocks of the finite Riccati matrix and its off-diagonal blocks have different scalings; the off-diagonal embedding removes a band of positive measure $h_N$; random coefficients create common-noise martingale integrands; and the feedback comparison couples local multiplication operators with interaction kernels. These features must be controlled simultaneously.

The \textbf{Main Contributions} are :

\begin{enumerate}
\item We derive the exact finite-dimensional BSRE of the centralized $N$-agent problem. We separate its local quadratic, interaction, linear, and scalar components, denoted by $(K^N,\bar K^N,Y^N,q^N)$. The computation retains both the terminal diagonal correction and the diagonal remainder generated by the off-diagonal blocks.

\item We prove uniform operator, Hilbert--Schmidt, and bounded mean oscillation (BMO) estimates and quantitative convergence of all components of the stochastic Riccati system, including their common-noise martingale integrands. The error explicitly accounts for coefficient and kernel approximation and for the diagonal blocks removed by the embedding. The diagonal-band modulus vanishes under the natural Hilbert-space regularity of the limiting interaction component, without a universal algebraic rate under that regularity alone. The full convergence result uses the regularity and consistency conditions in Assumption~\ref{ass:reg_consistency}.

\item We show stability of the feedback maps and use a cellwise conditional coupling to compare the finite optimizer with independent copies of the continuum optimizer. This gives convergence of the centralized optimal states and controls and of the initial social values. Algebraic rates are stated only when all the relevant moduli and conditional moments have corresponding quantitative bounds.

\item We distinguish centralized and decentralized information. A cellwise projection of the limiting representative-agent feedback is admissible under decentralized finite information. Its social cost converges to the continuum value, and a squeeze argument proves that its gap relative both to the centralized and decentralized finite values vanishes.
\end{enumerate}

The quantitative assumptions are stronger than those needed for well-posedness. Piecewise Lipschitz regularity allows finitely many heterogeneous subpopulations; related block regularity is used in \cite{BayraktarChakrabortyWu2023,CaoLauriere2025GraphonMFC}. Our comparison uses strong $L^2$ kernel consistency and additional row/column and spatial moment estimates on the limiting backward solution. These solution estimates are not deduced from block regularity alone in the general model; they hold in the finite-type class described in Remark~\ref{rem:finite_type_models}. The diagonal-band modulus vanishes under the natural Hilbert-space regularity, but this alone gives no universal algebraic rate. For exactly sampled finite-type data on compatible grids, all error terms are controlled and the non-squared backward estimates have order $N^{-1/2}$. Stronger rates require stronger bounds on every error component; block compatibility alone does not remove the diagonal correction. The trajectory estimates additionally depend on conditional moment bounds, stated separately below.

The paper is organized as follows. Section~\ref{sec:notations} introduces the functional framework. Sections~\ref{sec:problem_discret} and~\ref{sec:solution_continuum} formulate and solve the continuum problem. Sections~\ref{sec:problem_discrete} and~\ref{sec:discrete_LQ} treat the finite-population problem and derive its exact Riccati representation. Section~\ref{sec:convergence} proves convergence of the backward systems. Sections~\ref{sec:feedback_convergence} and \ref{sec:state_convergence} propagate the estimates to feedbacks, states, controls, and values. Section~\ref{sec:decentralized_lift} establishes asymptotic optimality of the decentralized lifted feedback.

\section{Notations and functional framework}\label{sec:notations}
\begin{enumerate}
\item [$\bullet$] Throughout the paper, $I=[0,1]$ is endowed with its Borel sigma-field and Lebesgue measure $du$. The state and control dimension are $d\ge1$ and $m\ge1$ respectively. The common brownian motion has a fixed dimension $d_0\ge1$, while the idiosyncratic Brownian motion in one-dimensioned.

Euclidean norms and scalar products are denoted by $|\cdot|$ and $\langle\cdot,\cdot\rangle$ and we write $\mathbb S^d$ (resp. $\mathbb S_+^d$, $\mathbb S_{++}^d$) for the
symmetric (resp. non-negative, positive definite) $d\times d$ matrices. For matrices, $\langle A,B\rangle_{\mathrm F}:=\operatorname{Tr}(A^\top B)$
and $|A|_{\mathrm F}^2:=\langle A,A\rangle_{\mathrm F}$. Matrix norms are Frobenius norms unless the operator norm is explicitly indicated.

\item [$\bullet$] For a separable Hilbert space $E$, $L^p(I;E)$ has its usual meaning. We use the spatial spaces
\[
\mathsf H:=L^2(I;\mathbb R^d),\qquad
\mathsf U:=L^2(I;\mathbb R^m),\qquad
\overline{\mathsf H}:=L^2(I^2;\mathbb R^{d\times d}).
\]
The target space in an $L^2(I)$ or $L^2(I^2)$ norm is omitted when it is determined by the object being measured.

Their inner products are denoted by $\langle\cdot,\cdot\rangle_{\mathsf H}$
and $\langle\cdot,\cdot\rangle_{\overline{\mathsf H}}$. 
For Hilbert spaces $E_1,E_2$, $\mathcal L(E_1;E_2)$ and $\mathcal L_2(E_1;E_2)$ denote the bounded and Hilbert--Schmidt operators; the target space is omitted when $E_1=E_2$. 

\item[$\bullet$] Let $\mathbb G=(\mathcal G_s)_{s\in[a,b]}$ be a filtration. We use
\begin{align*}
S_{\mathbb G}^p([a,b];E)
&:=\left\{X:\ X\text{ is continuous and }\mathbb G\text{-adapted},\quad
\E\sup_{a\le s\le b}\|X_s\|_E^p<\infty\right\},\\
H_{\mathbb G}^2([a,b];E)
&:=\left\{Z:\ Z\text{ is }\mathbb G\text{-predictable},\quad
\E\int_a^b\|Z_s\|_E^2ds<\infty\right\},\\
S_{\mathbb G}^{\infty}([a,b];E) 
&:=\left\{X:\ X\text{ is continuous and }\mathbb G\text{-adapted}, \quad \|X\|_{S_{\mathbb G}^{\infty}([a,b];E)} :=\operatorname*{ess\,sup}_{\omega\sim\mathbb P} \sup_{a\le s\le b}\|X_s(\omega)\|_E<\infty \right\}
\end{align*}
Intervals equal to $[0,T]$ are omitted. For progressive coefficient and feedback fields, which need not be continuous, uniform estimates are written using essential suprema in time:
\[
\operatorname*{ess\,sup}_{(s,\omega)}
\|a_s\|_{L^2},
\qquad
\mathbb E\left[
\operatorname*{ess\,sup}_{s\in[0,T]}
\|a_s\|_{L^2}^2
\right].
\]
In the first expression, the essential supremum is taken with respect to $ds\otimes d\mathbb P$. In the second expression, the inner essential supremum is taken with respect to Lebesgue measure in time, for almost every $\omega$, and the resulting random variable is then integrated with respect to $\mathbb P$. When $S^p(L^\infty(I))$ is used for a backward field, it denotes the corresponding supremum-moment bound on a jointly measurable version, without requiring continuity in the spatial $L^\infty$ norm.

\item[$\bullet$] Let $\mathcal T_{a,b}(\mathbb G)$ denote the set of $\mathbb G$-stopping times with values in $[a,b]$. An adapted real-valued process $X$ is of class $(D)$ on $[a,b]$ if $X_\tau$ is integrable for every $\tau\in\mathcal T_{a,b}(\mathbb G)$ and $\lim_{R\to\infty} \sup_{\tau\in\mathcal T_{a,b}(\mathbb G)} \mathbb E\left[|X_\tau|\mathbf1_{\{|X_\tau|>R\}}\right]=0$.

\item[$\bullet$] We denote by $\mathcal P^0$ the progressive $\sigma$-field of $\mathbb F^0$:
\[
\mathcal P^0
:=
\left\{A\subset[0,T]\times\Omega:
A\cap([0,t]\times\Omega)\in
\mathcal B([0,t])\otimes\mathcal F_t^0
\text{ for every }t\in[0,T]\right\}.
\]
For a finite-dimensional space $E$, we write
\[
\begin{aligned}
L^\infty_{\mathbb F^0}(I\times[0,T];E)
:=\big\{a:\;&a\text{ is }
\mathcal P^0\otimes\mathcal B(I)\text{-measurable},\quad\operatorname*{ess\,sup}_{(t,\omega,u)}
|a_t(u,\omega)|_E<\infty\big\},
\end{aligned}
\]
where the essential supremum is taken with respect to $dt\otimes d\mathbb P\otimes du$.

\item[$\bullet$] For an $\mathbb F^0$-predictable process $Z$ with values in $\mathcal L_2(\mathbb R^{d_0};E)$, 
\[
\|Z\|_{\mathrm{BMO}(E)}^2
:=\sup_{\tau\in\mathcal T_{0,T}}
\left\|
\E\left[\int_\tau^T\|Z_s\|_{\mathcal L_2(\mathbb R^{d_0};E)}^2ds
\,\middle|\,\mathcal F_\tau^0\right]
\right\|_{L^\infty(\Omega)}
\]
where $\mathcal T_{0,T}$ is the set of $\mathbb F^0$-stopping times with values in $[0,T]$

\item[$\bullet$] For an essentially bounded measurable matrix field
$a:I\to\mathbb R^{r\times q}$, $M_a$ is the operator such as $M_a:L^2(I;\mathbb R^q)\to L^2(I;\mathbb R^r)$ with $(M_af)(u)=a(u)f(u)$. For $G\in L^2(I^2;\mathbb R^{r\times q})$, define
\[
(T_Gf)(u):=\int_I G(u,v)f(v)\,dv,
\qquad f\in L^2(I;\mathbb R^q).
\]
Then $T_G$ is Hilbert--Schmidt and $\|T_G\|_{\mathrm{op}}\le\|T_G\|_{\mathrm{HS}}=\|G\|_{L^2(I^2)}$. We set $G^\dagger(u,v):=G(v,u)^\top$; hence $T_G^*=T_{G^\dagger}$, and a square kernel is called symmetric when $G=G^\dagger$ a.e. We also use
\[
\|G\|_{\mathrm{row}}
:=\operatorname*{ess\,sup}_{u\in I}
\left(\int_I|G(u,v)|_{\mathrm F}^2dv\right)^{1/2},
\quad
\|G\|_{\mathrm{col}}
:=\operatorname*{ess\,sup}_{v\in I}
\left(\int_I|G(u,v)|_{\mathrm F}^2du\right)^{1/2}.
\]

\item[$\bullet$]For an integrable state process $X$ and a uniform label $U$, we write $\bar X_s(u):=\E[X_s\mid\mathcal F_s^0,U=u]$ for a jointly measurable version. The space $\mathcal P_2^{\lambda}(I\times\mathbb R^d)$ consists of probability measures on $I\times\mathbb R^d$ whose first marginal is Lebesgue measure and whose second moment is finite. Conditional laws are denoted by $\mathcal L(\,\cdot\mid\mathcal F_s^0)$.

\item[$\bullet$]For $N\ge1$, set $h_N=N^{-1}$, $u_i^N=i/N$, and $I_i^N=((i-1)/N,i/N]$. The state and control spaces are $\mathbb H_N=(\mathbb R^d)^N$ and $\mathbb U_N=(\mathbb R^m)^N$. For $\mathbf x^N=(x_i)_{i=1}^N\in\mathbb H_N$ and $\mathbf a^N=(a_i)_{i=1}^N\in\mathbb U_N$, set
\[
|\mathbf x^N|_N^2=h_N\sum_{i=1}^N|x_i|^2,\qquad
|\mathbf a^N|_N^2=h_N\sum_{i=1}^N|a_i|^2.
\]
For any finite-dimensional Hilbert space $E$, the step embedding is
\[
\mathsf E_N:E^N\longrightarrow L^2(I;E),\qquad
(\mathsf E_N\mathbf x^N)(u)
:=\sum_{i=1}^Nx_i\mathbf1_{I_i^N}(u).
\]
It is an isometry when $E^N$ is equipped with the weighted norm above. The orthogonal projection onto cellwise constant functions is
\[
\begin{aligned}
\mathsf P_N:L^2(I;E)&\longrightarrow L^2(I;E),\quad
(\mathsf P_Nf)(u)
:=\sum_{i=1}^N
\left(h_N^{-1}\int_{I_i^N}f(v)\,dv\right)\mathbf1_{I_i^N}(u).
\end{aligned}
\]
Its two-variable counterpart, 
\[(\mathsf P_N^{(2)}G)(u,v)
:=\sum_{i,j=1}^N
\left(h_N^{-2}\int_{I_i^N}\int_{I_j^N}G(x,y)\,dy\,dx\right)
\mathbf1_{I_i^N}(u)\mathbf1_{I_j^N}(v), G\in L^2(I^2;E)
\]
For a common-noise coefficient $\Sigma^{0,N}=(\Sigma_i^{0,N})_{i=1}^N$, we set $\|\Sigma^{0,N}\|_{N,\mathrm{HS}}^2 :=\sum_{k=1}^{d_0}|\Sigma^{0,N,k}|_N^2 =h_N\sum_{i=1}^N|\Sigma_i^{0,N}|_{\mathrm F}^2$, with $\Sigma_i^{0,N}\in\mathbb R^{d\times d_0}$. Scalar quadratic terms with common-noise coefficients gives $\langle\Sigma^0,A\Sigma^0\rangle_{\mathrm F} =\operatorname{Tr}\big((\Sigma^0)^\top A\Sigma^0\big)$. 

\item[$\bullet$]Define
\[
\mathcal D_N:=\bigcup_{i=1}^NI_i^N\times I_i^N,
\qquad
\Pi_N^\circ G:=\mathbf1_{\mathcal D_N^c}G,
\]
The set $\mathcal D_N$ is the union of the $N$ diagonal cell blocks
and has measure $h_N$. The map $\Pi_N^\circ$ removes these blocks.
This separates the diagonal Riccati blocks, encoded by $K^N$,
from the off-diagonal blocks, encoded by $\bar K^N$.\\
For a step kernel, $(\operatorname{diag}_NG^N)(u):=G_{ii}^N$ on $I_i^N$. The discrete operator associated with the blocks $G_{ij}^N$ is
\[
(\mathsf G^Nx^N)^i:=h_N\sum_{j=1}^NG_{ij}^Nx^j.
\]
Under the step embedding it agrees with $T_{G^N}$. We shall also use
the column quadratic norm
\[
a_i^N(G):=h_N^2\sum_{\ell=1}^N|G_{\ell i}^N|_{\mathrm F}^2,
\qquad
\sum_{i=1}^Na_i^N(G)=\|G^N\|_{L^2(I^2)}^2.
\]

\item[$\bullet$] For the common Brownian motion $W^0=(W^{0,1},\ldots,W^{0,d_0})$, we identify $\mathcal L_2(\mathbb R^{d_0};E)$ with $E^{d_0}$ by writing $Z= (Z^1,\ldots,Z^{d_0})$, with $\|Z\|_{\mathcal L_2(\mathbb R^{d_0};E)}^2 =\sum_{k=1}^{d_0}\|Z^k\|_E^2$. In particular, $\Sigma_s^0(u)\in\mathbb R^{d\times d_0}$. Products involving common-noise coordinates are summed over $k$:
\[
Z_s^K(u)\Sigma_s^0(u)
:=\sum_{k=1}^{d_0}Z_s^{K,k}(u)\Sigma_s^{0,k}(u)\in\mathbb R^d,
\qquad
\langle Z_s^Y(u),\Sigma_s^0(u)\rangle_{\mathrm F}
:=\sum_{k=1}^{d_0}
\langle Z_s^{Y,k}(u),\Sigma_s^{0,k}(u)\rangle\in\mathbb R.
\]
The same convention applies to $Z^{\bar K}$ and to the finite system. 

Constants denoted by $C$ may change from line to line but never depend
on $N$; their dependence on fixed model parameters is suppressed. 

\end{enumerate}

\section{Problem formulation of the continuum problem}\label{sec:problem_discret}

We introduce the continuum problem on a complete probability space $(\Omega,\mathcal F,\mathbb P)$ rich enough to support mutually independent random objects: a real-valued Brownian motion $W$, an $\mathbb R^{d_0}$ valued Brownian motion $W^0$, a label $U\sim\mathcal U(I)$, and an auxiliary random variable $\vartheta\sim\mathcal U([0,1])$. We set $\mathcal F_t^0 := \sigma(W_s^0:s\le t)^{\mathbb P}$, and $\mathcal F_t := \sigma \big( U,\vartheta,W_s,W_s^0:s\le t \big)^{\mathbb P}$. The auxiliary random variable $\vartheta$ is used only to generate the initial condition and is not observed separately by the controller. For an initial time $t\in[0,T]$ and an admissible $\mathbb R^d$-valued initial state $\xi$ (defined below), we write $\mathbb G^{t,\xi} := \big(\mathcal G_s^{t,\xi}\big)_{s\in[t,T]}$ and the relevant observation filtration on $[t,T]$ is $\mathcal G_s^{t,\xi} := \sigma\big( U,\xi, W_{r\wedge s},W^0_{r\wedge s}:r\in[0,T] \big)^{\mathbb P}$, $s\in[t,T]$.

\paragraph{State equation}
Given $s \in [t,T]$, an admissible initial condition $\xi$ and a control $\alpha \in \Ac_t$, we consider the controlled state equation
\begin{align}\label{eq:state_equations}
\begin{cases}
        \d X_s &= \big[ A_s(U) + B_s(U) X_s +  T_{G_s^B}(\bar{X}_s)(U) + C_s(U) \alpha_s \big] \d s  \\
    &\qquad + \big[ D_s(U) + E_s(U) X_s + T_{G_s^E}(\bar{X}_s)(U) + F_s(U) \alpha_s \big] \d W_s +\Sigma_s^0(U) \d W_s^0, \\
    X_t &= \xi,
\end{cases}
\end{align}
where we use the notation $\bar X_s(u):=\E[X_s\mid \Fc_s^0,U=u]$, $s\in[t,T]$, $u\in I$, and where for every $s \in [t,T]$ we defined 
\begin{align}
\begin{cases}
       T_{G_s^B}(\bar{X}_s)(U) := \int_I G_s^B(U,v) \E[X_s | \Fc^0_s, U=v] \d v, \quad \P-\text{a.s}, \\
       T_{G_s^{E}}(\bar{X}_s)(U) :=  \int_I G_s^E(U,v) \E[X_s | \Fc^0_s, U=v] \d v, \quad \P-\text{a.s}.
\end{cases}
\end{align}

\noindent The interaction terms may also be expressed through the $\mathbb F^0$-adapted conditional-law process $\mu_s^0:=\mathcal L((U,X_s)\mid\mathcal F_s^0) \in\mathcal P_2^\lambda(I\times\mathbb R^d)$, $s\in[t,T]$.

We now impose the following assumptions on the forward coefficients and the main structure which ensures well-posedness of the system (\ref{eq:state_equations}).

\begin{Assumption}\label{ass:forward_coeff}
The forward coefficients satisfy:
\begin{enumerate}
    \item[(i)] $A\in S^2_{\F^0}(L^2(I;\R^d)),
    \qquad
    D\in S^2_{\mathbb F^0}(L^2(I;\mathbb R^d)),
    \qquad
    \Sigma^0 \in S_{\mathbb F^0}^2 \big(L^2(I;\mathbb R^{d\times d_0})\big) \cap L^\infty_{\mathcal P^0\otimes\mathcal B(I)} \big( \Omega\times[0,T]\times I; \mathbb R^{d\times d_0} \big)$,
    \item[(ii)] $B\in L^\infty_{\F^0}(I\times[0,T];\R^{d\times d})$,  $E\in L^\infty_{\F^0}(I\times[0,T];\R^{d\times d})$,
    and $C\in L^\infty_{\F^0}(I\times[0,T];\R^{d\times m})$,  $F\in L^\infty_{\F^0}(I\times[0,T];\R^{d\times m})$,
    \item[(iii)] $G^B,\;G^E \in S^2_{\F^0}(L^2(I\times I;\R^{d\times d}))$,
    and  $T_{G^B},\,T_{G^E}
    \in
    L^\infty\big(\Omega;C([0,T];\mathcal L(L^2(I;\R^d)))\big)$.
\end{enumerate}
\end{Assumption}

\begin{Remark}[Multidimensional idiosyncratic noise]
\label{rem:multidimensional_private_noise}
For readability, the idiosyncratic Brownian motion is taken to be one-dimensional. The results extend to a fixed dimension $d_1\ge1$. In that case the diffusion coefficients are written $(D^a,E^a,F^a,G^{E,a})_{a=1}^{d_1}$ and every quadratic diffusion term is summed over $a$. For example, $O=R+\sum_{a=1}^{d_1}(F^a)^\top KF^a$. The definitions of $\Psi$, $M$, and the scalar driver we will see later are modified in the same way, using Hilbert--Schmidt contractions in the private-noise index. In the finite system, each agent carries $d_1$ private Brownian coordinates and the corresponding global sums acquire this additional index. No new compactness, Riccati-stability, or propagation-of-chaos argument is required; only the notation and constants change.
\end{Remark}

\begin{Definition}[Admissible condition]
    \begin{enumerate}
        \item[(i)]
        An $\mathbb R^d$-valued random variable $\xi$ is called an admissible initial condition at time $t$ if there exists a Borel measurable map $\Xi: [0,T]\times I\times[0,1] \times C([0,T];\mathbb R) \times C([0,T];\mathbb R^{d_0}) \longrightarrow\mathbb R^d$
        such that
        \[
        \xi
        =
        \Xi
        \big(
        t,U,\vartheta,
        W_{\cdot\wedge t},
        W^0_{\cdot\wedge t}
        \big),
        \qquad \mathbb P\text{-a.s.},
        \quad
        and
        \quad
        \mathbb E|\xi|^2<\infty.
        \]
        We denote by $\Ic_t$ the set of admissible initial conditions at time $t$.
        \item[(ii)] An $\mathbb R^m$-valued process $\alpha=(\alpha_s)_{s\in[t,T]}$ is admissible if there exists a jointly Borel measurable non-anticipative map $\mathfrak a:[t,T]\times[0,T]\times I\times\mathbb R^d \times C([0,T];\mathbb R)\times C([0,T];\mathbb R^{d_0}) \longrightarrow\mathbb R^m$, \[ \alpha_s=\mathfrak a(s,t,U,\xi,W_{\cdot\wedge s},W^0_{\cdot\wedge s}),  \quad ds\otimes d\P\text{-a.e.}, \quad \textit{and} \quad \E\int_t^T|\alpha_s|^2ds<\infty. \] We denote this class by $\mathcal A_t$
    \end{enumerate}
\end{Definition}
\paragraph{Well-posedness of the state equation.}
Under Assumption \ref{ass:forward_coeff}, the state equation \eqref{eq:state_equations} is well posed in the following sense.

\begin{Theorem}\label{thm:wellposed_limit_state}
Let $t\in[0,T]$, let $\xi\in \Ic_t$, and $\alpha\in \Ac_t$. Then the state equation admits a unique solution $X^\alpha\in S_{\mathbb G^{t,\xi}}^2([t,T];\mathbb R^d)$ and  that the system \eqref{eq:state_equations} holds $\P$-a.s. for every $s\in[t,T]$.

\noindent Moreover, the conditional mean-field term $\bar X_s^\alpha(u):=\E[X_s^\alpha\mid \Fc_s^0,U=u], (s,u)\in[t,T]\times I$,
admits a jointly measurable version on $[t,T]\times I\times \Omega$, and the corresponding conditional marked law $\mu_s^{0,\alpha}:=\Lc\big((U,X_s^\alpha)\mid \Fc_s^0\big)$
is well defined as a $\Pc_2^\lambda(I\times \R^d)$-valued $\F^0$-adapted process.

\noindent In addition, there exists a constant $C_T>0$, depending only on the data of the problem, such that
\[
\E\left[ \sup_{t\le s\le T}|X_s^\alpha|^2 \,\middle|\,\mathcal F_t^0 \right] \le C_T\E\left[ |\xi|^2 +\int_t^T\Big( |\alpha_s|^2 +\|A_s\|_{\mathsf H}^2 +\|D_s\|_{\mathsf H}^2 +\|\Sigma_s^0\|_{L^2(I;\mathbb R^{d\times d_0})}^2 \Big)ds \,\middle|\,\mathcal F_t^0 \right], \qquad \P\text{-a.s.}
\]
Moreover, combined with the measurable representation result for the frozen equation established in Proposition \ref{prop:measurable_rpz}, it yields a measurable representation of the actual solution $X^\alpha$ as a functional of the initial inputs.
\end{Theorem}
\begin{proof}
The proof of Theorem \ref{thm:wellposed_limit_state} relies on a standard fixed-point argument on the conditional flow $(\mu_s^{0,\alpha})_{s\in[t,T]}$, combined with the measurable representation result for the frozen equation established later on. For readability, it is postponed to the Appendix \ref{app:P1} (see also \cite{DeFeoMekkaoui2025,DeCrescenzoDeFeoPham2026}).
\end{proof}

\paragraph{Conditional cost function} Given $t \in [0,T], \xi \in \Ic_t,$ we define the conditional functional that we have to minimize over all $\alpha \in \Ac_t$ with the state process \eqref{eq:state_equations}
\begin{align}
J(t,\xi,\alpha)
:=
\E\Bigg[
&\int_t^T
\Big(
\langle X_s,Q_s(U)X_s\rangle
+
\big\langle \bar X_s(U),\;(T_{G_s^Q}\bar X_s)(U)\big\rangle
+
\langle \alpha_s+\iota_s(U),R_s(U)(\alpha_s+\iota_s(U))\rangle
\Big)\,ds
\notag\\
&\qquad
+
\langle X_T,H(U)X_T\rangle
+
\big\langle \bar X_T(U),\;(T_{G^H}\bar X_T)(U)\big\rangle
\;\Big|\;\Fc_t^0
\Bigg].
\label{eq:conditional_cost}
\end{align}
The associated value function is $V_t(\xi):=\operatorname*{ess\,inf}_{\alpha \in \Ac_t} J(t,\xi,\alpha)$. We now impose the standing assumptions on the quadratic cost coefficients.
\begin{Assumption}\label{ass:cost_coeff}
The cost coefficients satisfy:
\begin{enumerate}
    \item[(i)] $Q\in S^2_{\F^0}(L^2(I;\S_+^d))\cap L^\infty(I\times\Omega\times[0,T];\mathbb S^d)$, $H\in L^2(\Omega\times I;\mathbb S_+^d)\cap L^\infty(I\times\Omega;\mathbb S^d)$,
    \item[(ii)] $G^Q\in S^2_{\F^0}(L^2(I\times I;\R^{d\times d}))$, $G^H\in L^2(\Omega;L^2(I\times I;\R^{d\times d}))$, with $G_s^Q(u,v)^\top=G_s^Q(v,u)$,  $G^H(u,v)^\top=G^H(v,u)$ and\\
    $T_{G^Q}\in L^\infty\big(\Omega;C([0,T];\mathcal L(L^2(I;\R^d)))\big)$,
    $T_{G^H}\in L^\infty\big(\Omega;\mathcal L(L^2(I;\R^d))\big)$ 
    \item[(iii)] $R\in S^2_{\F^0}(L^2(I;\mathbb S_{++}^m))\cap L^\infty(I\times\Omega\times[0,T];\mathbb S^m)$,
    $\iota\in S^2_{\F^0}(L^2(I;\R^m))$,
    and there exists a constant $c_R>0$ such that
    \[
    R_s(u)\succeq c_R I_m,
    \qquad du\otimes d\P\text{-a.e., for every } s\in[0,T].
    \]
    \item[(iv)] For every $X\in S^2([t,T];\R^d)$, we have $\forall s\in[t,T]$, \\
    $\E\Big[
    \langle X_s,Q_s(U)X_s\rangle
    +
    \langle \bar X_s(U),(T_{G_s^Q}\bar X_s)(U)\rangle
    \,\big|\,\Fc_t^0
    \Big]\ge 0,
    \quad 
    \E\Big[
    \langle X_T,H(U)X_T\rangle
    +
    \langle \bar X_T(U),(T_{G^H}\bar X_T)(U)\rangle
    \,\big|\,\Fc_t^0
    \Big]\ge 0,
    \qquad \P\text{-a.s.}$
    Moroevoer, the terminal fields $H$ and $G^H$ are respectively $\mathcal F_T^0\otimes\mathcal B(I)$- and $\mathcal F_T^0\otimes\mathcal B(I^2)$-measurable
\end{enumerate}
\end{Assumption}

\section{Problem formulation of the finite-N state equation}\label{sec:problem_discrete}

We now introduce the finite population system which will later be compared to the continuum problem. Fix $N\ge1$ and let $W^{1,N},\dots,W^{N,N}$ be independent real-valued Brownian motions, independent of the common noise $W^0$. In addition to that, we introduce independent random variables $\vartheta^{1,N},\ldots,\vartheta^{N,N}$, uniformly distributed on $[0,1]$, independent of all Brownian motions. We denote by $\mathcal F_s^N := \sigma\Big( \xi^N, W_{\cdot\wedge s}^{1,N},\ldots,W_{\cdot\wedge s}^{N,N}, W_{\cdot\wedge s}^0 \Big)^{\mathbb P}$, the natural filtration of the finite system. All running coefficients introduced below are $\mathbb F^0$-progressively measurable, while terminal coefficients are $\mathcal F_T^0$ measurable. 

\begin{Definition}[Finite admissible initial conditions and controls]
Fix $N\ge1$ and $t\in[0,T]$.

\begin{enumerate}
\item[(i)] An $\mathbb H_N$-valued random variable $\xi^N$ is called an admissible initial condition at time $t$ if there exists a Borel measurable map $\Xi^N: [0,T]\times[0,1]^N \times C([0,T];\mathbb R)^N \times C([0,T];\mathbb R^{d_0}) \longrightarrow\mathbb H_N$
such that 
\[
\xi^N = \Xi^N \big( t,\vartheta^{1,N},\ldots,\vartheta^{N,N}, W_{\cdot\wedge t}^{1,N},\ldots,W_{\cdot\wedge t}^{N,N}, W_{\cdot\wedge t}^0 \big),\quad \P\text{-a.s.}\quad\textit{and}\quad \E|\xi^N|_N^2<\infty
\]
We denote the set of such initial conditions by $\Ic_t^N$.

\item[(ii)] An $\mathbb U_N$-valued process $\alpha^N=(\alpha_s^N)_{s\in[t,T]}$ is called an admissible control on $[t,T]$ if there exists a family of Borel measurable non-anticipative maps $\mathfrak a^N: [t,T]\times \mathbb H_N \times C([0,T];\R)^N \times C([0,T];\R^{d_0}) \to \mathbb U_N$ such that 
\[
\alpha_s^N = \mathfrak a^N \big(s,\xi^N, W^{1,N}_{\cdot\wedge s},\ldots,W^{N,N}_{\cdot\wedge s}, W^0_{\cdot\wedge s})\quad ds\otimes d\mathbb P\hbox{-a.e.}\quad \textit{and}\quad\E\int_t^T|\alpha_s^N|_N^2\,ds<\infty
\]
We denote the set of controls by $\Ac_t^N$.
\end{enumerate}
\end{Definition}

\begin{Remark}[Centralized and representative information structures]
\label{rem:centralized_vs_decentralized}
The discrete admissible class $\mathcal A^N$ is formu-lated as a centralized social planner problem: controls $\alpha^N=(\alpha^{1,N},\ldots,\alpha^{N,N})$ are adapted to the full filtration generated by the common noise, all idiosyncratic Brownian motions, and all initial conditions. Hence the optimal finite-dimensional
feedback may depend on the full state vector $\mathbf X_t^N= (X_t^{1,N},\ldots,X_t^{N,N})$.
This should not be confused with a decentralized finite-\(N\) information
structure in which agent \(i\) only observes its own state, its own
idiosyncratic noise, and the common noise. 
\noindent In contrast, the continuum feedback obtained in Section~\ref{sec:solution_continuum} has a representative-agent form. It
depends on the label $U$, the individual state $X_t$, and the
common-noise conditional aggregate field. In this sense, the limiting
control is naturally implementable under decentralized information. A fully decentralized finite-$N$ control problem would be a different problem under restricted information; in the LQG setting, decentralized $\varepsilon$-Nash strategies for large populations of nonuniform agents were constructed in \cite{HuangCainesMalhame2007,HuangMalhameCaines2006}, and would here require additional filtering or dynamic-team arguments.
\end{Remark}

\noindent Given $\xi^N$ and $\alpha^N\in\Ac_t^N$, we consider the finite population state equation

\begin{equation}\label{eq:finite_N_state}
\left\{
\begin{aligned}
dX_s^{i,N}
&=
\Big[
A_s^{i,N}
+
B_s^{i,N}X_s^{i,N}
+
\frac1N\sum_{j=1}^N G_{B,s}^N(u_i^N,u_j^N)X_s^{j,N}
+
C_s^{i,N}\alpha_s^{i,N}
\Big]\,ds
\\
&\quad+
\Big[
D_s^{i,N}
+
E_s^{i,N}X_s^{i,N}
+
\frac1N\sum_{j=1}^N G_{E,s}^N(u_i^N,u_j^N)X_s^{j,N}
+
F_s^{i,N}\alpha_s^{i,N}
\Big]\,dW_s^{i,N}
\\
&\quad+
\Sigma_{i,s}^{0,N}\,dW_s^0,
\qquad i=1,\dots,N,\quad s\in[t,T],
\\
X_t^{i,N}&=\xi^{i,N}.
\end{aligned}
\right.
\end{equation}

\noindent For the local coefficients, define $\mathsf M_{B_s^N},\mathsf M_{E_s^N}:\mathbb H_N\to\mathbb H_N$, and $\mathsf M_{C_s^N},\mathsf M_{F_s^N}:\mathbb U_N\to\mathbb H_N$, by $(\mathsf M_{\beta_s^N}\mathbf z)^i=\beta_s^{i,N}z_i$ for $\beta\in\{B,C,E,F\}$ and a vector $\mathbf z=(z_i)_{i=1}^N$ in the corresponding domain. The state equation is equivalently
\begin{equation}\label{eq:finite_N_vector_state}
\left\{
\begin{aligned}
dX_s^N
&=
\Big[
A_s^N
+
\mathsf M_{B_s^N} X_s^N
+
\mathsf G_{B,s}^N X_s^N
+
\mathsf M_{C_s^N}\alpha_s^N
\Big]\,ds
\\
&\quad+
\Big[
D_s^N
+
\mathsf M_{E_s^N} X_s^N
+
\mathsf G_{E,s}^N X_s^N
+
\mathsf M_{F_s^N}\alpha_s^N
\Big]\,dW_s^N
+
\Sigma_s^{0,N}\,dW_s^0,
\\
X_t^N&=\xi^N.
\end{aligned}
\right.
\end{equation}
Here the notation in the stochastic integral with respect to $W^N=(W^{1,N},\dots,W^{N,N})$ is understood componentwise.

\begin{Assumption}\label{ass:finite_forward_coeff}
For every $N\ge1$, the finite-dimensional coefficients satisfy the following conditions.

\begin{enumerate}
    \item[(i)] The additive coefficients are $\mathbb F^0$-progressively measurable and
    satisfy $A^N \in H_{\mathbb F^0}^2([t,T];\mathbb H_N)$, $D^N\in H_{\mathbb F^0}^2 \left( [t,T]; (\mathbb R^d)^N \right)$, $\Sigma^{0,N} \in H_{\mathbb F^0}^2 \left( [t,T]; (\mathbb R^{d\times d_0})^N \right)$ and there exists a deterministic constant $C_\Sigma>0$, independent of $N$, such that $\sup_{N\ge1} \operatorname*{ess\,sup}_{(t,\omega)} \max_{1\le i\le N} |\Sigma_{i,t}^{0,N}(\omega)| \le C_\Sigma$.
    \item[(ii)] The local coefficient processes $B^N,C^N,E^N,F^N$ are $\mathbb F^0$-progressively measurable, and there exists a deterministic constant $L>0$, independent of $N$, such that $\sup_{1\leq i\leq N} \Big( |B_s^{i,N}|+|C_s^{i,N}|+|E_s^{i,N}|+|F_s^{i,N}| \Big) \leq L$, $ds\otimes d\P\text{-a.e.}$
    \item[(iii)] The interaction kernels $G_B^N$ and $G_E^N$ are $\mathbb F^0$-progressively measurable. Moreover, there exists a deterministic constant $L_G>0$, independent of $N$, such that $\|\mathsf G_{B,s}^N\|_{\mathcal L(\mathbb H_N)} + \|\mathsf G_{E,s}^N\|_{\mathcal L(\mathbb H_N)} \leq L_G$, $ds\otimes d\P\text{-a.e.}$
\end{enumerate}
\end{Assumption}

\begin{Definition}[Solution of the finite population system]\label{def:solution_discret}
Let $N\ge1$, $t\in[0,T]$, $\xi^N\in \mathcal I_t^N$ and $\alpha^N\in\Ac_t^N$. We say that an $\mathbb H_N$-valued process $X^N=(X^{1,N},\dots,X^{N,N})$ is a solution to \eqref{eq:finite_N_state} on $[t,T]$ if:
\begin{enumerate}
    \item each $X^{i,N}$ is continuous and $\F^N$-adapted;
    \item $X^N\in S^2_{\mathbb F^N}([t,T];\mathbb H_N)$, i.e. $\E\left[\sup_{t\le s\le T}|X_s^N|_N^2\right] <\infty$;
    \item the equations in \eqref{eq:finite_N_state} hold $\P$-a.s. for every $i=1,\dots,N$ and every $s\in[t,T]$.
\end{enumerate}
\end{Definition}

\begin{Theorem}[Well-posedness of the finite population system]\label{thm:wellposed_finite_state}
Let $N\ge1$, $t\in[0,T]$, $\xi^N\in\mathcal I_t^N$, and $\alpha^N\in\Ac_t^N$. Under Assumption \ref{ass:finite_forward_coeff}, the finite-dimensional system \eqref{eq:finite_N_state} admits a unique strong solution $X^N=(X^{1,N},\dots,X^{N,N})$  $\in S^2_{\F^N}([t,T];\mathbb H_N)$ in the sense of Definition \ref{def:solution_discret} .
Moreover, there exists a constant $C_T>0$, independent of $N$, such that
\[
\begin{aligned}
\E\left[\sup_{t\le s\le T}|X_s^N|_N^2\right]
\le
C_T
\Bigg(
\E|\xi^N|_N^2
+
\E\int_t^T |A_s^N|_N^2\,ds
+
\E\int_t^T |D_s^N|_N^2\,ds
+
\E\int_t^T \|\Sigma_s^{0,N}\|_{N,\mathrm{HS}}^2\,ds
+
\E\int_t^T |\alpha_s^N|_N^2\,ds
\Bigg).
\end{aligned}
\]
Finally, the solution $X^N$ may be represented as a measurable functional of the input data $\big(\xi^N,W^{1,N},\dots,W^{N,N},W^0\big)$.
\end{Theorem}
\begin{proof}
The proof is standard and follows from a Picard fixed-point argument for finite-dimensional SDEs with globally Lipschitz coefficients. We provide the details in Appendix \ref{app:P2}.
\end{proof}

\subsection{Conditional cost function}

We now define the finite population cost functional. For $\bx^N=(x^1,\ldots,x^N)\in\mathbb H_N$ and $a^N=(a^1,\ldots,a^N)\in\mathbb U_N$, we introduce the quadratic forms
\[
\begin{aligned}
\mathfrak q_s^N(\bx^N)
&:=
\frac1N\sum_{i=1}^N \langle x^i,Q_s^{i,N}x^i\rangle
+
\frac1{N^2}\sum_{i,j=1}^N
\left\langle x^i,G_{Q,s}^N(u_i^N,u_j^N)x^j\right\rangle,\qquad
\mathfrak r_s^N(a^N)
:=
\frac1N\sum_{i=1}^N
\left\langle a^i+\iota_s^{i,N},R_s^{i,N}(a^i+\iota_s^{i,N})\right\rangle,
\\
\mathfrak h^N(\bx^N)
&:=
\frac1N\sum_{i=1}^N \langle x^i,H^{i,N}x^i\rangle
+
\frac1{N^2}\sum_{i,j=1}^N
\left\langle x^i,G_H^N(u_i^N,u_j^N)x^j\right\rangle.
\end{aligned}
\]

\noindent Given $t\in[0,T]$, $\xi^N\in\Ic_t^N$ and $\alpha^N\in\Ac_t^N$, we define the $\mathcal F_t^0$-measurable random cost functional by
\begin{align}
J^N(t,\xi^N,\alpha^N)
:=
\E\left[
\int_t^T
\Big(
\mathfrak q_s^N(X_s^N)
+
\mathfrak r_s^N(\alpha_s^N)
\Big)\,ds
+
\mathfrak h^N(X_T^N)
\;\middle|\;\Fc_t^0
\right].
\label{eq:finite_N_cost}
\end{align}

\noindent The finite population value function is
\[
V_t^N(\xi^N):=
\operatorname*{ess\,inf}_{\alpha^N\in\Ac_t^N}
J^N(t,\xi^N,\alpha^N).
\]

\begin{Assumption}\label{ass:finite_cost_coeff}
$\forall N\in\N^{\star}$, the finite-dimensional cost coefficients are $\mathbb F^0$-progressively measurable and the terminal cost coefficients are $\mathcal F_T^0$-measurable. They satisfy: 
\begin{enumerate}
    \item[(i)]  For $1\le i\le N$, $Q_s^{i,N}\in\mathbb S_+^d$ $ds\otimes d\mathbb P\text{-a.e.}$;  $H^{i,N}\in\mathbb S_+^d $, $\mathbb P\text{-a.s.}$, and\\ 
    $\sup_{N\ge1}\max_{1\le i\le N} \operatorname {ess\,sup}_{(s,\omega)}|Q_s^{i,N}|_{\mathrm F} + \sup_{N\ge1}\max_{1\le i\le N} \operatorname*{ess\,sup}_{\omega}|H^{i,N}|_{\mathrm F} <\infty$.
    \item[(ii)] $R^{i,N}$ is symmetric and uniformly coercive: there exists $c_R>0$, independent of $i$ and $N$, such that $R_s^{i,N}\succeq c_R I_m$, $ds\otimes d\P\text{-a.e.}$. Moreover, $R^{i,N}$ is uniformly bounded.
    \item[(iii)] The kernels $G_Q^N$ and $G_H^N$ are symmetric in the sense that $G_Q^N(u_i^N,u_j^N)^\top=G_Q^N(u_j^N,u_i^N)$, and $G_H^N(u_i^N,u_j^N)^\top=G_H^N(u_j^N,u_i^N)$, and the associated discrete operators are uniformly bounded on $(\mathbb H_N,|\cdot|_N)$.
    \item[(iv)] The total running and terminal state costs are non-negative, namely $\mathfrak q_s^N(\bx^N)\ge0$, and $\mathfrak h^N(x^N)\ge0$, for every $\bx^N\in\mathbb H_N$, $ds\otimes d\P$-a.e.
    \item[(v)] The affine control-cost coefficient $\iota^N$ is $\mathbb F^0$-progressively measurable and satisfies $\iota^N \in H_{\mathbb F^0}^2([0,T];\mathbb U_N)$.
\end{enumerate}
\end{Assumption}

\begin{Remark}\label{rem:finite_cost_positive}
Under Assumption~\ref{ass:finite_cost_coeff}, the finite population cost is bounded from below. More precisely, since the state cost is non-negative and $R^{i,N}\succeq c_R I_m$, the control cost is uniformly convex in $\alpha^N$. This uniform convexity will be used in the completion of squares and in the uniqueness of the optimal feedback control.
\end{Remark}

\section{Measurable representation and preliminary structural results}
This section records the measurable representations needed to compare different realizations of the initial data. They are used in the fixed-point construction of the state equation and allow an admissible control rule to be transported between inputs with the same conditional law. The resulting law-invariance statements clarify the probabilistic meaning of the value functions; the quantitative estimates are proved separately in the subsequent sections.

Fix $t \in [0,T]$, an admissible control $\alpha \in \Ac_t$ and let $m=(m_s)_{s\in[t,T]}$ be an $\F^0$-adapted process, progressively measurable with values in $L^2(I,\R^d)$. We consider the linearized state equation :
\begin{equation}\label{eq:frozen_state_eq}
\left\{
\begin{aligned}
dX_s^{m,\alpha}
&=
\Big[
A_s(U)+B_s(U)X_s^{m,\alpha}+\int_I G_s^B(U,v)m_s(v)\,dv + C_s(U)\alpha_s
\Big]\,ds
\\
&\quad
+
\Big[
D_s(U)+E_s(U)X_s^{m,\alpha}+\int_I G_s^E(U,v)m_s(v)\,dv + F_s(U)\alpha_s
\Big]\,dW_s
+
\Sigma_s^0(U)\,dW_s^0,
\qquad s\in[t,T],
\\
X_t^{m,\alpha}&=\xi.
\end{aligned}
\right.
\end{equation}

\begin{Proposition}\label{prop:measurable_rpz}
Assume the coefficient satisfy Assumption\ref{ass:forward_coeff}. Let $t\in[0,T]$, $\xi\in\Ic_t$, $\alpha\in\Ac_t$, and $m=(m_s)_{s\in[t,T]}$ be an $\F^0$-adapted process in $L^2(I;\R^d)$ progressively measurable such that $\E\left[\int_t^T\|m_s\|^2_{L^2(I;\R^d)}\,ds\right]<\infty$. Then the equation \eqref{eq:frozen_state_eq} admits a unique strong solution $X^{m,\alpha}\in S^2([t,T];\R^d).$
Moreover, there exist a Borel measurable map $x^{m,\alpha}: I\times \R^d\times C([0,T];\R)\times C([0,T];\R^{d_0})\longrightarrow C([t,T];\R^d)$ such that 
\[
X^{m,\alpha}
=
x^{m,\alpha}(U,\xi,W,W^0),
\qquad \P\text{-a.s.}
\]
Equivalently, for every $s\in[t,T]$, there exists a Borel measurable map $x_s^{m,\alpha}: I\times \R^d\times C([0,T];\R)\times C([0,T];\R^{d_0}) \longrightarrow \R^d$
such that
\[
X_s^{m,\alpha}
=
x_s^{m,\alpha}(U,\xi,W_{.\wedge s},W^0_{.\wedge s}),
\qquad \P\text{-a.s.}
\]
\end{Proposition}
\begin{proof}
The result is a simplified version of the measurable representation property established in \cite[Theorem 3.5]{DeFeoMekkaoui2025} for the full heterogeneous mean-field system with common noise. In the present frozen setting, the proof follows from the same measurable-representation strategy, combining measurable versions of the coefficients on the canonical path space, measurable factorization of admissible controls, and a standard Euler approximation argument for the frozen SDE. A proof is given in Appendix~\ref{app:P0}.
\end{proof}

\subsection{Law invariance and marked conditional laws}

We now record law-invariance properties of the finite and continuum value functions. Since admissible control rules may depend on the past idiosyncratic noise, we retain this history as a mark when transporting the same rule between initial inputs. The resulting marked conditional law is a sufficient input for this argument. In the LQ setting, the explicit value formula below depends on this law only through its $(u,x)$-marginal, together with the common-noise history.

\paragraph{Continuum marked conditional law.}
For $\xi\in\Ic_t$, we define $\rho_t^\xi
:=
\Lc\big(U,\xi,W_{\cdot\wedge t}\mid \Fc_t^0\big)
\in
\Pc_2\big(I\times\R^d\times C([0,T];\R)\big)$.

\paragraph{Finite marked conditional law.}
For $\xi^N\in\Ic_t^N$, we define $\rho_t^{N,\xi}
:=
\Lc\Big(
\xi^N,
W_{\cdot\wedge t}^{1,N},\dots,W_{\cdot\wedge t}^{N,N}
\mid \Fc_t^0
\Big)$.
Equivalently, one may encode the finite initial configuration by the random marked empirical measure
\[
\mu_t^{N,\xi}
:=
\frac1N\sum_{i=1}^N
\delta_{(u_i^N,\xi^{i,N},W_{\cdot\wedge t}^{i,N})}.
\]
\begin{Remark}
Since the labels $u_1^N,\dots,u_N^N$ are distinct, the marked empirical 
measure $\mu_t^{N,\xi}$ uniquely determines the ordered vector 
$(\xi^{i,N}, W^{i,N}_{\cdot\wedge t})_{i=1}^N$, and conversely. Hence 
the law-invariance statement of Proposition\ref{prop:law_invariance}(ii) 
can equivalently be expressed in terms of $\mu_t^{N,\xi}$, which is 
the formulation used in the convergence analysis of Sections\ref{sec:convergence}.
\end{Remark}

\begin{Proposition}[Law invariance]\label{prop:law_invariance}
Assume that the state equations are well posed and that the corresponding solutions admit measurable representations with respect to their input data.

\begin{enumerate}
    \item[(i)] \textbf{Continuum case.}
    Let $\xi,\tilde\xi\in\Ic_t$. If $\Lc\big(U,\xi,W_{\cdot\wedge t}\mid \Fc_t^0\big) = \Lc\big(U,\tilde\xi,W_{\cdot\wedge t}\mid \Fc_t^0\big)$, $\P\text{-a.s.}$,
    then
    \[
    V_t(\xi)=V_t(\tilde\xi),
    \qquad \P\text{-a.s.}
    \]
    Under Assumptions~\ref{ass:forward_coeff} and~\ref{ass:cost_coeff}, there exists a Borel measurable functional $\mathfrak V_t$ such that $V_t(\xi) =\mathfrak V_t\big(W^0_{\cdot\wedge t},\rho_t^\xi\big)$, $\mathbb P\text{-a.s.}$. Its construction from the quadratic value formula is given in Remark~\ref{rem:construction_value_function}. The common-noise history cannot in general be omitted, because the coefficients and the backward fields may depend on it.
    \item[(ii)] \textbf{Finite population case.}
    Let $\xi^N,\tilde\xi^N\in\Ic_t^N$. If
    \[
    \Lc\Big(
    \xi^N,
    W_{\cdot\wedge t}^{1,N},\dots,W_{\cdot\wedge t}^{N,N}
    \mid \Fc_t^0
    \Big)
    =
    \Lc\Big(
    \tilde\xi^N,
    W_{\cdot\wedge t}^{1,N},\dots,W_{\cdot\wedge t}^{N,N}
    \mid \Fc_t^0
    \Big),
    \qquad \P\text{-a.s.},
    \]
    then
    \[
    V_t^N(\xi^N)=V_t^N(\tilde\xi^N),
    \qquad \P\text{-a.s.}
    \]
\end{enumerate}
\end{Proposition}

\begin{proof}
We prove the continuum statement first. Let $\alpha\in\Ac_t$. By admissibility, there exists a jointly Borel non-anticipative map $a$ such that $\alpha_s=a(s,t,U,\xi,W_{\cdot\wedge s},W^0_{\cdot\wedge s})$ $ds\otimes d\mathbb P$-a.e. Define $\widetilde\alpha_s :=a(s,t,U,\widetilde\xi,W_{\cdot\wedge s},W^0_{\cdot\wedge s})$. Then $\tilde\alpha\in\Ac_t$. By assumption,
\[
\Lc(U,\xi,W_{\cdot\wedge t}\mid\Fc_t^0)
=
\Lc(U,\tilde\xi,W_{\cdot\wedge t}\mid\Fc_t^0).
\]

\noindent For $r\ge1$, let $\mathcal C_{t,0}^{r}:=\{z\in C([t,T];\mathbb R^r):z_t=0\}$,
and let $\mathcal W_t^{r}$ be Wiener measure on this space. For $w\in C([0,T];\mathbb R^r)$ and $z\in\mathcal C_{t,0}^{r}$, define $(w\oplus_tz)_s=w_s$ for $s\le t$ and $(w\oplus_tz)_s=w_t+z_s$ for $s>t$. Independence of the future Brownian increments from $\mathcal F_t$ and from each other gives, for every bounded Borel $\varphi$,
\[
\begin{aligned}
\mathbb E[\varphi(U,\xi,W,W^0)\mid\mathcal F_t^0]
={}&\int
\varphi\big(u,x,w\oplus_tz,
W^0_{\cdot\wedge t}\oplus_tz^0\big)
\,\rho_t^\xi(du,dx,dw)
\,\mathcal W_t^1(dz)\,\mathcal W_t^{d_0}(dz^0),
\quad\mathbb P\text{-a.s.}
\end{aligned}
\]
The right-hand side is a measurable function of $(W^0_{\cdot\wedge t},\rho_t^\xi)$. The same formula holds with $\widetilde\xi$ in place of $\xi$, using the same common-noise history. Hence this equality extends to the full input data :
\[
\Lc(U,\xi,W,W^0\mid\Fc_t^0)
=
\Lc(U,\tilde\xi,W,W^0\mid\Fc_t^0).
\]
Applying the same control rule gives equality of the conditional laws of the full input tuples, including the entire common-noise path. Fix now a progressively measurable common-noise field m in the frozen-equation space. The frozen equations for the two initial inputs have the same coefficients and the same rule a. Their measurable solution representation can therefore be chosen as the same map $x^{a,m}$. Hence
\[
\mathcal L(U,X^{\xi,a,m},\alpha,W^0\mid\mathcal F_t^0) =\mathcal L(U,X^{\widetilde\xi,a,m},\widetilde\alpha,W^0 \mid\mathcal F_t^0).
\]
Keeping $W^0$ in this identity and disintegrating with respect to $(U,W^0_{\cdot\wedge s})$ shows that the two frozen conditional-mean maps coincide, as elements of the progressively measurable $L^2$ field space. Their unique fixed points therefore coincide. The actual controlled states consequently have the same joint conditional law with the label, control and common-noise path. The cost functionals are thus equal rule by rule. Since the cost functional is a measurable function of the controlled state, the control and the common noise, it follows that
\[
J(t,\xi,\alpha)=J(t,\tilde\xi,\tilde\alpha),
\qquad \P\text{-a.s.}
\]
Taking the essential infimum over $\alpha\in\Ac_t$ yields $V_t(\tilde\xi)\le V_t(\xi)$. Exchanging the roles of $\xi$ and $\tilde\xi$ gives the reverse inequality, hence
\[
V_t(\xi)=V_t(\tilde\xi),
\qquad \P\text{-a.s.}
\]
This proves the continuum statement. The proof of the finite population statement is identical, replacing the input data $(U,\xi,W,W^0)$ by $(\xi^N,W^{1,N},\dots,W^{N,N},W^0)$.
Indeed, given any admissible control $\alpha^N$, choose a Borel non-anticipative rule $a^N$ such that 
$\alpha_s^N=a^N(s,\xi^N,W^{1,N}_{\cdot\wedge s},\dots,W^{N,N}_{\cdot\wedge s},W^0_{\cdot\wedge s})$.
And define similarly $\widetilde\alpha_s^N :=a^N(s,\widetilde\xi^N, W^{1,N}_{\cdot\wedge s},\dots,W^{N,N}_{\cdot\wedge s}, W^0_{\cdot\wedge s})$. The equality of conditional laws of the initial input data implies the equality of conditional laws of the full input data, and the measurable representation of the finite state equation gives
\[
J^N(t,\xi^N,\alpha^N)=J^N(t,\tilde\xi^N,\tilde\alpha^N).
\]
Passing to the essential infimum and then exchanging the roles of $\xi^N$ and $\tilde\xi^N$ proves
\[
V_t^N(\xi^N)=V_t^N(\tilde\xi^N).
\]
The proof is complete.
\end{proof}

\section{Solution of the continuum LQ control problem}\label{sec:solution_continuum}

In this section, we will work on the problem introduced in section \ref{sec:problem_discret}. We will introduce the associated system of Backward Stochastic Riccati Equation (BSRE) valued in the associated Hilbert space. Although this work is close to the one done in \cite{DeFeoMekkaoui2025} and \cite{DeCrescenzoDeFeoPham2026}, we will show the derivated results of our continuum system.

\noindent Throughout this section, we assume that Assumptions \ref{ass:forward_coeff} and \ref{ass:cost_coeff} hold. We also assume that the backward system introduced below admits a sufficiently integrable solution. Its solvability is proved in Theorem \ref{thm:solvability_continuum_backward_system}.

\subsection{Quadratic ansatz and backward Riccati system}

\paragraph{Quadratic Ansatz}

Let $X^\alpha$ be the state process associated with an admissible control $\alpha\in\Ac_t$, and define $\bar X_s^\alpha(u):=\E[X_s^\alpha\mid \Fc_s^0,U=u]$.
We look for the following real-valued adapted quadratic process :
\begin{align}
v_s(X^\alpha)
:=
&\;
\langle X_s^\alpha,K_s(U)X_s^\alpha\rangle
+
\left\langle
\bar X_s^\alpha(U),
(T_{\bar K_s}\bar X_s^\alpha)(U)
\right\rangle
+
2\langle Y_s(U),X_s^\alpha\rangle
+
\lambda_s,
\label{eq:continuum_quadratic_ansatz}
\end{align}
where $K_s(u)\in\mathbb S^d$, $\bar K_s(u,v)\in\R^{d\times d}$, $Y_s(u)\in\R^d$, $\lambda_s\in\R$. At the initial time, its conditional expectation gives the candidate value functional: $\mathcal V_t(\xi):= \E[v_t(X^\alpha)\mid\mathcal F_t^0]$. Since $X_t^\alpha=\xi$, this quantity does not depend on the choice of $\alpha$.\\
The terminal conditions are $K_T(u)=H(u)$, $\bar K_T(u,v)=G^H(u,v)$, $Y_T(u)=0$, $\lambda_T=0$.

\smallskip
\noindent We introduce the following coefficients, which arise in the completion of squares:
\begin{alignat}{2}
& O_s(u)
:=
R_s(u)+F_s(u)^\top K_s(u)F_s(u),
&\qquad
& U_s(u)
:=
C_s(u)^\top K_s(u)+F_s(u)^\top K_s(u)E_s(u),
\\
& V_s(u,v)
:=
C_s(u)^\top \bar K_s(u,v)
+F_s(u)^\top K_s(u)G_s^E(u,v),
&\qquad
& \Gamma_s(u)
:=
C_s(u)^\top Y_s(u)
+F_s(u)^\top K_s(u)D_s(u)
+R_s(u)\iota_s(u)
\end{alignat}

\noindent For multidimensional idiosyncratic noise, the above products are understood with the usual Hilbert--Schmidt contraction over the noise component.

\begin{Remark}\label{rem:O_inverse_continuum}
Whenever $K_s(u)\in\mathbb S_+^d$, by Assumption \ref{ass:cost_coeff}, there exists $\eta>0$ such that $O_s(u)\succeq \eta I_m$, and $O_s(u)$ is invertible $ds\otimes du\otimes d\P$-a.e. Moreover, $|O_s^{-1}(u)|_F\leq\sqrt{m}\eta^{-1}$, so it is uniformly invertible once K is known non-negative.
\end{Remark}

\paragraph{Standard BSRE}

\begin{Definition}\label{def:solution_K_continuum}
A solution of \eqref{eq:K_BSRE_continuum} is a pair of processes $(K,Z^K)=\big(K(u),Z^K(u)\big)_{u\in I}$ such that, for $du$-a.e. $u\in I$, the pair $(K(u),Z^K(u))$ solves \eqref{eq:K_BSRE_continuum} in the usual matrix-valued BSDE sense, and $(K,Z^K)\in S^2_{\mathbb F^0}(L^2(I;\mathbb S^d)) \times H^2_{\mathbb F^0} \big(L^2(I;\mathcal L_2(\mathbb R^{d_0};\mathbb S^d))\big)$.
Uniqueness of $K$ is understood up to indistinguishability in $L^2(I;\mathbb S^d)$, whereas uniqueness of $Z^K$ is understood up to equality $ds\otimes d\mathbb P\otimes du$-a.e.
\end{Definition}

\noindent The local Riccati BSDE is, for a.e. $u\in I$, $0\le s\le T$, 
\begin{equation}\label{eq:K_BSRE_continuum}
\left\{
\begin{aligned}
dK_s(u)
&=
-\Big[
Q_s(u)
+B_s(u)^\top K_s(u)+K_s(u)B_s(u)
+E_s(u)^\top K_s(u)E_s(u)
-U_s(u)^\top O_s(u)^{-1}U_s(u)
\Big]\,ds
+
Z_s^K(u)\,dW_s^0,
\\
K_T(u)&=H(u).
\end{aligned}
\right.
\end{equation}

\begin{Remark}
\label{rem:K_bounded_continuum}
The standard stochastic Riccati theory yields more than square-integrability.
Under the boundedness, non-negativity and coercivity assumptions on the
coefficients, the local Riccati component $K$ admits a bounded version
satisfying $\|K\|_{S^\infty(L^\infty(I))} := \esssup_{\omega\in\Omega} \sup_{t\in[0,T]} \esssup_{u\in I} \|K_t(u,\omega)\| \le C_T $. Indeed, the equation for $K_t(u)$ is the local stochastic Riccati equation associated with the homogeneous one-dimensional-in-label LQ problem. For a fixed label $u$, a fixed time $t$, and an initial state $x\in\mathbb R^d$, the homogeneous value is $V_t^{0,u}(x)=x^\top K_t(u)x$. By admissibility of the zero control, $V_t^{0,u}(x) \le J_t^{0,u}(x,0)$. The homogeneous state under zero control satisfies the standard estimate $\E\left[ \sup_{s\in[t,T]}|X_s^{t,u,x,0}|^2 \,\middle|\,\mathcal F_t^0 \right] \le C|x|^2$, where $C$ depends only on the uniform bounds of the state coefficients and on $T$, but not on $t,u$ or $x$. Since $Q$ and $H$ are uniformly bounded and non-negative, $0\le x^\top K_t(u)x = V_t^{0,u}(x) \le C|x|^2$. Since the constant $C$ is independent of $t$ and $u$, and since the BSRE solution $K$ admits a version with continuous time paths (by standard stochastic Riccati theory), the bound $\|K_t(u)\| \le C$ extends to all 
$t \in [0,T]$ simultaneously, $\P$-a.s., for a.e. $u \in I$. This yields 
the stated $S^\infty(L^\infty(I))$-bound.
\end{Remark}

\paragraph{Abstract BSRE}
Given a solution $(K,Z^K)$ to \eqref{eq:K_BSRE_continuum}, we have the  Riccati BSDE on $L^2(I\times I;\mathbb R^{d\times d})$, for $du\otimes dv$-a.e. $(u,v)\in I\times I$, 
\begin{equation}\label{eq:Kbar_BSRE_continuum}
d\bar K_s(u,v)
=
-\mathcal F(s,\bar K_s)(u,v)\,ds
+
Z_s^{\bar K}(u,v)\,dW_s^0,
\qquad
\bar K_T(u,v)=G^H(u,v),
\qquad 
0\le s \le T.
\end{equation}

\noindent where we have the random map $\mathcal F: \Omega\times[0,T]\times L^2(I\times I;\mathbb R^{d\times d}) \longrightarrow L^2(I\times I;\mathbb R^{d\times d})$ defined by,
\[
\mathcal F(s,\bar k)(u,v)
:=
\Psi(s,u,v;K_s,\bar k)
-
U_s(u)^\top O_s(u)^{-1}V(s,u,v;K_s,\bar k)
-
V(s,v,u;K_s,\bar k)^\top O_s(v)^{-1}U_s(v)
\]
\[
\qquad
-
\int_I
V(s,w,u;K_s,\bar k)^\top
O_s(w)^{-1}
V(s,w,v;K_s,\bar k)\,dw,
\]
and where $\forall s \in [0,T]$, $\bar k \in L^2(I \times I; R^{d\times d})$ we have :
\begin{align}
\Psi(s,u,v;K_s,\bar k)
:=
&\;
K_s(u)G_s^B(u,v)
+
G_s^B(v,u)^\top K_s(v)
+
E_s(u)^\top K_s(u)G_s^E(u,v)
+
G_s^E(v,u)^\top K_s(v)E_s(v)
\notag\\
&+
\int_I G_s^E(w,u)^\top K_s(w)G_s^E(w,v)\,dw
+
B_s(u)^\top \bar k(u,v)
+
\bar k(u,v)B_s(v)
\notag\\
&+
\int_I G_s^B(w,u)^\top \bar k(w,v)\,dw
+
\int_I \bar k(u,w)G_s^B(w,v)\,dw
+
G_s^Q(u,v).
\label{eq:Psi_continuum}
\end{align}

\begin{Definition}\label{def:solution_Kbar_continuum}
A solution of \eqref{eq:Kbar_BSRE_continuum} is a pair $(\bar K,Z^{\bar K})$
such that $(\bar K,Z^{\bar K})\in S^2_{\mathbb F^0}(L^2(I^2;\mathbb R^{d\times d})) \times H^2_{\mathbb F^0} \big(L^2(I^2;\mathcal L_2(\mathbb R^{d_0}; \mathbb R^{d\times d}))\big)$. Uniqueness of $\bar K$ is understood up to indistinguishability as an $\overline{\mathsf H}$-valued process; uniqueness of $Z^{\bar K}$ is understood up to equality $ds\otimes d\mathbb P$-a.e. in $\mathcal L_2(\mathbb R^{d_0};\overline{\mathsf H})$. The pair satisfies \eqref{eq:Kbar_BSRE_continuum} as a BSDE in $\overline{\mathsf H}$. 
\end{Definition}

\begin{Remark}\label{rem:F_properties_continuum}
The random map $\mathcal F$ satisfies the following properties, with $\overline{\mathsf H}:=L^2(I\times I;\mathbb R^{d\times d})$.
\begin{enumerate}
    \item[(i)] For any $\bar k\in \overline{\mathsf H}$, there exists a constant $C_T>0$ such that, for $ds\otimes d\P$-a.e. $(s,\omega)$ :
    \[
    \|T_{\mathcal F(s,\bar k)}\|_{\mathcal L(L^2(I;\mathbb R^d))}
    \le
    C_T
    \left(
    1+\|T_{\bar k}\|_{\mathcal L(L^2(I;\mathbb R^d))}
    +\|T_{\bar k}\|_{\mathcal L(L^2(I;\mathbb R^d))}^2
    \right),
    \]
    \item[(ii)] For any $\bar k,\bar \ell\in \overline{\mathsf H}$, there exists a constant $C_T>0$ such that, for $ds\otimes d\P$-a.e. $(s,\omega)$ :
    \[
    \|\mathcal F(s,\bar k)-\mathcal F(s,\bar \ell)\|_{\overline{\mathsf H}}
    \le
    C_T
    \left(
    1+\|T_{\bar k}\|_{\mathcal L(L^2(I;\R^d))}
    +
    \|T_{\bar \ell}\|_{\mathcal L(L^2(I;\R^d))}
    \right)
    \|\bar k-\bar \ell\|_{\overline{\mathsf H}},
    \]
    \item[(iii)] For every $\bar k\in\overline{\mathsf H}$, the process $s\mapsto\mathcal F(s,\bar k)$ is progressively measurable, and $\bar k\mapsto\mathcal F(s,\bar k)$ is continuous for $ds\otimes d\mathbb P$-a.e. $(s,\omega)$. Moreover, $\mathcal F$ preserves symmetric kernels: writing $\bar k^\dagger(u,v):=\bar k(v,u)^\top$, we have $\bar k=\bar k^\dagger \Longrightarrow \mathcal F(s,\bar k)=\mathcal F(s,\bar k)^\dagger$.
    \item[(iv)] We will show in the solvability proof below, that the solution actually satisfies the stronger estimate\\ 
    $T_{\bar K}\in L^\infty\big(\Omega;C([0,T];\mathcal L(L^2(I;\mathbb R^d)))\big)$, namely there exists a constant $C_T>0$ such that\\ 
    $\sup_{0\le s\le T} \|T_{\bar K_s}\|_{\mathcal L(L^2(I;\mathbb R^d))} \le C_T$, $\P\text{-a.s.}$
\end{enumerate}
\end{Remark}
\begin{proof}
The estimates follow from the boundedness of the coefficients, the uniform bound on $K$, the coercivity of $R$, and the submultiplicativity of the operator norm. The proof is based on the same operator estimates as in the abstract Riccati equation of \cite[Remark 4.4]{DeFeoMekkaoui2025} and \cite[Remark 4.6]{DeFeoMekkaoui2025}, adapted here to our additive common-noise structure.
\end{proof}

\paragraph{Linear Equations} Given $(K,Z^K)$ and $(\bar K, Z^{\bar K})$ respective solution of \eqref{eq:K_BSRE_continuum} and \eqref{eq:Kbar_BSRE_continuum}, we have the linear BSDE in $L^2(I;\R^d)$
\begin{equation}\label{eq:Y_BSDE_continuum}
\left\{
\begin{aligned}
dY_s(u)
=
-\Big[
&M_s(u)
+
B_s(u)^\top Y_s(u)
+
\int_I G_s^B(v,u)^\top Y_s(v)\,dv
\\
&-
U_s(u)^\top O_s(u)^{-1}\Gamma_s(u)
-
\int_I V_s(v,u)^\top O_s(v)^{-1}\Gamma_s(v)\,dv
\Big]ds
+
Z_s^Y(u)\,dW_s^0,
\\
Y_T(u)=0
\end{aligned}
\right.
\end{equation}
where we defined $\forall s \in [0,T]$, $du$ a.e. 
\begin{align}
M_s(u)
:=
&\;
K_s(u)A_s(u)
+
Z_s^K(u)\Sigma_s^0(u)
+
E_s(u)^\top K_s(u)D_s(u)
\notag\\
&+
\int_I G_s^E(v,u)^\top K_s(v)D_s(v)\,dv
+
\int_I \bar K_s(u,v)A_s(v)\,dv
+
\int_I Z_s^{\bar K}(u,v)\Sigma_s^0(v)\,dv.
\label{eq:M_continuum}
\end{align}

\begin{Definition}\label{def:solution_Y_continuum}
A solution of \eqref{eq:Y_BSDE_continuum} is a pair of process $(Y,Z^Y)$ such that it solves \eqref{eq:Y_BSDE_continuum} and $(Y,Z^Y)\in S^2_{\mathbb F^0}(L^2(I;\mathbb R^d)) \times H^2_{\mathbb F^0} \big(L^2(I;\mathcal L_2(\mathbb R^{d_0};\mathbb R^d))\big)$.Uniqueness of $Y$ is understood up to indistinguishability as an $\mathsf H$-valued process; uniqueness of $Z^Y$ is understood up to equality $ds\otimes d\mathbb P$-a.e. in $\mathcal L_2(\mathbb R^{d_0};\mathsf H)$.
\end{Definition}

\noindent To finish, we introduce the integrated scalar BSDE
\begin{equation}\label{eq:lambda_BSDE_continuum}
\left\{
\begin{aligned}
d\lambda_s
={}&
-\Bigg\{
\int_I
\Bigg[
    \langle D_s(u),K_s(u)D_s(u)\rangle_{\mathrm F}
    +
    \left\langle
        \Sigma_s^0(u),
        K_s(u)\Sigma_s^0(u)
    \right\rangle_{\mathrm F}
    +2A_s(u)^\top Y_s(u)
    +2\left\langle
        Z_s^Y(u),
        \Sigma_s^0(u)
    \right\rangle_{\mathrm F}
\\
&\qquad
    +\int_I
    \left\langle
        \Sigma_s^0(u),
        \bar K_s(u,v)\Sigma_s^0(v)
    \right\rangle_{\mathrm F}\,dv +\iota_s(u)^\top R_s(u)\iota_s(u)
    -\Gamma_s(u)^\top O_s(u)^{-1}\Gamma_s(u)
\Bigg]\,du
\Bigg\}\,ds
+Z_s^\lambda\,dW_s^0,
\\
\lambda_T={}&0.
\end{aligned}
\right.
\end{equation}
We denote by $\ell_s$ the scalar quantity inside braces in \eqref{eq:lambda_BSDE_continuum}. Thus, $d\lambda_s = -\ell_s\,ds + Z_s^\lambda\,dW_s^0$, $\lambda_T=0$.

\begin{Definition}
An integrable solution of \eqref{eq:lambda_BSDE_continuum} is a pair $(\lambda,Z^\lambda)$ such that $\lambda$ is a continuous $\mathbb F^0$-adapted process of class $(D)$, $Z^\lambda$ is predictable and locally square-integrable, and \eqref{eq:lambda_BSDE_continuum} holds.
\end{Definition}

\begin{Remark}\label{rem:triangular_structure_backward_system}
The backward system has a triangular structure. The equation for $K$ is independent of the interaction kernels $G^B,G^E,G^Q,G^H$ and of the additive coefficients $A,D,\Sigma^0,\iota$. Once $K$ is known, the equation for $\bar K$ is closed. Once $(K,\bar K)$ are known, the equation for $Y$ is linear. 
Once $(K,Z^K,\bar K,Z^{\bar K},Y,Z^Y)$ are known, the process $\ell$ is given and satisfies $\E\int_0^T|\ell_s|\,ds<\infty$. Indeed, this follows from the boundedness of $K$, $O^{-1}$ and $\Sigma^0$, the operator bound on $T_{\bar K}$, the coefficients $A,D,Y\in S^2_{\mathbb F^0}(\mathsf H)$ and $\iota\in S^2_{\mathbb F^0}(\mathsf U)$, and $Z^Y\in H^2_{\mathbb F^0} (\mathcal L_2(\mathbb R^{d_0};\mathsf H))$, using Cauchy--Schwarz. Consequently, the scalar component is uniquely determined by $\lambda_t = \E\left[ \int_t^T\ell_s\,ds \;\middle|\; \Fc_t^0 \right]$. Equivalently, setting $M_t := \E\left[ \int_0^T\ell_s\,ds \;\middle|\; \Fc_t^0 \right]$, one has $\lambda_t=M_t-\int_0^t\ell_s\,ds$.The martingale $M$ is uniformly integrable. By localization and the Brownian martingale representation theorem, there exists a predictable locally square-integrable process $Z^\lambda$ such that $M_t=M_0+\int_0^tZ_s^\lambda\,dW_s^0$. Hence $(\lambda,Z^\lambda)$ is the unique integrable solution of \eqref{eq:lambda_BSDE_continuum}, with $\lambda$ of class $(D)$.
\end{Remark}

\noindent We now prove the fundamental relation to decompose the conditional cost $J(t,\xi,\alpha)$ saw in \eqref{eq:conditional_cost}. For an initial condition $\xi\in\mathcal I_t$, we denote by $\bar\xi(u) := \E\left[ \xi \;\middle|\; \Fc_t^0,U=u \right]$ a jointly measurable version of its conditional mean field.

\begin{Proposition}[Fundamental relation for the continuum problem]
\label{prop:fundamental_relation_continuum}
Assume that the backward system
\eqref{eq:K_BSRE_continuum}--\eqref{eq:lambda_BSDE_continuum}
admits a solution $(K,Z^K,\bar K,Z^{\bar K},Y,Z^Y,\lambda,Z^\lambda)$ with the integrability and boundedness properties stated in Theorem~\ref{thm:solvability_continuum_backward_system}. Let $t\in[0,T]$, $\xi\in\Ic_t$, and let $\alpha\in\Ac_t$ be an admissible control, the following fundamental relation holds :
\begin{align}
J(t,\xi,\alpha)
=
\mathcal V_t(\xi)
+
\E\Big[
\int_t^T
\left\langle
\alpha_s-\alpha_s^{\star,\alpha},
O_s(U)(\alpha_s-\alpha_s^{\star,\alpha})
\right\rangle ds
\;\Big|\;\Fc_t^0
\Big],
\qquad \P\text{-a.s.}
\label{eq:fundamental_relation_continuum}
\end{align}

\noindent with $\mathcal V$ being $\mathcal V_t(\xi):=\E\Big[
\langle \xi,K_t(U)\xi\rangle
+
\left\langle
\bar \xi(U),
(T_{\bar K_t}\bar \xi)(U)
\right\rangle
+
2\langle Y_t(U),\xi\rangle
\;\Big|\;\Fc_t^0
\Big]+\lambda_t$, and where for any admissible control $\alpha$, we define the feedback expression evaluated along the controlled state $X^\alpha$ by
\begin{align}
\alpha_s^{\star,\alpha}
:=
-O_s(U)^{-1}
\left(
U_s(U)X_s^\alpha
+
\int_I V_s(U,v)\bar X_s^\alpha(v)\,dv
+
\Gamma_s(U)
\right).
\label{eq:feedback_along_alpha_continuum}
\end{align}
\end{Proposition}

\begin{proof}
The proof is postponed to Appendix \ref{app:fundamental_relation_continuum}.
\end{proof}

\subsection{Solvability of the continuum backward system}

We are here going to show the solvability of the Riccati BSDE system \eqref{eq:K_BSRE_continuum}, \eqref{eq:Kbar_BSRE_continuum}, \eqref{eq:lambda_BSDE_continuum}, and \eqref{eq:Y_BSDE_continuum}.

\begin{Theorem}[Solvability of the continuum backward system]
\label{thm:solvability_continuum_backward_system}
Assume that Assumptions \ref{ass:forward_coeff} and \ref{ass:cost_coeff} hold. 

\noindent Then the backward system \eqref{eq:K_BSRE_continuum}- \eqref{eq:lambda_BSDE_continuum} admits a solution $(K,Z^K,\bar K,Z^{\bar K},Y,Z^Y,\lambda,Z^\lambda)$ with the following integrability and boundedness properties. Uniqueness is asserted within this class:

\noindent $(K,Z^K)\in S^2_{\mathbb F^0}(L^2(I;\mathbb S^d)) \times H^2_{\mathbb F^0} \big(L^2(I;\mathcal L_2(\mathbb R^{d_0};\mathbb S^d))\big)$, $(\bar K,Z^{\bar K})\in S^2_{\mathbb F^0}(L^2(I^2;\mathbb R^{d\times d})) \times H^2_{\mathbb F^0} \big(L^2(I^2;\mathcal L_2(\mathbb R^{d_0}; \mathbb R^{d\times d}))\big)$,

\noindent $(Y,Z^Y)\in S^2_{\mathbb F^0}(L^2(I;\mathbb R^d)) \times H^2_{\mathbb F^0} \big(L^2(I;\mathcal L_2(\mathbb R^{d_0};\mathbb R^d))\big)$ and \eqref{eq:lambda_BSDE_continuum} admits a unique integrable solution $(\lambda,Z^\lambda)$, with $\lambda$ of class $(D)$.

\noindent The component $K$ is non negative and satisfies $K_s(u)\in\mathbb S_+^d$, $ds\otimes du\otimes d\P\text{-a.e.}$, and there exists a constant $C_T>0$ such that $\esssup_{(s,u,\omega)}|K_s(u,\omega)|\le C_T$.

\noindent The kernel $\bar K$ is symmetric and satisfies $T_{\bar K}\in L^\infty (\Omega;C([0,T];\mathcal L(\mathsf H)))$, that is, $\sup_{0\le s\le T}\|T_{\bar K_s}\|_{\mathcal L(L^2(I;\mathbb R^d))} \le C_T$, $\P\text{-a.s.}$.
\end{Theorem}

\begin{proof}
We proceed in three steps.

\medskip
\noindent
\textbf{Step 1: Solvability of the standard Riccati BSDE for $K$.}

For $du$-a.e. $u\in I$, equation \eqref{eq:K_BSRE_continuum} is a finite-dimensional matrix-valued stochastic Riccati equation driven by the common Brownian motion $W^0$. Its coefficients are $\F^0$-adapted and uniformly bounded, the terminal condition $H(u)$ is non-negative, and the control weight $R_s(u)$ is uniformly coercive.

\noindent Under these assumptions, the local LQ problem associated with the Riccati equation has a uniformly convex cost. Hence, by the standard solvability theory for stochastic Riccati equations with random coefficients, see \cite{Bismut1976LQ}, \cite{Tang2003GeneralLQ}, \cite{Pham2016LQConditional}, and \cite{BaseiPham2019WeakMartingale}, there exists a unique bounded non-negative solution $(K(u),Z^K(u)) \in S_{\F^0}^2(\mathbb S^d) \times H_{\F^0}^2(\mathcal L_2(\R^{d_0};\S^d))$ for $du$-a.e. $u\in I$. Moreover, $K_s(u)\in\mathbb S_+^d$, $ds\otimes d\P\text{-a.e.}$, and $\esssup_{s\in[0,T]}|K_s(u)|\le C_T$, where $C_T$ depends only on the uniform bounds of the model coefficients and not on $u$. Related solvability results for stochastic Riccati equations under indefinite weights can be found in \cite{HuZhou2003IndefiniteSRE}.

\smallskip
\noindent It remains to choose jointly measurable versions of this family. Let $C_*$ be a deterministic bound for the non-negative local Riccati solutions and let $\pi$ denote the orthogonal projection, for the Frobenius inner product, onto the closed convex set $\mathcal C_*:=\{k\in\mathbb S^d:0\preceq k\preceq C_*I_d\}$. In the local Riccati driver, replace every occurrence of $k$ by $\pi(k)$. Since $R_s(u)+F_s(u)^\top\pi(k)F_s(u)\succeq c_RI_m$, the resulting driver is defined on all of $\mathbb S^d$ and is globally Lipschitz, uniformly in $(s,u,\omega)$. Its coefficients are jointly measurable. The standard Picard construction for the corresponding BSDE in $L^2(I;\mathbb S^d)$ therefore gives jointly measurable versions of its solution and its martingale integrand. For almost every $u$, uniqueness for this Lipschitz BSDE identifies the solution with the bounded non-negative local Riccati solution, because $\pi(K_s(u))=K_s(u)$. Consequently, $(K,Z^K)\in S^2_{\mathbb F^0}(L^2(I;\mathbb S^d)) \times H^2_{\mathbb F^0} \big(L^2(I;\mathcal L_2(\mathbb R^{d_0};\mathbb S^d))\big)$.

\bigskip
\noindent
\textbf{Step 2: Solvability of the abstract Riccati BSDE for $\bar K$.}

We first derive the a priori operator estimate needed to iterate the local fixed-point construction.  Let $(\bar K,Z^{\bar K})$ be a local solution on $[T_0,T]$.  Consider the homogeneous problem obtained by setting the affine coefficients $A,D,\Sigma^0,\iota$ equal to zero. In that problem $Y=0$ and the scalar component is zero.  Applying Ito's formula only to the quadratic functional determined by $(K,\bar K)$ and completing the control square gives, for $x\in L^2(I;\mathbb R^d)$, $V_s^{\rm hom}(x)
=\langle x,M_{K_s}x\rangle_\mathsf H\ +\langle x,T_{\bar K_s}x\rangle_\mathsf H\ $. 

\smallskip
\noindent This calculation uses neither the equation for $Y$ nor the scalar BSDE. The non-negativity of the homogeneous cost and the admissibility of the zero control imply $0\le V_s^{\rm hom}(x)\le C_T\|x\|_{\mathsf H}^2 $. Since $K$ is uniformly bounded, there is a deterministic $C_T$ such that $|\langle x,T_{\bar K_s}x\rangle_\mathsf H| \le C_T\|x\|_{\mathsf H}^2 $. The kernel symmetry implies that $T_{\bar K_s}$ is self-adjoint; hence $\sup_{s\in[T_0,T]} \|T_{\bar K_s}\|_{\mathcal L(\mathsf H)}\le C_T$, $\mathbb P\text{-a.s.}$. The constant is independent of $T_0$.  It therefore prevents blow-up of the local solution and permits the local fixed-point intervals to be concatenated backwards until time zero; compare \cite[Proposition~4.13]{DeFeoMekkaoui2025}.

\medskip
\noindent We now construct a solution locally in time. Fix $S<T$ and consider the interval $[S,T]$. Let $M_{\mathsf H}:= \esssup_{\omega} \|T_{G^H(\omega)}\|_{\mathcal L(\mathsf H)}$. Choose $r>M_{\mathsf H}$ and define 

\[\mathcal B(r,S):= \left\{ \bar k\in S_{\mathbb F^0}^2([S,T];\overline{\mathsf H}): \bar k_s=\bar k_s^\dagger\ \text{for all }s\in[S,T],\quad \sup_{s\in[S,T]}\|T_{\bar k_s}\|_{\mathcal L(\mathsf H)}\le r, \quad\mathbb P\text{-a.s.} \right\}\]
This is a closed subset of $S_{\F^0}^2([S,T];\overline{\mathsf H})$, hence a complete metric space.

\smallskip
\noindent We also check square-integrability in the HS norm. Set $g_s:=\|G_s^B\|_{\overline{\mathsf H}} +\|G_s^E\|_{\overline{\mathsf H}} +\|G_s^Q\|_{\overline{\mathsf H}}$. The coefficient assumptions imply $\mathbb E\int_0^T g_s^2\,ds<\infty$. Using the deterministic operator bounds on the data with the inequality $\|AB\|_{\mathrm{HS}}\le \|A\|_{\mathrm{op}}\|B\|_{\mathrm{HS}}$, we obtain $\|\mathcal F(s,0)\|_{\overline{\mathsf H}}\le Cg_s$, with $\|\mathcal F(s,\bar k_s)\|_{\overline{\mathsf H}} \le Cg_s+C_r\|\bar k_s\|_{\overline{\mathsf H}}$, $\bar k\in\mathcal B(r,S)$. Hence $\mathcal F(\cdot,\bar k)\in H^2_{\mathbb F^0}([S,T];\overline{\mathsf H})$. Doob's inequality, applied to the conditional expectation of $G^H+\int_S^T\mathcal F(q,\bar k_q)\,dq$, shows that $\Phi\bar k\in S^2_{\mathbb F^0}([S,T];\overline{\mathsf H})$.

\medskip
\noindent For $\bar k\in\mathcal B(r,S)$, define $(\Phi\bar k)_s:=\E\left[G^H+\int_s^T\mathcal F(q,\bar k_q)\,dq \;\middle|\;\Fc_s^0 \right]$, $s\in[S,T]$. By conditional Jensen's inequality and the growth estimate on $\mathcal F$,
\[
\|T_{(\Phi\bar k)_s}\|_{\mathcal L(\mathsf H)}
\le
\E\left[
\|T_{G^H}\|_{\mathcal L(\mathsf H)}
+
\int_s^T
\|T_{\mathcal F(q,\bar k_q)}\|_{\mathcal L(\mathsf H)}\,dq
\;\middle|\;\Fc_s^0
\right].
\]
Thus, for $\bar k\in\mathcal B(r,S)$, $\sup_{s\in[S,T]} \|T_{(\Phi\bar k)_s}\|_{\mathcal L(\mathsf H)} \le M_{\mathsf H}+(T-S)C_T(1+r+r^2)$.
Therefore, if $\delta>0$ is chosen such that $\delta C_T(1+r+r^2)\le r-M_\mathsf H$, then $\Phi(\mathcal B(r,S))\subset\mathcal B(r,S)$ whenever $T-S\le\delta$.

\smallskip
\noindent We next prove that $\Phi$ is a contraction. Let $\bar k,\bar\ell\in\mathcal B(r,S)$. Then
\[
(\Phi\bar k)_s-(\Phi\bar\ell)_s
=
\E\left[
\int_s^T
\big(\mathcal F(q,\bar k_q)-\mathcal F(q,\bar\ell_q)\big)\,dq
\;\middle|\;\Fc_s^0
\right].
\]
The local Lipschitz estimate and Doob's $L^2$ inequality give $\|\Phi\bar k- \Phi\bar\ell\|_{S^2([S,T];\overline{\mathsf H})} \le 2C_r(T-S) \|\bar k-\bar\ell\|_{S^2([S,T];\overline{\mathsf H})}$.
Taking $\delta>0$ smaller if necessary so that $2C_r\delta<1$, the map $\Phi$ is a contraction on $\mathcal B(r,S)$ for every interval $[S,T]$ of length at most $\delta$ (and all iterates are symmetric, since the terminal kernel is symmetric and $\mathcal F$ preserves symmetric kernels).

\medskip
\noindent By Banach's fixed point theorem, there exists a unique $\bar K\in\mathcal B(r,S)$ such that $\bar K_s = \E\left[ G^H+\int_s^T\mathcal F(q,\bar K_q)\,dq \;\middle|\;\Fc_s^0 \right]$.
Define $M_s := \bar K_s+\int_S^s\mathcal F(q,\bar K_q)\,dq$.
Then $M$ is a square-integrable $\overline{\mathsf H}$-valued $\F^0$-martingale. Since $\overline{\mathsf H}$ is a separable Hilbert space and $\F^0$ is generated by the Brownian motion $W^0$, the Hilbert-valued martingale representation theorem yields a unique $Z^{\bar K} \in H_{\F^0}^2
([S,T];\mathcal L_2(\mathbb R^{d_0};\overline{\mathsf H}))$ such that $M_s=M_S+\int_S^s Z_q^{\bar K}\,dW_q^0$.
For the Hilbert-space martingale representation and linear BSDE arguments used here, see for instance \cite{FabbriGozziSwiech2017}. Therefore, $(\bar K,Z^{\bar K})$ solves the abstract Riccati BSDE on $[S,T]$.

\smallskip
\noindent Let $C_*$ be the deterministic a priori bound above, enlarged if necessary to dominate $\|T_{G^H}\|_{L^\infty(\Omega;\mathcal L(\mathsf H))}$, and choose $r=C_*+1$. Choose $\delta>0$ so that $\delta C_T(1+r+r^2)\le 1$, $2C_r\delta<1$. At the first step, the terminal kernel is $G^H$ at time $T$. If the solution has been constructed on $[b,T]$, repeat the same construction on $[a,b]$, where $b-a\le\delta$, with terminal kernel $\bar K_b$ in place of $G^H$. This terminal kernel is symmetric, belongs to $L^2(\Omega;\overline{\mathsf H})$, and satisfies $\|T_{\bar K_b}\|_{\mathrm{op}}\le C_*$. The two solutions concatenate continuously at $b$. Apply the homogeneous verification argument to the concatenated solution on $[a,T]$, with the original terminal cost at $T$. It yields the same bound $C_*$, independently of $a$. Thus finitely many steps cover $[0,T]$ and give $(\bar K,Z^{\bar K})\in S^2_{\mathbb F^0}(\overline{\mathsf H}) \times H^2_{\mathbb F^0} (\mathcal L_2(\mathbb R^{d_0};\overline{\mathsf H}))$, $T_{\bar K}\in L^\infty(\Omega;C([0,T];\mathcal L(\mathsf H)))$. The solution is symmetric by construction. Uniqueness among symmetric solutions with a deterministic uniform operator bound follows by applying the same local Lipschitz estimate backwards on a common finite partition.

\medskip
\noindent
\textbf{Step 3: Solvability of the linear BSDE for $Y$.}
The driver of \eqref{eq:Y_BSDE_continuum} can be written as $\mathcal A_sY_s+f_s$, where $\mathcal A_s$ is a uniformly bounded linear operator on $\mathsf H$. Moreover, the definitions of $M$, $V$ and $\Gamma$, together with the boundedness of $\Sigma^0$, give $\|f_s\|_{\mathsf H}\le C\Big( 1+\|A_s\|_{\mathsf H}+\|D_s\|_{\mathsf H}+\|\iota_s\|_{\mathsf U} +\|Z_s^K\|_{\mathcal L_2(\mathbb R^{d_0};L^2(I;\mathbb S^d))} +\|Z_s^{\bar K}\|_{\mathcal L_2(\mathbb R^{d_0};\overline{\mathsf H})} \Big)$. Therefore, $\E\int_0^T\|f_s\|_{\mathsf H}^2\,ds<\infty$. 

\smallskip
\noindent Hence \eqref{eq:Y_BSDE_continuum} is a linear BSDE in the separable Hilbert space $\mathsf H$. By the standard theory, see \cite[Proposition 6.20]{FabbriGozziSwiech2017}, it admits a unique solution $(Y,Z^Y)\in S^2_{\mathbb F^0}(L^2(I;\mathbb R^d)) \times H^2_{\mathbb F^0} \big(L^2(I;\mathcal L_2(\mathbb R^{d_0};\mathbb R^d))\big)$. Finally, the existence and uniqueness of the scalar component $(\lambda,Z^\lambda)$ follow from Remark~\ref{rem:triangular_structure_backward_system}.
\end{proof}

\begin{Remark}\label{rem:construction_value_function}
For completeness, the quadratic value formula also yields the measurable functional asserted in Proposition~\ref{prop:law_invariance}. For a marked input law $\rho$, let $\mu$ be its $(u,x)$-marginal. Choose a jointly Borel disintegration $\mu(du,dx)=du\,\mu^u(dx)$ and set $m_\mu(u)=\int x\,\mu^u(dx)$. Using Borel versions of the backward fields as functions of the common-noise past, replace the conditional expectations in
\eqref{eq:value_function_continuum_verification} by integration against $\mu$, and $\bar\xi$ by $m_\mu$. This defines a Borel functional of the common-noise past and $\rho$. The integrals are finite by the boundedness of $K$ and $T_{\bar K}$, the $L^2$-regularity of $Y$, and the second-moment condition on $\mu$. In particular, this functional depends on $\rho$ only through its $(u,x)$-marginal $\mu$.
\end{Remark}

\subsection{Verification theorem}

The fundamental relation in Proposition \ref{prop:fundamental_relation_continuum}
suggests considering the feedback control obtained by cancelling the non-negative quadratic term in
\eqref{eq:fundamental_relation_continuum}. More precisely, substituting the feedback expression into the state dynamics leads to the closed-loop equation :
\begin{equation}\label{eq:closed_loop_continuum}
\left\{
\begin{aligned}
d\widehat X_s
&=
\Big[
A_s(U)+B_s(U)\widehat X_s
+
T_{G_s^B}(\widehat{\bar X}_s)(U)
+
C_s(U)\widehat\alpha_s
\Big]ds
\\
&\quad
+
\Big[
D_s(U)+E_s(U)\widehat X_s
+
T_{G_s^E}(\widehat{\bar X}_s)(U)
+
F_s(U)\widehat\alpha_s
\Big]dW_s
+
\Sigma_s^0(U)dW_s^0,
\\
\widehat X_t&=\xi,
\end{aligned}
\right.
\end{equation}
with $\widehat{\bar X}_s(u):=\E[\widehat X_s\mid \Fc_s^0,U=u]$ and $\widehat\alpha$ being defined by
\begin{align}
\widehat \alpha_s
=
-O_s(U)^{-1}
\left(
U_s(U)\widehat X_s
+
\int_I V_s(U,v)\widehat{\bar X}_s(v)\,dv
+
\Gamma_s(U)
\right).
\label{eq:feedback_continuum_hat}
\end{align}

\begin{Theorem}[Verification theorem for the continuum problem]
\label{thm:verification_continuum}
Assume that Assumptions \ref{ass:forward_coeff} and \ref{ass:cost_coeff} hold, and let $(K,Z^K,\bar K,Z^{\bar K},Y,Z^Y,\lambda,Z^\lambda)$ be the unique solution of the backward system \eqref{eq:K_BSRE_continuum} - \eqref{eq:lambda_BSDE_continuum} given by Theorem \ref{thm:solvability_continuum_backward_system}.

\noindent Then, for every $t\in[0,T]$ and every $\xi\in\Ic_t$, the closed-loop equation \eqref{eq:closed_loop_continuum} admits a unique solution $\widehat X\in S_{\mathbb G^{t,\xi}}^2([t,T];\mathbb R^d)$.
Moreover, the feedback process $\widehat \alpha$ defined by
\eqref{eq:feedback_continuum_hat} is admissible, i.e. $\widehat\alpha\in\Ac_t$, and it is the unique optimal control.

\noindent The value function is given by
\begin{align}
V_t(\xi)
=
\E\Big[
&\langle \xi,K_t(U)\xi\rangle
+
\left\langle
\bar \xi(U),
(T_{\bar K_t}\bar \xi)(U)
\right\rangle
+
2\langle Y_t(U),\xi\rangle
\;\Big|\;\Fc_t^0
\Big]+\lambda_t,
\label{eq:value_function_continuum_verification}
\end{align}
Equivalently, $V_t(\xi)=\mathcal V_t(\xi)$, where $\mathcal V_t(\xi)$ is the quantity defined in Proposition \ref{prop:fundamental_relation_continuum} and $V_t$ in \eqref{eq:conditional_cost}.
\end{Theorem}

\begin{proof}
From the fundamental relation \eqref{eq:fundamental_relation_continuum} we have $ J(t,\xi,\alpha)\ge\mathcal V_t(\xi)$ $\forall \alpha \in \Ac_t$ and equality if $\alpha$ satisfies \eqref{eq:feedback_continuum_hat}. We also see that $\hat \alpha \in \Ac_t$ by Theorem\ref{thm:wellposed_limit_state} with existence/uniqueness of \eqref{eq:closed_loop_continuum} that gives the optimality of our unique control $\hat \alpha$ defined in \eqref{eq:feedback_continuum_hat}
\end{proof}

\section{Solution of the discrete problem}\label{sec:discrete_LQ}

In this section, we solve the finite-population LQ problem  \eqref{eq:finite_N_state}--\eqref{eq:finite_N_cost}  introduced in Section \ref{sec:problem_discrete}. Unlike the continuum case, this is a  genuinely finite-dimensional stochastic LQ problem in $\R^{Nd}$, which we solve through a quadratic ansatz on the global state  $\mathbf X^N := (X^{1,N},\dots,X^{N,N}) \in \R^{Nd}$. We then  exploit the block structure of the resulting backward objects to  prepare the discrete-to-continuum convergence analysis carried out in Section \ref{sec:convergence}.

Throughout this section, we assume that Assumptions 
\ref{ass:finite_forward_coeff} and \ref{ass:finite_cost_coeff} hold.

\subsection{Vector formulation and quadratic ansatz}

For notational convenience, we rewrite the finite-$N$ system 
\eqref{eq:finite_N_state} in vector form. Let 
$\boldsymbol\alpha_s^N := (\alpha_s^{1,N},\dots,\alpha_s^{N,N}) \in \R^{Nm}$
and define the global drift matrix $\mathbf B_s^N \in \R^{Nd\times Nd}$ 
block-wise by
\[
(\mathbf B_s^N)_{ij}
:=
\delta_{ij}B_s^{i,N} + \frac{1}{N}G_{B,s}^N(u_i^N,u_j^N),
\qquad i,j=1,\dots,N.
\]
The control matrix $\mathbf C_s^N \in \R^{Nd\times Nm}$ is 
block-diagonal with $(\mathbf C_s^N)_{ii} = C_s^{i,N}$. For each 
$r=1,\dots,N$, the diffusion matrix $\mathbf E_s^{r,N} \in \R^{Nd\times Nd}$ 
has its only non-zero block-row at index $r$, namely
\[
(\mathbf E_s^{r,N})_{rj}
:=
\delta_{rj}E_s^{r,N} + \frac{1}{N}G_{E,s}^N(u_r^N,u_j^N),
\]
and analogously $\mathbf F_s^{r,N} \in \R^{Nd\times Nm}$ has its only 
non-zero block at position $(r,r)$ equal to $F_s^{r,N}$. The additive 
vectors $\mathbf A^N, \mathbf D^{r,N} \in \R^{Nd}$, $\boldsymbol\Sigma^{0,N} \in \R^{Nd\times d_0}$ and $\boldsymbol\iota_s^N \in \R^{Nm}$ are defined componentwise in the obvious way.

\noindent Similarly, the cost matrices $\mathbf Q_s^N, \mathbf H^N \in \R^{Nd\times Nd}$ 
and $\mathbf R_s^N \in \R^{Nm\times Nm}$ are defined block-wise by
\[
(\mathbf Q_s^N)_{ij}
:= 
\delta_{ij}\frac{1}{N}Q_s^{i,N} + \frac{1}{N^2}G_{Q,s}^N(u_i^N,u_j^N),
\quad
(\mathbf H^N)_{ij}
:= 
\delta_{ij}\frac{1}{N}H^{i,N} + \frac{1}{N^2}G_H^N(u_i^N,u_j^N),
\]
and $\mathbf R_s^N$ is block-diagonal with $(\mathbf R_s^N)_{ii} 
= \frac{1}{N}R_s^{i,N}$. Under Assumption \ref{ass:finite_cost_coeff}, 
$\mathbf R_s^N \succeq \frac{c_R}{N}I_{Nm}$ and $\mathbf Q_s^N, 
\mathbf H^N \succeq 0$.

\noindent With these notations, the system \eqref{eq:finite_N_state} becomes
\begin{equation}\label{eq:discrete_global_state}
\left\{
\begin{aligned}
d\mathbf X_s^N
&=
\left[
\mathbf A_s^N+\mathbf B_s^N\mathbf X_s^N+\mathbf C_s^N\boldsymbol\alpha_s^N
\right]\,ds
+
\sum_{r=1}^N
\left[
\mathbf D_s^{r,N}
+
\mathbf E_s^{r,N}\mathbf X_s^N
+
\mathbf F_s^{r,N}\boldsymbol\alpha_s^N
\right]\,dW_s^{r,N}
+
\mathbf\Sigma_s^{0,N}\,dW_s^0,
\\
\mathbf X_t^N&=\xi^N.
\end{aligned}
\right.
\end{equation}
and the cost \eqref{eq:finite_N_cost} becomes
\begin{align}
J^N(t,\xi^N,\alpha^N)
=
\E\left[
\int_t^T
\left(
(\mathbf X_s^N)^\top\mathbf Q_s^N\mathbf X_s^N
+
(\boldsymbol\alpha_s^N+\boldsymbol{\iota}_s^N)^\top\mathbf R_s^N
(\boldsymbol\alpha_s^N+\boldsymbol\iota_s^N)
\right)\,ds
+
(\mathbf X_T^N)^\top\mathbf H^N\mathbf X_T^N
\;\middle|\;\Fc_t^0
\right].
\label{eq:discrete_cost_global}
\end{align}

\paragraph{Quadratic Ansatz}
As in the continuum case, we look for a value function of quadratic form
\begin{equation}\label{eq:discrete_quadratic_ansatz}
v_s^N(\mathbf X_s^N)
:=
(\mathbf X_s^N)^\top P_s^N\mathbf X_s^N + 2(p_s^N)^\top \mathbf X_s^N + q_s^N,
\end{equation}
with $\F^0$-adapted backward processes $P^N \in \mathbb S^{Nd}$, 
$p^N \in \R^{Nd}$, $q^N \in \R$, satisfying
\begin{equation}\label{eq:discrete_backward_dynamics}
dP_s^N = -\mathscr F_s^{P,N}ds + Z_s^{P,N}dW_s^0,
\quad
dp_s^N = -\mathscr F_s^{p,N}ds + Z_s^{p,N}dW_s^0,
\quad
dq_s^N = -\mathscr F_s^{q,N}ds + Z_s^{q,N}dW_s^0,
\end{equation}
with terminal conditions $P_T^N = \mathbf H^N$, $p_T^N = 0$, $q_T^N = 0$.

\noindent The completion-of-squares coefficients are
\begin{align}
\mathcal O_s^N &:= \mathbf R_s^N + \sum_{r=1}^N (\mathbf F_s^{r,N})^\top P_s^N \mathbf F_s^{r,N},
\label{eq:def_O_discrete}\\
\mathcal U_s^N &:= (\mathbf C_s^N)^\top P_s^N + \sum_{r=1}^N (\mathbf F_s^{r,N})^\top P_s^N \mathbf E_s^{r,N},
\label{eq:def_U_discrete}\\
\boldsymbol\Gamma_s^N &:= (\mathbf C_s^N)^\top p_s^N + \sum_{r=1}^N (\mathbf F_s^{r,N})^\top P_s^N \mathbf D_s^{r,N} + \mathbf R_s^N \boldsymbol \iota_s^N.
\label{eq:def_Gamma_discrete}
\end{align}

\begin{Remark}\label{rem:O_inverse_discrete}
Under Assumption \ref{ass:finite_cost_coeff}, if $P_s^N \succeq 0$, 
then block-wise $\mathcal O_s^N$ is block-diagonal with $(\mathcal O_s^N)_{ii} = \frac{1}{N}R_s^{i,N} + (F_s^{i,N})^\top P_{ii,s}^N F_s^{i,N}$, and we have $(\mathcal O_s^N)_{ii} \succeq \frac{c_R}{N}I_m$. Therefore $\mathcal O_s^N$ is invertible and $|(\boldsymbol{\mathcal O}^N)^{-1}|_{F}\leq \sqrt{m} N^{3/2}/c_R$, $\P$-a.s. The factor $N$ is consistent with the $1/N$ scaling of the cost. In the convergence analysis, this apparent growth will be compensated by the natural scaling of the blocks of $P^N$ and $p^N$; see Section~\ref{sec:convergence}.
\end{Remark}

\subsection{Discrete Backward Riccati system}

\paragraph{Riccati BSDE for $P^N$.}
Applying It\^o's formula to the ansatz \eqref{eq:discrete_quadratic_ansatz} 
and matching the quadratic terms in $\mathbf X^N$ with those of the running 
cost yields the Riccati BSDE
\begin{equation}\label{eq:discrete_Riccati_P}
\left\{
\begin{aligned}
dP_s^N &= -\Big[
\mathbf Q_s^N + (\mathbf B_s^N)^\top P_s^N + P_s^N\mathbf B_s^N
+\sum_{r=1}^N(\mathbf E_s^{r,N})^\top P_s^N\mathbf E_s^{r,N}
- (\mathcal U_s^N)^\top (\mathcal O_s^N)^{-1}\mathcal U_s^N
\Big]\,ds + Z_s^{P,N}\,dW_s^0,
\\
P_T^N &= \mathbf H^N.
\end{aligned}
\right.
\end{equation}

\paragraph{Linear BSDE for $p^N$.}
Matching the linear terms in $\mathbf X^N$ yields
\begin{equation}\label{eq:discrete_linear_p}
\left\{
\begin{aligned}
dp_s^N &= -\Big[
P_s^N\mathbf A_s^N + Z_s^{P,N}\boldsymbol\Sigma_s^{0,N}
+ (\mathbf B_s^N)^\top p_s^N
+\sum_{r=1}^N(\mathbf E_s^{r,N})^\top P_s^N\mathbf D_s^{r,N}
- (\mathcal U_s^N)^\top (\mathcal O_s^N)^{-1}\boldsymbol\Gamma_s^N
\Big]\,ds + Z_s^{p,N}\,dW_s^0,
\\
p_T^N &= 0.
\end{aligned}
\right.
\end{equation}

\paragraph{Scalar BSDE for $q^N$.}
Matching the constant terms gives
\begin{equation}\label{eq:discrete_scalar_q}
\left\{
\begin{aligned}
dq_s^N &= -\Big[
2(\mathbf A_s^N)^\top p_s^N
+
\sum_{r=1}^N
\left\langle
\mathbf D_s^{r,N},
P_s^N\mathbf D_s^{r,N}
\right\rangle_{\mathrm F}
+
\left\langle
\boldsymbol\Sigma_s^{0,N},
P_s^N\boldsymbol\Sigma_s^{0,N}
\right\rangle_{\mathrm F}
+
2\left\langle
Z_s^{p,N},
\boldsymbol\Sigma_s^{0,N}
\right\rangle_{\mathrm F}
+ (\boldsymbol \iota_s^N)^\top\mathbf R_s^N\boldsymbol \iota_s^N
- (\boldsymbol\Gamma_s^N)^\top(\mathcal O_s^N)^{-1}\boldsymbol\Gamma_s^N
\Big]\,ds\\ 
&+ Z_s^{q,N}\,dW_s^0,
\\
q_T^N &= 0.
\end{aligned}
\right.
\end{equation}

\begin{Remark}\label{rem:triangular_discrete}
The discrete backward system has the same triangular structure as  the continuum system: equation \eqref{eq:discrete_Riccati_P} for  $P^N$ is independent of the additive coefficients  $\mathbf A^N, \mathbf D^{r,N}, \boldsymbol\Sigma^{0,N}, \mathbf \iota^N$; once $P^N$ is known, \eqref{eq:discrete_linear_p} is a linear BSDE for $(p^N,Z^{p,N})$; once $(P^N,p^N,Z^{p,N})$ are known, the driver of \eqref{eq:discrete_scalar_q} is prescribed and the scalar equation is solved by conditional expectation.
\end{Remark}

\subsection{Fundamental relation, solvability, and verification}

We will here prove the fundamental relation associated to our discrete system, along the solvability of and the verification theorem.

\begin{Proposition}[Fundamental relation, discrete case]
\label{prop:fundamental_relation_discrete}
Assume that the discrete backward system 
\eqref{eq:discrete_Riccati_P}--\eqref{eq:discrete_scalar_q} 
admits a solution $(P^N,Z^{P,N},p^N,Z^{p,N},q^N,Z^{q,N})$ in the solution class stated in Theorem~\ref{thm:solvability_discrete}. Let $t\in[0,T]$, $\xi^N\in\Ic_t^N$, 
and $\alpha^N\in\Ac_t^N$. Then
\begin{equation}\label{eq:fundamental_relation_discrete}
J^N(t,\xi^N,\alpha^N)
=
\mathcal V_t^N(\xi^N)
+
\E\Big[
\int_t^T(\boldsymbol\alpha_s^N-\widehat{\boldsymbol\alpha}_s^{N,\alpha})^\top\mathcal O_s^N(\boldsymbol\alpha_s^N-\widehat{\boldsymbol\alpha}_s^{N,\alpha})\,ds
\;\Big|\;\Fc_t^0
\Big],
\quad \P\text{-a.s.},
\end{equation}
where 
\[
\mathcal V_t^N(\xi^N) := \E\left[ (\xi^N)^\top P_t^N\xi^N + 2(p_t^N)^\top\xi^N \;\middle|\; \Fc_t^0 \right] + q_t^N
\] 
and the feedback expression along $\alpha$ is
\[
\widehat{\boldsymbol\alpha}_s^{N,\alpha}
:=
-(\mathcal O_s^N)^{-1}\big(\mathcal U_s^N \mathbf X_s^N + \boldsymbol\Gamma_s^N\big).
\]
\end{Proposition}

\begin{proof}
For $s \in [t,T]$, let $\mathcal S_s^{N,\alpha}:= (\mathbf X_s^N)^\top P_s^N\mathbf X_s^N + 2(p_s^N)^\top\mathbf X_s^N + q_s^N + \int_t^s \left[ (\mathbf X_r^N)^\top\mathbf Q_r^N\mathbf X_r^N + (\boldsymbol\alpha_r^N+\boldsymbol \iota_r^N)^\top\mathbf R_r^N
(\boldsymbol\alpha_r^N+\boldsymbol\iota_r^N) \right]dr.$
Applying It\^o's formula applied to $\mathcal S_s^{N,\alpha}$, and using \eqref{eq:discrete_global_state}, square completion in the 
control, and the fact that $(P^N, p^N, q^N)$ solve the backward 
system \eqref{eq:discrete_Riccati_P}--\eqref{eq:discrete_scalar_q} yields to the result.
The argument is the finite-dimensional analogue of the proof of 
Proposition \ref{prop:fundamental_relation_continuum}, and is detailed in Appendix \ref{app:fundamental_relation_discrete}.
\end{proof}

In this theorem, we are gonna prove the solvability of our discrete backwad system arising from our optimal control problem.

\begin{Theorem}[Solvability of the discrete backward system]
\label{thm:solvability_discrete}
Under Assumptions \ref{ass:finite_forward_coeff} and 
\ref{ass:finite_cost_coeff}, for every $N\ge1$, the discrete 
backward system \eqref{eq:discrete_Riccati_P}--\eqref{eq:discrete_scalar_q} 
admits a unique solution $(P^N,Z^{P,N},p^N,Z^{p,N},q^N,Z^{q,N})$ with
\[
P^N \in S^\infty_{\F^0}(\mathbb S_+^{Nd}),
\quad
Z^{P,N} \in H^2_{\F^0}(\mathcal L_2(\mathbb R^{d_0};\mathbb S^{Nd})),
\quad
(p^N,Z^{p,N}) \in S^2_{\F^0}(\R^{Nd})\times H^2_{\F^0}((\R^{Nd})^{d_0})
\]
Moreover, the scalar equation \eqref{eq:discrete_scalar_q} admits a unique integrable solution $(q^N,Z^{q,N})$, where $q^N$ is of class $(D)$ and $Z^{q,N}$ is locally square-integrable.
\end{Theorem}
\begin{proof}
\noindent \textbf{Solvability of $P^N$ :}
Equation \eqref{eq:discrete_Riccati_P} is a finite-dimensional 
matrix-valued backward stochastic Riccati equation of the type studied 
in \cite{Bismut1976LQ,Tang2003GeneralLQ,Tang2015LQ,YongZhou1999}. 
Under the boundedness of the coefficients (Assumption \ref{ass:finite_forward_coeff}), 
the non-negativity of $\mathbf Q_s^N$ and $\mathbf H^N$, and the 
uniform coercivity $\mathbf R_s^N \succeq (c_R/N) I_{Nm}$ 
(Assumption \ref{ass:finite_cost_coeff}), the standard theory yields 
existence and uniqueness of a non-negative bounded solution 
$P^N \in S^\infty_{\F^0}(\mathbb S_+^{Nd})$, together with 
$Z^{P,N} \in H^2(\mathcal L_2(\R^{d_0};\mathbb S^{Nd}))$.

\medskip
\noindent \textbf{Solvability of $p^N$ :}
Once $P^N$ is known, the equation \eqref{eq:discrete_linear_p} for 
$p^N$ becomes a linear BSDE in $\R^{Nd}$ with bounded linear coefficient and 
square-integrable source term. By the classical theory of linear 
BSDEs \cite{PardouxPeng1990,ElKarouiPengQuenez1997,FuhrmanTessitore2002,GuatteriTessitore2005}, it admits a 
unique solution $(p^N, Z^{p,N}) \in S^2_{\F^0}(\R^{Nd})\times H^2(\mathcal L_2(\R^{d_0};\R^{Nd}))$.

\medskip
\noindent\textbf{Solvability of $q^N$ :}
Let $\ell^N$ denote the driver of \eqref{eq:discrete_scalar_q}. Once $(P^N,p^N,Z^{p,N})$ are known, the assumptions imply $\E\int_0^T|\ell_s^N|\,ds<\infty$. Therefore, $q_t^N =\E\left[ \int_t^T\ell_s^N\,ds \;\middle|\; \Fc_t^0 \right]$ defines a process of class $(D)$. The Brownian martingale representation theorem, applied after localization, yields a predictable locally square-integrable process $Z^{q,N}$. Uniqueness follows from the conditional representation.
\end{proof}

\subsection{Verification theorem and optimal feedback}

The fundamental relation \eqref{eq:fundamental_relation_discrete} 
suggests considering the feedback control obtained by cancelling the 
quadratic term, which leads to the closed-loop equation
\begin{equation}\label{eq:discrete_closed_loop}
\left\{
\begin{aligned}
d\widehat{\mathbf X}_s^N
&=
\left[
\mathbf A_s^N+\mathbf B_s^N\widehat{\mathbf X}_s^N
+\mathbf C_s^N\widehat{\boldsymbol\alpha}_s^N
\right]ds+
\sum_{r=1}^N
\left[
\mathbf D_s^{r,N}
+
\mathbf E_s^{r,N}\widehat{\mathbf X}_s^N
+
\mathbf F_s^{r,N}\widehat{\boldsymbol\alpha}_s^N
\right]dW_s^{r,N}
+
\mathbf\Sigma_s^{0,N}\,dW_s^0,
\\
\widehat{\mathbf X}_t^N&=\xi^N.
\end{aligned}
\right.
\end{equation}
with feedback
\begin{equation}\label{eq:discrete_optimal_feedback}
\widehat{\boldsymbol\alpha}_s^N 
= 
-(\mathcal O_s^N)^{-1}\big(\mathcal U_s^N\widehat{\mathbf X}_s^N + \boldsymbol\Gamma_s^N\big).
\end{equation}

\begin{Theorem}[Verification, discrete case]\label{thm:verification_discrete}
Under Assumptions \ref{ass:finite_forward_coeff} and 
\ref{ass:finite_cost_coeff}, let $(P^N, p^N, q^N)$ be the unique 
solution of the discrete backward system given by 
Theorem \ref{thm:solvability_discrete}. Then, for every $t\in[0,T]$ 
and $\xi^N\in\Ic_t^N$, the closed-loop equation 
\eqref{eq:discrete_closed_loop} admits a unique solution 
$\widehat{\mathbf X}^N \in S^2_{\F^N}([t,T];\R^{Nd})$. The feedback 
$\widehat{\boldsymbol\alpha}^N$ defined by \eqref{eq:discrete_optimal_feedback} 
is admissible and is the unique optimal control. The value function is
\begin{equation}
\label{eq:discrete_value_function}
V_t^N(\xi^N)
=
\E\left[
(\xi^N)^\top P_t^N\xi^N
+
2(p_t^N)^\top\xi^N
\;\middle|\;
\Fc_t^0
\right]
+
q_t^N,
\qquad
\P\text{-a.s.}
\end{equation}
\end{Theorem}

\begin{proof}
The proof is analogous to the continuum case (Theorem \ref{thm:verification_continuum}): 
well-posedness of the closed-loop equation follows from  Theorem \ref{thm:wellposed_finite_state} applied to the modified  coefficients (which satisfy Assumption \ref{ass:finite_forward_coeff}  since the feedback coefficients are progressively measurable and satisfy the required square-integrability estimates for fixed $N$);  admissibility of $\widehat{\boldsymbol\alpha}^N$ follows from  $\widehat{\mathbf X}^N \in S^2$; optimality and uniqueness follow  from the fundamental relation \eqref{eq:fundamental_relation_discrete} combined with the strict positive definiteness of $\mathcal O_s^N$.
\end{proof}

\begin{Proposition}[Full-state feedback structure of the optimal discrete control]
\label{prop:markovian_feedback}
The optimal discrete control $\widehat\alpha^N$ obtained in Theorem~\ref{thm:verification_discrete} admits the closed-loop representation $\widehat{\boldsymbol\alpha}_t^N =-(\mathcal O_t^N)^{-1} \big(\mathcal U_t^N\mathbf X_t^N+\boldsymbol\Gamma_t^N\big)$, or componentwise, $\widehat\alpha_t^{i,N}=-(O_{i,t}^N)^{-1} \left(\widehat U_{i,t}^NX_t^{i,N} +h_N\sum_{j\ne i}V_{ij,t}^NX_t^{j,N}+\Gamma_{i,t}^N\right)$, $i=1,\ldots,N$.

\noindent The operators $O_t^N,U_t^N$ and the vector $\Gamma_t^N$ are constructed from the backward Riccati objects $(P^N,p^N)$ and are $\mathcal F_t^0$-adapted. Consequently, $\widehat\alpha_t^N$ is measurable with respect to $\sigma(\mathbf X_t^N)\vee\mathcal F_t^0$.In particular, $\widehat\alpha^N$ is admissible for the centralized finite-dimensional problem.
\end{Proposition}
\begin{proof}
The feedback formula is obtained in the verification argument by completing
the square in the control variable. The backward coefficients
\((P^N,p^N)\), and hence $O^N,U^N,\Gamma^N$, are driven only by the
common noise and are therefore $\mathbb F^0$-adapted. The remaining
randomness in the feedback enters through the current full state vector
$\mathbf X_t^N$. Hence the feedback is measurable with respect to
$\sigma(\mathbf X_t^N)\vee\mathcal F_t^0$, and therefore progressively
measurable with respect to the centralized filtration. This proves
admissibility in the centralized problem. Optimality follows from the
verification theorem.
\end{proof}

\begin{Remark}\label{rem:discrete_blocks}
The decomposition introduced in Definition~\ref{def:embeddings} is exactly $P_{ii,t}^N=h_NK_{i,t}^N$, $P_{ij,t}^N=h_N^2\bar K_{ij,t}^N$, $i\neq j$. Since the interaction embedding is restricted to the off-diagonal blocks and $\bar K_{ii}^N=0$ by convention, the interaction contributions carried by the diagonal blocks of the global Riccati equation are not discarded. They appear as explicit diagonal remainders in the equation satisfied by $K^N$. These reminders must be retained in the discrete-to-continuum comparison.
\end{Remark}
 
    \section{Discrete-to-continuum convergence of the Riccati systems}
\label{sec:convergence}

In this section, we establish the quantitative convergence of the discrete backward objects $(K^N, \bar K^N, Y^N, q^N)$ to the continuum objects $(K, \bar K, Y, \lambda)$ as $N \to \infty$. The argument proceeds in three stages. We first introduce, in §\ref{subsec:setup}, the regularity framework and aggregate error rates that quantify the discrepancy between the discrete and continuum models. We then derive, in §\ref{subsec:apriori}, a set of uniform-in-$N$ a priori bounds on the discrete backward objects. Building on these, we prove in §\ref{subsec:stability} an abstract stability lemma for Hilbert-valued quadratic backward equations, and combine it with a decomposition of the discrete drivers to obtain the main convergence theorem.
The state consistency estimates used below are related to classical propagation-of-chaos arguments \cite{Sznitman1991,Meleard1996}, but the label-dependent non-exchangeable structure requires estimates in graphon-type norms.

\subsection{Setup, embeddings and assumptions}
\label{subsec:setup}

We work with the uniform grid $u_i^N$, the grid mesh $h_N$  and the partition $I_i^N$ introduced in Section~\ref{sec:notations}. The standing assumptions of the continuum and finite problems remain in force throughout this section, we now impose additional regularity and consistency conditions for the quantitative comparison. Throughout this section, every discrete object $(\cdot)^N$ indexed by $i$ (or by $(i,j)$ for two-index objects) is identified with its step-function embedding $K_t^N(u) := K_{i,t}^N$ for $u \in I_i^N$, and for two-index objects $\bar K_t^N(u,v) := \bar K_{ij,t}^N$ for $(u,v) \in I_i^N \times I_j^N$,
with the analogous embeddings for $Y^N, Z^{K,N}, Z^{\bar K,N}, Z^{Y,N}$ and for the discrete coefficients and kernels. We point out that, with these embeddings,
\[
\|\bx^N\|^2_{L^2(I; \R^d)} = \frac{1}{N}\sum_{i=1}^N |x_i^N|^2,
\qquad
\|M^N\|^2_{L^2(I^2; \R^{d \times d})} = \frac{1}{N^2}\sum_{i,j=1}^N |M_{ij}^N|^2,
\]
so that the natural $L^2$ norms of the embedded objects coincide with the rescaled $\ell^2$ norms used in the discrete vector formulation of Section~\ref{sec:problem_discrete}.

\begin{Definition}[Backward Riccati objects in continuum scale]
\label{def:embeddings}
Recalling the rescalings introduced in Section~\ref{sec:problem_discrete}, we set, for $1 \leq i, j \leq N$,
\[
K_{i,t}^N:=h_N^{-1}P_{ii,t}^N,
\qquad
\bar K_{ij,t}^N:=h_N^{-2}P_{ij,t}^N,\quad i\neq j,
\qquad
Y_{i,t}^N:=h_N^{-1}p_{i,t}^N.
\]
The corresponding martingale integrands are rescaled by
\[
Z_{i,t}^{K,N}:=h_N^{-1}Z_{ii,t}^{P,N},
\qquad
Z_{ij,t}^{\bar K,N}:=h_N^{-2}Z_{ij,t}^{P,N},\quad i\neq j,
\qquad
Z_{i,t}^{Y,N}:=h_N^{-1}Z_{i,t}^{p,N}.
\]
We set $\bar K_{ii}^N=0$ and  $Z_{ii}^{\bar K,N}=0$ by convention. The corresponding step-function embeddings are
\[
K_t^N(u)
=
\sum_{i=1}^N K_{i,t}^N\mathbf 1_{I_i^N}(u),
\qquad
\bar K_t^N(u,v)
=
\sum_{i\neq j}\bar K_{ij,t}^N
\mathbf 1_{I_i^N}(u)\mathbf 1_{I_j^N}(v),
\qquad
Y_t^N(u)
=
\sum_{i=1}^NY_{i,t}^N\mathbf 1_{I_i^N}(u),
\]
with analogous embeddings for $Z^{K,N},Z^{\bar K,N},Z^{Y,N}$. The scalar component is unchanged, $q^N$ denotes the constant term of the discrete value function obtained in Theorem~\ref{thm:verification_discrete}.
\end{Definition}

\paragraph{Exact diagonal/off-diagonal decomposition.}

For step functions $k=(k_i)_{i=1}^N$ with $k_i\in\mathbb S^d$ and $\bar k=(\bar k_{ij})_{i\neq j}$, with $\bar k_{ii}=0$, let $P^N[k,\bar k]$ be the block matrix defined by
$P_{ij}^N[k,\bar k]
=
h_N k_i\mathbf 1_{\{i=j\}}
+
h_N^2\bar k_{ij}\mathbf 1_{\{i\neq j\}}$. We denote by $\mathcal K_N^+ := \left\{ k=(k_i)_{i=1}^N: k_i\in\mathbb S_+^d \text{ for every }i \right\}$. Whenever the rescaled gains below are used, we take $k\in\mathcal K_N^+$. 

\noindent For $k\in\mathcal K_N^+$ and $1\le i\le N$, set
\[
\widehat E_i^N
:=
E_i^N+h_NG_{ii}^{E,N},
\quad
O_i^N(k)
:=
R_i^N+(F_i^N)^\top k_iF_i^N,
\quad
U_i^N(k)
:=
(C_i^N)^\top k_i+(F_i^N)^\top k_iE_i^N,
\quad
\widehat U_i^N(k)
:=
(C_i^N)^\top k_i+(F_i^N)^\top k_i\widehat E_i^N.
\]
For $i\neq j$, define $V_{ij}^N(k,\bar k) := (C_i^N)^\top\bar k_{ij} + (F_i^N)^\top k_iG_{ij}^{E,N}$ and for later use, we also introduce the rescaled affine gain $\Gamma_i^N(k,y) := (C_i^N)^\top y_i + (F_i^N)^\top k_iD_i^N + R_i^N\iota_i^N$. The global affine gain satisfies $(\boldsymbol\Gamma^N)_i = h_N\Gamma_i^N(K^N,Y^N)$.

\begin{Proposition}[Exact embedded quadratic Riccati equations]
\label{prop:exact_embedded_riccati}
Let $\mathscr F^{P,N}(t,P)$ denote the global Riccati driver in \eqref{eq:discrete_Riccati_P} and $(k,\bar k)$ be such that $k\in\mathcal K_N^+$ and $\bar k_{ji}=\bar k_{ij}^{\top}$. Define, for $i\neq j$, $
\mathfrak F_{ij}^{\bar K,N}(t,k,\bar k)
:=
h_N^{-2}
\big[
\mathscr F^{P,N}(t,P^N[k,\bar k])
\big]_{ij}$ and $\mathfrak F_{ii}^{\bar K,N}(t,k,\bar k):=0$. Then the embedded quadratic objects satisfy
\[
\left\{
\begin{aligned}
dK_{i,t}^N
&=
-\Big[
\mathfrak F_i^{K,\mathrm{loc},N}(t,K_t^N)
+
\mathfrak d_{i,t}^{K,N}(K_t^N,\bar K_t^N)
\Big]dt
+
Z_{i,t}^{K,N}\,dW_t^0,
\\
K_{i,T}^N
&=
H_i^N+h_NG_{ii}^{H,N},
\end{aligned}
\right.
\]
and, for \(i\neq j\),
\[
\left\{
\begin{aligned}
d\bar K_{ij,t}^N
&=
-\mathfrak F_{ij}^{\bar K,N}
(t,K_t^N,\bar K_t^N)\,dt
+
Z_{ij,t}^{\bar K,N}\,dW_t^0,
\\
\bar K_{ij,T}^N
&=
G_{ij}^{H,N}.
\end{aligned}
\right.
\]
Here $\mathfrak F_i^{K,\mathrm{loc},N}(t,k)
=
Q_i^N
+
(B_i^N)^\top k_i+k_iB_i^N
+
(E_i^N)^\top k_iE_i^N
-
U_i^N(k)^\top O_i^N(k)^{-1}U_i^N(k)$, and $\mathfrak d^{K,N}$ is the diagonal interaction remainder given explicitly below :

\begin{align}
\mathfrak d_i^{K,N}(k,\bar k)
={}
h_NG_{ii}^{Q,N}
+
h_N\Big[
(G_{ii}^{B,N})^\top k_i
+
k_iG_{ii}^{B,N}
\Big]
\notag
+
h_N^2
\sum_{\ell\neq i}
\Big[
(G_{\ell i}^{B,N})^\top\bar k_{\ell i}
+
\bar k_{i\ell}G_{\ell i}^{B,N}
\Big]
\notag+
(\widehat E_i^N)^\top k_i\widehat E_i^N
-
(E_i^N)^\top k_iE_i^N
\notag\\
+
h_N^2
\sum_{\ell\neq i}
(G_{\ell i}^{E,N})^\top k_\ell G_{\ell i}^{E,N}
-
\Big[
\widehat U_i^N(k)^\top O_i^N(k)^{-1}\widehat U_i^N(k)
-
U_i^N(k)^\top O_i^N(k)^{-1}U_i^N(k)
\Big]
-
h_N^2
\sum_{\ell\neq i}
V_{\ell i}^N(k,\bar k)^\top
O_\ell^N(k)^{-1}
V_{\ell i}^N(k,\bar k).
\label{eq:exact_diagonal_remainder}
\end{align}

\noindent Then the discrete BSRE becames :
\[
\left\{
\begin{aligned}
dK_t^N
&=
-\Big[
\mathfrak F^{K,\mathrm{loc},N}(t,K_t^N)
+
\mathfrak d_t^{K,N}(K_t^N,\bar K_t^N)
\Big]dt
+
Z_t^{K,N}\,dW_t^0,
\\
K_T^N
&=
H^N+h_N\operatorname{diag}_NG^{H,N},
\end{aligned}
\right.
\quad
\left\{
\begin{aligned}
d\bar K_t^N
&=
-\Pi_N^\circ
\mathfrak F^{\bar K,N}
(t,K_t^N,\bar K_t^N)\,dt
+
Z_t^{\bar K,N}\,dW_t^0,
\\
\bar K_T^N
&=
\Pi_N^\circ G^{H,N}.
\end{aligned}
\right.
\]
\end{Proposition}

\begin{proof}
By the block definitions of the global gains, one has $(\mathcal O^N)_{ii} = h_NO_i^N(k)$, $(\mathcal U^N)_{ii} = h_N\widehat U_i^N(k)$, and, for $i\neq j$, $(\mathcal U^N)_{ij} = h_N^2V_{ij}^N(k,\bar k)$. Taking the $(i,i)$-block of the global Riccati equation \eqref{eq:discrete_Riccati_P} and multiplying it by $h_N^{-1}$ gives the local Riccati driver $\mathfrak F_i^{K,\mathrm{loc},N}$. The remaining terms are exactly: the diagonal entries of the interaction costs and dynamics, the off-diagonal column contributions, the correction $\widehat E_i^\top k_i\widehat E_i-E_i^\top k_iE_i$, and the difference between the full and local quadratic feedback terms. Collecting them gives \eqref{eq:exact_diagonal_remainder}. For $i\neq j$, taking the $(i,j)$-block and multiplying by $h_N^{-2}$ gives $\mathfrak F_{ij}^{\bar K,N} = h_N^{-2} \big[ \mathscr F^{P,N}(t,P^N[k,\bar k]) \big]_{ij}$. By definition, the diagonal extension of this driver is zero, and therefore its step-function embedding is $\Pi_N^\circ\mathfrak F^{\bar K,N}$. Finally, the terminal blocks satisfy $P_{ii,T}^N = h_NH_i^N+h_N^2G_{ii}^{H,N}$, $P_{ij,T}^N = h_N^2G_{ij}^{H,N}$, $i\neq j$. Dividing respectively by $h_N$ and $h_N^2$ yields $K_{i,T}^N = H_i^N+h_NG_{ii}^{H,N}$, $\bar K_{ij,T}^N = G_{ij}^{H,N}$. This proves both embedded equations.
\end{proof}

We now state the structural assumption that controls the discrepancy between discrete and continuum coefficients.

\begin{Assumption}[Regularity and discretization consistency]
\label{ass:reg_consistency}
There exist constants $C,L_{\mathrm{spat}}>0$ and sequences $\eta_N\to0$, $\varepsilon_N^{\mathrm{ker}}\to0$ and $\varepsilon_N^\Sigma\to0$ such that the following holds.
\begin{enumerate}
\item[\textup{(A-block)}] \emph{Piecewise Lipschitz regularity.} There exists a finite partition $0 = a_0 < a_1 < \cdots < a_M = 1$, with $M \geq 1$ fixed independently of $N$, such that the coefficients $A, B, C, D, E, F, \Sigma^0, \iota, R, Q, H$ are uniformly bounded, $R$ is uniformly coercive with constant $c_R > 0$ (same for $R^N$), and each coefficient is Lipschitz in the label variable $u$ on every block $J_m := (a_{m-1}, a_m]$ with constant $L_{\mathrm{spat}}$, uniformly in $(s, \omega)$. The kernels $G^B, G^E, G^Q, G^H$ are uniformly bounded and Lipschitz on every rectangle $J_m \times J_n$ with the same constant.
\item[\textup{(A-coeff)}] \emph{Coefficient and kernels consistency.} The discrete local coefficients satisfy, for $\P$-a.e.\ $\omega$ and all $t \in [0,T]$,
\[
\|A_t^N - A_t\|_{L^2(I)} 
+\|B^N_t-B_t\|_{L^2(I)}
+\|C^N_t-C_t\|_{L^2(I)}
+\|D^N_t-D_t\|_{L^2(I)}
+\|E^N_t-E_t\|_{L^2(I)}
+\|F^N_t-F_t\|_{L^2(I)}
\]
\[
+\|\iota^N_t-\iota_t\|_{L^2(I)}
+\|R^N_t-R_t\|_{L^2(I)}
+\|Q^N_t-Q_t\|_{L^2(I)}
+\|H^N-H\|_{L^2(I)}
\le \eta_N+\sqrt{h_N},
\]
the discrete kernels satisfy, for $\P$-a.e.\ $\omega$ and all $t \in [0,T]$,
\[
\|G^{B,N}_t-G^B_t\|_{L^2(I^2)}
+\|G^{E,N}_t-G^E_t\|_{L^2(I^2)}
+\|G^{Q,N}_t-G^Q_t\|_{L^2(I^2)}
+\|G^{H,N}-G^H\|_{L^2(I^2)}
\le \varepsilon_N^{\mathrm{ker}}.
\]
Here $\eta_N$ measures the discrepancy between the discrete coefficients and the sampled continuum coefficients $\Pi_NB,\ldots,\Pi_NH$ such as $(\Pi_N f)(u) := \sum_{i=1}^N f(u_i^N)\mathbf1_{I_i^N}(u)$ and for a kernel $(\Pi_N^{(2)}G)(u,v):= \sum_{i,j} G(u_i^N,u_j^N) \mathbf1_{I_i^N}(u)\mathbf1_{I_j^N}(v)$. The additional term $\sqrt{h_N}$ is the deterministic projection error caused by the finitely many grid cells crossing block interfaces.

Moreover, the discrete local coefficients $A^N, B^N, C^N, D^N, E^N, F^N, \Sigma^{0,N}, \iota^N, R^N, Q^N, H^N$ are uniformly bounded and the discrete kernels $G^{B,N}, G^{E,N}, G^{Q,N}, G^{H,N}$ are uniformly bounded in $L^2(I^2)$, with constants independent of $N$. 

\item[\textup{(A-Sigma)}]
\emph{Mixed-norm consistency of the common-noise coefficient.}
The common-noise coefficients satisfy
\[
\int_I
\operatorname*{ess\,sup}_{(t,\omega)\in[0,T]\times\Omega}
\left|
\Sigma_t^{0,N}(u,\omega)-\Sigma_t^0(u,\omega)
\right|^2du
\leq
\big(\varepsilon_N^\Sigma\big)^2.
\]
The essential supremum is taken with respect to
$dt\otimes d\P$.

\item[\textup{(A-row/col)}] \emph{Continuum row/column regularity} We denote by $\mathcal B_N$ the union of grid cells whose interior contains a block interface:
\[ 
\mathcal B_N := \bigcup \left\{ \operatorname{int}(I_i^N):\ \operatorname{int}(I_i^N)\cap\{a_1,\ldots,a_{M-1}\}\neq\varnothing \right\}.
\]
Since the number of block interfaces is finite, there exists a constant $C_M$, independent of $N$, such that $|\mathcal B_N|\le C_M h_N$.
The limiting interaction Riccati kernel satisfies
\[
\E\left[
\sup_{t\in[0,T]}
\left(
\esssup_{u\in I}\int_I|\bar K_t(u,v)|^2\,dv
+
\esssup_{v\in I}\int_I|\bar K_t(u,v)|^2\,du
\right)
\right]\le C.
\]
Moreover,
\[
\E\left[
\sup_{t\in[0,T]}\|Y_t\|_{L^\infty(I)}^2
\right]\le C.
\]  
\end{enumerate}
\end{Assumption}

\begin{Remark}[Role of the limiting $Y$-regularity]
\label{rem:Y_limit_regularity}
The assumption on $Y$ is an $S^2(L^\infty(I))$-estimate and not a deterministic essential bound in $(t,\omega)$. This is sufficient for the quantitative driver comparison because Lemma~\ref{lem:K_uniform} below yields the stronger deterministic estimate $\|K^N-K\|_{S^\infty_{\F^0}(L^2(I))} \leq C\delta_N^K$, together with analogous estimates for the local feedback gains. Consequently, products between $Y$ and the local backward differences can be estimated without any independence assumption or localization.
\end{Remark}

\begin{Remark}
\label{rem:assumption_reg}

\textup{(i)} Assumption \ref{ass:reg_consistency} \textup{(A-block)} covers the practically important case of non-exchangeable systems with finitely many sub-populations: within each block $J_m$, the coefficients vary smoothly, while across blocks they may jump. Global Lipschitz regularity in $u$ is recovered as the special case $M = 1$.

\textup{(ii)} Since the grid $\{u_i^N\}$ is uniform and the partition points $\{a_k\}$ need not coincide with grid points, at most $M-1$ cells $I_i^N$ may straddle two adjacent blocks. On these straddling cells the local Lipschitz argument fails and only an $O(1)$ pointwise bound is available. As we shall see, this loss leads to an $L^2(I)$-projection error of order $\sqrt{h_N}$ rather than $h_N$.

\textup{(iii)} Assumption~\ref{ass:reg_consistency} \textup{(A-row/col)} imposes additional spatial estimates on the limiting backward solution. These estimates are stronger than its natural Hilbert-space bounds and are used to control products of local coefficient errors with $\bar K$ and $Y$ in the driver comparison. They are not inferred here from \textup{(A-block)} in general. In finite-type models, they follow directly from the finite-dimensional block structure and the $S^2$ estimates.

\textup{(iv)} The uniform boundedness of the limiting kernels in 
\textup{(A-block)} is a convenient, non-minimal sufficient condition: 
it is natural in graphon-based models \cite{CainesHuang2021GraphonGames} 
and consistent with \cite{DeFeoMekkaoui2025,DeCrescenzoDeFeoPham2026}. The proof uses it to bound products of local differences with limiting kernels and the mass of the limiting kernels on the diagonal band, $\|\mathbf 1_{\mathcal D_N}G\|_{L^2(I^2)}\le \|G\|_{L^{\infty}(I^2)}\sqrt{h_N}$, used in Appendix~\ref{app:R1}, which row/column bounds alone would not give. By contrast, the vanishing of $d_N^{\mathrm{diag}}$ only uses the natural $S^2\times H^2$ regularity of $(\bar K,Z^{\bar K})$; see Lemma~\ref{lem:diagonal_band_modulus}.
\end{Remark}

\begin{Remark}[Finite-type models]\label{rem:finite_type_models}
Suppose that all local coefficients and kernels are constant on a fixed finite block partition, that the grids are block-compatible, and that the discrete data are obtained by exact sampling. Then $\eta_N=\varepsilon_N^{\rm ker}=\varepsilon_N^\Sigma=0$ and there is no interface error. By uniqueness, $\bar K$ and $Z^{\bar K}$ are constant in the label variables on each of the finitely many block rectangles. The corresponding finite family of $S^2\times H^2$ estimates gives $\E\sup_{t\le T}\|\bar K_t\|_{L^\infty(I^2)}^2 +\E\int_0^T\|Z_t^{\bar K}\|_{L^\infty(I^2)}^2dt\le C$. The same block structure holds for $Y$. Thus \textup{(A-row/col)} follows from the Hilbert-space $S^2$ bounds in this class. Since $|\mathcal D_N|=h_N$, it follows that $d_N^{\rm diag}\le C\sqrt{h_N}$. Consequently $r_N\le C\sqrt{h_N}$, and the non-squared estimates of Theorem~\ref{thm:convergence} have order $N^{-1/2}$.
\end{Remark}

\begin{Lemma}[Vanishing of the diagonal-band modulus]
\label{lem:diagonal_band_modulus}
We define
\[
\big(d_N^{\mathrm{diag}}\big)^2
:={}
\E\left[
\sup_{t\in[0,T]}
\|
\mathbf 1_{\mathcal D_N}\bar K_t
\|_{L^2(I^2)}^2
\right]
+
\E\int_0^T
\|
\mathbf 1_{\mathcal D_N}Z_t^{\bar K}
\|_{L^2(I^2)}^2dt
\]
Under the solvability assumptions of Theorem~\ref{thm:solvability_continuum_backward_system}, $d_N^{\mathrm{diag}}\longrightarrow0$.
\end{Lemma}

\begin{proof}
Let $P_N$ denote multiplication by $\mathbf 1_{\mathcal D_N}$ on $L^2(I^2)$. Since
$|\mathcal D_N|=h_N\to0$, the operators $P_N$ are contractions and converge strongly to zero on $L^2(I^2)$. For $\P$-a.e.\ $\omega$, the map $t\mapsto\bar K_t(\omega)$ is continuous in $L^2(I^2)$. Hence its image $\{\bar K_t(\omega):t\in[0,T]\}$ is compact in $L^2(I^2)$. Since a uniformly bounded sequence of operators converging strongly to zero converges uniformly on compact sets, $\sup_{t\in[0,T]} \|P_N\bar K_t(\omega)\|_{L^2(I^2)} \longrightarrow0$. Moreover, $\sup_t\|P_N\bar K_t\|_{L^2(I^2)}^2 \leq \sup_t\|\bar K_t\|_{L^2(I^2)}^2$, whose expectation is finite. Dominated convergence therefore gives the first convergence. For the martingale integrand, $\mathbf 1_{\mathcal D_N}(u,v) |Z_t^{\bar K}(u,v)|^2 \longrightarrow0$ for $dt\otimes du\otimes dv\otimes d\P$-a.e.\ $(t,u,v,\omega)$, and it is dominated by $|Z_t^{\bar K}(u,v)|^2$. Another application of dominated convergence gives $\E\int_0^T \| P_NZ_t^{\bar K} \|_{L^2(I^2)}^2dt \longrightarrow0$.
\end{proof}

We aggregate the discretization errors into a single rate that will appear throughout the section.

\begin{Definition}[Aggregate error rate]
\label{def:rate}
We set
\[
\delta_N^K := \eta_N + \sqrt{h_N},
\qquad
r_N := \delta_N^K + \varepsilon_N^{\mathrm{ker}} + \varepsilon_N^\Sigma + d_N^{\mathrm{diag}}.
\]
\end{Definition}

\begin{Remark}
\label{rem:rate_block_compatible}
Under Assumption~\ref{ass:reg_consistency} and Lemma~\ref{lem:diagonal_band_modulus}, one has $r_N\to0$. If the partition points $\{a_k\}$ are grid points, the cells crossing the block interfaces disappear. Hence the interface contribution to the local coefficient projection improves from $\sqrt{h_N}$ to $h_N$. This does not automatically improve the complete rate $r_N$, because the exact diagonal remainder and the diagonal-band modulus are independent of the block partition. Accordingly, block compatibility alone does not imply an $O(h_N)$ rate for the complete backward system in the present $L^2(I^2)$ topology. Such an improvement additionally requires algebraic control of the diagonal contributions.
\end{Remark}

\subsection{A priori bounds}
\label{subsec:apriori}

We now establish a set of uniform-in-$N$ bounds on the discrete backward objects. The strategy is triangular: we first isolate the homogeneous component of the value function, then derive a pointwise bound on the local Riccati component $K^N$ via a value-function representation, then operator and $L^2$-bounds on the interaction component $\bar K^N$, and finally the $L^2$-bounds on the linear and constant terms $Y^N, q^N$. We close the subsection with an \(L^2\)-in-label convergence estimate for the local Riccati component \(K^N\), which will be the key technical input in §\ref{subsec:stability}.
\medskip

We begin with a structural observation that decouples the quadratic part of the discrete value function from the affine and constant terms.

\begin{Lemma}[Homogeneous decoupling]
\label{lem:decoupling}
Let $P^N$ denote the quadratic component of the discrete value function $V^N$ obtained in Theorem~\ref{thm:verification_discrete}. Then $P^N$ depends only on the coefficients $B^N, C^N, E^N, F^N, R^N, Q^N, H^N, G^{B,N}, G^{E,N}, G^{Q,N}, G^{H,N}$ and is independent of the affine coefficients $A^N, D^N, \Sigma^{0,N}, \iota^N$. The value function is then $V_t^{N,0}(\bx^N)=(\bx^N)^\top P_t^N\bx^N$
\end{Lemma}

\begin{proof}
Immediate as the BSRE verified by $P_t^N$ is independant of $A^N, D^N, \Sigma^{0,N}, \iota^N$ and set all these coefficients to $0$.
\end{proof}

\noindent The decoupling reduces the study of $K^N$ and $\bar K^N$ to the homogeneous LQ problem, which we exploit in the following pointwise bound.

\begin{Lemma}[Pointwise uniform bound on $K^N$]
\label{lem:K_pointwise}
Under Assumption~\ref{ass:reg_consistency}, there exists a constant $C>0$, independent of $N$, such that
\[
\sup_{N \geq 1} \, \sup_{1 \leq i \leq N} \, \sup_{t \in [0,T]} \, \big\|K_{i,t}^N\big\| \leq C, \qquad \P\text{-a.s.}
\]
\end{Lemma}

\begin{proof}
By Lemma~\ref{lem:decoupling}, $P^N$ is the quadratic component of the homogeneous discrete value function. Fix $i\in\llbracket 1,n \rrbracket$ and $z\in\mathbb R^d$, and consider the initial state $\bx^N=(0,\ldots,0,z,0,\ldots,0)$, where the only non-zero component is the $i$-th one. Then $|\bx^N|_N^2=h_N|z|^2$ and, by the block decomposition of $P^N$, $V_t^{N,0}(\bx^N) = z^\top P_{ii,t}^Nz = h_Nz^\top K_{i,t}^Nz$. The non-negativity of the homogeneous cost and the admissibility of the zero control yield $0 \leq V_t^{N,0}(\bx^N) \leq J_t^{N,0}(\bx^N,0)$. By the standard conditional estimate for the homogeneous state equation and the uniform bounds on the local coefficients and the discrete kernel operators, $J_t^{N,0} (\bx^N,0) \leq C|\bx^N|_N^2 = Ch_N|z|^2$, $\P\text{-a.s.}$. Consequently, $0 \leq z^\top K_{i,t}^Nz \leq C|z|^2$. Since $K_{i,t}^N$ is symmetric, taking the supremum over $|z|=1$ gives the result uniformly in $N,i$ and $t$.
\end{proof}

\begin{Remark}
\label{rem:K_continuum_bound}
\textup{(i)} The analogous pointwise bound holds for the continuum local Riccati
component. By Remark~\ref{rem:K_bounded_continuum}, up to choosing the
bounded version of the solution, there exists $C>0$ such that $\sup_{t\in[0,T]}\esssup_{u\in I}\|K_t(u)\|\le C$, $\P\text{-a.s.}$

\textup{(ii)} Throughout this subsection we make repeated use of the following standard estimate. Let $E$ be a separable Hilbert space, and let $(X,Z)\in S^2_{\mathbb F^0}([0,T];E) \times H^2_{\mathbb F^0} ([0,T];\mathcal L_2(\mathbb R^{d_0};E))$ solve the backward equation $X_t=\xi+\int_t^T F(s,X_s)\,ds-\int_t^T Z_s\,dW_s^0$, with terminal condition $\xi\in L^2(\Omega;E)$ and a driver $F$ satisfying the affine growth condition $\|F(t,x)\|_E\leq\kappa(1+\|x\|_E)$ for some $\kappa > 0$. Then there exists a constant $C = C(T, \kappa) > 0$ such that
\[
\mathbb E\left[\sup_{t\in[0,T]}\|X_t\|_E^2 +\int_0^T\|Z_t\|_{\mathcal L_2(\mathbb R^{d_0};E)}^2\,dt\right] \leq C\big(1+\mathbb E\|\xi\|_E^2\big)
\]
The proof is by Itô's formula on $\|X_t\|^2_E$ combined with the BDG inequality and a backward Grönwall argument; see, e.g., \cite{FuhrmanTessitore2002,ElKarouiPengQuenez1997} for details.
\end{Remark}

We now turn to the interaction Riccati component. The pointwise approach used for $K^N$ is no longer available: $\bar K^N$ enters the system through an interaction kernel, and its natural bounds are of operator and $L^2$ type.

\begin{Lemma}[Uniform estimates on the interaction Riccati components]
\label{lem:Kbar_bound}
Under Assumption~\ref{ass:reg_consistency}, there exists a deterministic
constant $C>0$, independent of $N$, such that:

\begin{enumerate}
\item[\textup{(a)}] \emph{Operator bounds.}
\[
\sup_{t\in[0,T]}
\|T_{\bar K_t}\|_{\mathcal L(\mathsf H)}
+
\sup_{N\geq1}\sup_{t\in[0,T]}
\|T_{\bar K_t^N}\|_{\mathcal L(\mathsf H)}
\leq C,
\qquad
\P\text{-a.s.}
\]

\item[\textup{(b)}] \emph{Pathwise Hilbert--Schmidt bounds.}
\[
\|\bar K\|_{S^\infty_{\F^0}(\overline{\mathsf H})}
+
\sup_{N\geq1}
\|\bar K^N\|_{S^\infty_{\F^0}(\overline{\mathsf H})}
\leq C.
\]

\item[\textup{(c)}] \emph{BMO bounds.}
\[
\|Z^{\bar K}\|_{\mathrm{BMO}(\overline{\mathsf H})}
+
\sup_{N\geq1}
\|Z^{\bar K,N}\|_{\mathrm{BMO}(\overline{\mathsf H})}
\leq C.
\]
\end{enumerate}
\end{Lemma}

\begin{proof}
\textit{item(a) :} The continuum operator bound was obtained in Theorem~\ref{thm:solvability_continuum_backward_system}. For the discrete problem, let $\mathsf P_t^N$ denote the operator on the space of step functions corresponding to the rescaled homogeneous quadratic value. For every step function $x$, $\langle x,\mathsf P_t^Nx\rangle_{\mathsf H} = V_t^{N,0}(x)$. The zero-control estimate gives $0\leq \langle x,\mathsf P_t^Nx\rangle_{\mathsf H} \leq C\|x\|_H^2$. Since $\mathsf P_t^N$ is self-adjoint and non-negative, $\|\mathsf P_t^N\|_{\mathcal L(\mathsf H)} \leq C$ on the step-function subspace. The exact block decomposition gives, on that subspace,
\[
\mathsf P_t^N
=
M_{K_t^N}
+
T_{\bar K_t^N}.
\]
By Lemma~\ref{lem:K_pointwise}, $\|M_{K_t^N}\|_{\mathcal L(\mathsf H)} \leq \|K_t^N\|_{L^\infty(I)} \leq C$. Therefore, $\|T_{\bar K_t^N}\|_{\mathcal L(\mathsf H)}
\leq C$. For a general $f\in \mathsf H$, the step kernel satisfies $T_{\bar K_t^N}f = T_{\bar K_t^N}\mathsf P_Nf$, where $\mathsf P_N$ is the orthogonal cell-average projection ont the space of step functions. Hence the same operator bound holds on all of $\mathsf H$.

\medskip
\textit{item(b) :} Along the continuum and discrete solutions, the interaction drivers satisfy $\|\mathcal F_s(\bar K_s)\|_{\overline{\mathsf H}} \leq C\big(1+\|\bar K_s\|_{\overline{\mathsf H}}\big)$, and $\left\| \Pi_N^\circ \mathfrak F_s^{\bar K,N}(K_s^N,\bar K_s^N) \right\|_{\overline{\mathsf H}} \leq C\big(1+\|\bar K_s^N\|_{\overline{\mathsf H}}\big)$. Indeed, all affine terms are controlled by the bounded local coefficients and kernel operators. For the quadratic feedback term,
\[
\|T_V^\ast M_{O^{-1}}T_V\|_{\mathrm{HS}}
\leq
\|T_V\|_{\mathrm{op}}
\|M_{O^{-1}}\|_{\mathrm{op}}
\|T_V\|_{\mathrm{HS}}.
\]
The first two factors are uniformly bounded by item~(a) and coercivity, while $\|T_V\|_{\mathrm{HS}} \leq C\big(1+\|\bar K\|_{\overline{\mathsf H}}\big)$. For the discrete driver, the corresponding full kernel is $\widetilde V_s^N(u,v) :=(C_s^N(u))^\top\bar K_s^N(u,v) +(F_s^N(u))^\top K_s^N(u)G_s^{E,N}(u,v)$. Its operator norm is uniformly bounded by item (a), the local coefficient bounds and the operator bound on $G^{E,N}$, and
\[
\|\widetilde V_s^N\|_{L^2(I^2)}
\le C\big(\|\bar K_s^N\|_{\overline{\mathsf H}}
+\|G_s^{E,N}\|_{L^2(I^2)}\big).
\]
Applying the same Hilbert--Schmidt ideal inequality to $T_{\widetilde V_s^N}^{*}M_{(O_s^N)^{-1}}T_{\widetilde V_s^N}$, and using that $\Pi_N^\circ$ is a contraction in $L^2(I^2)$, gives the stated discrete driver bound. 

\medskip
\textit{item(c):}
Let $(\bar K^\bullet,Z^{\bar K,\bullet})$ denote either the continuum solution or one of the discrete solutions. Applying It\^o's formula to $\|\bar K_s^\bullet\|_{\overline{\mathsf H}}^2$ between an arbitrary stopping time $\tau$ and $T$, followed by conditional expectation and Young's inequality, gives
\[
\begin{aligned}
\|\bar K_\tau^\bullet\|_{\overline{\mathsf H}}^2
&+
\frac12
\E\left[
\int_\tau^T
\|Z_s^{\bar K,\bullet}\|_{\mathcal L_2(\mathbb R^{d_0};\overline{\mathsf H})}^2\,ds
\;\middle|\;
\Fc_\tau^0
\right]
\leq
C
+
C\E\left[
\int_\tau^T
\|\bar K_s^\bullet\|_{\overline{\mathsf H}}^2\,ds
\;\middle|\;
\Fc_\tau^0
\right].
\end{aligned}
\]
The terminal conditions are uniformly bounded in $\overline{\mathsf H}$ under Assumption~\ref{ass:reg_consistency}. The standard conditional backward Gr\"onwall estimate therefore yields $\operatorname*{ess\,sup}_{(\omega,t)} \|\bar K_t^\bullet(\omega)\|_{\overline{\mathsf H}} \leq C$, uniformly in $N$. Substituting this estimate back into the preceding conditional inequality gives $\sup_{\tau\in\mathcal T_{0,T}} \left\| \E\left[ \int_\tau^T \|Z_s^{\bar K,\bullet}\|_{\mathcal L_2(\mathbb R^{d_0};\overline{\mathsf H})}^2\,ds \;\middle|\; \Fc_\tau^0 \right] \right\|_{L^\infty} \leq C$. This proves all the assertions.
\end{proof}

\begin{Lemma}[Bounds on the exact diagonal remainder]
\label{lem:diagonal_remainder_bounds}
Under Assumption~\ref{ass:reg_consistency}, there exists a deterministic
constant $C>0$, independent of $N$, such that
\[
\operatorname*{ess\,sup}_{(t,\omega)}
\|\mathfrak d_t^{K,N}\|_{L^\infty(I)}
\leq C,
\quad
\operatorname*{ess\,sup}_{(t,\omega)}
\|\mathfrak d_t^{K,N}\|_{L^2(I)}^2
\leq
Ch_N.
\]
Moreover,
\[
\left\|
h_N\operatorname{diag}_NG^{H,N}
\right\|_{L^2(I)}^2
\leq
Ch_N,
\qquad
\P\text{-a.s.}
\]
\end{Lemma}

\begin{proof}
For a typical bilinear column term, $\left| h_N^2\sum_{\ell\neq i} (G_{\ell i}^{B,N})^\top\bar K_{\ell i}^N \right| \leq \sqrt{ a_i^N(G^{B,N}) a_i^N(\bar K^N) }$ with $a_i$ the operator defined in \ref{sec:notations} . It is therefore uniformly bounded pointwise in $i$. Moreover,
\[
\begin{aligned}
&h_N\sum_{i=1}^N
\left|
h_N^2\sum_{\ell\neq i}
(G_{\ell i}^{B,N})^\top\bar K_{\ell i}^N
\right|^2
\leq
h_N
\sum_{i=1}^N
a_i^N(G^{B,N})a_i^N(\bar K^N)
\leq
h_N
\|G^{B,N}\|_{L^2(I^2)}^2
\|\bar K^N\|_{L^2(I^2)}^2.
\end{aligned}
\]
For the quadratic feedback column term, $\left| h_N^2\sum_{\ell\neq i} (V_{\ell i}^N)^\top (O_\ell^N)^{-1} V_{\ell i}^N \right| \leq C a_i^N(V^N)$, and $h_N\sum_{i=1}^N|a_i^N(V^N)|^2 \leq h_N \left( \sum_{i=1}^Na_i^N(V^N) \right)^2$\\ 
\noindent $= h_N\|V^N\|_{L^2(I^2)}^4$. Lemma~\ref{lem:Kbar_bound} and the definition of $V^N$ imply $\sup_{N,t} \|V_t^N\|_{L^2(I^2)} \leq C$,$\P\text{-a.s.}$. The terms containing diagonal kernel entries are controlled by $h_N^3\sum_{i=1}^N|G_{ii}^N|^2 \leq h_N\|G^N\|_{L^2(I^2)}^2$, and, for the corresponding quadratic terms, $h_N\sum_{i=1}^N \big(h_N^2|G_{ii}^N|^2\big)^2 \leq h_N\|G^N\|_{L^2(I^2)}^4$. Applying these estimates to every term of \eqref{eq:exact_diagonal_remainder} proves $\|\mathfrak d_t^{K,N}\|_{L^\infty(I)} \leq C$, and $\|\mathfrak d_t^{K,N}\|_{L^2(I)}^2 \leq Ch_N$. Finally, $\left\| h_N\operatorname{diag}_NG^{H,N} \right\|_{L^2(I)}^2 =
h_N^3\sum_{i=1}^N|G_{ii}^{H,N}|^2 \leq h_N\|G^{H,N}\|_{L^2(I^2)}^2 \leq Ch_N$.
\end{proof}

\begin{Lemma}[BMO estimates for the local Riccati martingales]
\label{lem:ZK_BMO}
Under Assumption~\ref{ass:reg_consistency}, there exists a deterministic constant $C>0$, independent of $N$, such that
\[
\operatorname*{ess\,sup}_{u\in I}
\|Z^K(u)\|_{\mathrm{BMO}}^2
+
\sup_{N\geq1}
\max_{1\leq i\leq N}
\|Z_i^{K,N}\|_{\mathrm{BMO}}^2
\leq C.
\]
Consequently,
\[
\|Z^K\|_{\mathrm{BMO}(L^2(I))}
+
\sup_{N\geq1}
\|Z^{K,N}\|_{\mathrm{BMO}(L^2(I))}
\leq C.
\]
\end{Lemma}

\begin{proof}
For a.e.$u\in I$, the continuum process $K(u)$ is uniformly bounded and its local Riccati driver is uniformly bounded. Applying It\^o's formula to $|K_s(u)|^2$ between an arbitrary stopping time $\tau$ and $T$, then taking conditional expectation, gives $\E\left[\int_\tau^T|Z_s^K(u)|^2\,ds \;\middle|\; \Fc_\tau^0\right] \leq C$, uniformly in $u$. For the discrete equation, $K_i^N$ is uniformly bounded by Lemma~\ref{lem:K_pointwise}. Its driver is the sum of the uniformly 
bounded local Riccati driver and $\mathfrak d_i^{K,N}$, which is uniformly bounded by Lemma~\ref{lem:diagonal_remainder_bounds}. The same conditional It\^o estimate yields $\E\left[ \int_\tau^T|Z_{i,s}^{K,N}|^2\,ds \;\middle|\; \Fc_\tau^0 \right] \leq C$, uniformly in $N$ and $i$. The aggregate estimates follow by integrating with respect to $u$, or equivalently by multiplying the discrete estimates by $h_N$ and summing over $i$.
\end{proof}

The next lemma collects the analogous $L^2$-bounds on the linear term $Y^N$ and the constant term $q^N$.

\begin{Lemma}[Uniform bounds on the affine and scalar components]
\label{lem:remaining_bounds}
Under Assumption~\ref{ass:reg_consistency}, there exists a deter-ministic constant $C>0$, independent of $N$, such that
\[
\|Y\|_{S^\infty_{\F^0}(\mathsf H)}
+
\sup_{N\geq1}
\|Y^N\|_{S^\infty_{\F^0}(\mathsf H)}
\leq C,
\quad
\|Z^Y\|_{\mathrm{BMO}(\mathsf H)}
+
\sup_{N\geq1}
\|Z^{Y,N}\|_{\mathrm{BMO}(\mathsf H)}
\leq C.
\]
Moreover,
\[
\|\lambda\|_{S^\infty_{\F^0}(\mathbb R)}
+
\sup_{N\geq1}
\|q^N\|_{S^\infty_{\F^0}(\mathbb R)}
\leq C,
\quad
\|Z^\lambda\|_{\mathrm{BMO}(\mathbb R)}
+
\sup_{N\geq1}
\|Z^{q,N}\|_{\mathrm{BMO}(\mathbb R)}
\leq C.
\]
\end{Lemma}

\begin{proof}
The continuum and discrete linear equations can be written as $dY_s^\bullet = -\big(\mathcal A_s^\bullet Y_s^\bullet + f_s^\bullet \big)\,ds + Z_s^{Y,\bullet}\,dW_s^0$, $Y_T^\bullet=0$, where $\bullet$ denotes either the continuum system or a finite-$N$ system. The operator bounds on $K^\bullet$, $T_{\bar K^\bullet}$, $O^{-1,\bullet}$ and $T_{V^\bullet}$ imply $\|\mathcal A_s^\bullet\|_{\mathcal L(\mathsf H)} \leq C$. Moreover, the boundedness of the affine coefficients under Assumption~\ref{ass:reg_consistency}, together with the definitions of $M$ and $\Gamma$, gives $\|f_s^\bullet\|_{\mathsf H} \le C\left(1+ \|Z_s^{K,\bullet}\|_{\mathcal L_2(\mathbb R^{d_0};L^2(I;\mathbb S^d))} + \|Z_s^{\bar K,\bullet}\|_{\mathcal L_2(\mathbb R^{d_0};\overline{\mathsf H})} \right)$. Lemmas~\ref{lem:Kbar_bound} and \ref{lem:ZK_BMO} therefore imply $\sup_{\tau\in\mathcal T_{0,T}} \left\| \E\left[ \int_\tau^T \|f_s^\bullet\|_{\mathsf H}^2\,ds \;\middle|\; \Fc_\tau^0 \right] \right\|_{L^\infty} \leq C$. The standard conditional estimate for a linear Hilbert-space BSDE with uniformly bounded operator coefficient yields
\[
\begin{aligned}
\|Y_\tau^\bullet\|_\mathsf H
&\leq
C
\E\left[
\int_\tau^T
\|f_s^\bullet\|_\mathsf H\,ds
\;\middle|\;
\Fc_\tau^0
\right]
\leq
C\sqrt{T}
\left(
\E\left[
\int_\tau^T
\|f_s^\bullet\|_{\mathsf H}^2\,ds
\;\middle|\;
\Fc_\tau^0
\right]
\right)^{1/2}
\leq C.
\end{aligned}
\]
Thus, $\|Y^\bullet\|_{S^\infty(\mathsf H)} \leq C$. Applying It\^o's formula to $\|Y_s^\bullet\|_{\mathsf H}^2$ between $\tau$ and $T$, and using the preceding $S^\infty(\mathsf H)$-bound, gives $\E\left[ \int_\tau^T \|Z_s^{Y,\bullet}\|_{\mathcal L_2(\mathbb R^{d_0};\mathsf H)}^2\,ds \;\middle|\; \Fc_\tau^0 \right] \leq C$. This proves the BMO estimates for $Z^Y$ and $Z^{Y,N}$. Let $\ell$ and $\ell^N$ denote the scalar drivers of $\lambda$ and $q^N$. The exact formulas, together with the estimates already proved, give $|\ell_s|^2 \leq C\big(1+\|Z_s^Y\|_{\mathcal L_2(\mathbb R^{d_0};\mathsf H)}^2\big)$ and $|\ell_s^N|^2\leq C\big(1+\|Z_s^{Y,N}\|_{\mathcal L_2(\mathbb R^{d_0};\mathsf H)}^2\big)$. Hence $\sup_{\tau} \left\| \E\left[ \int_\tau^T |\ell_s^\bullet|^2\,ds \;\middle|\; \Fc_\tau^0 \right] \right\|_{L^\infty} \leq C$. Using $\lambda_t = \E\left[ \int_t^T\ell_s\,ds \;\middle|\; \Fc_t^0 \right]$ and $q_t^N = \E\left[ \int_t^T\ell_s^N\,ds \;\middle|\; \Fc_t^0 \right]$, conditional Cauchy--Schwarz gives $\|\lambda\|_{S^\infty} + \sup_N\|q^N\|_{S^\infty} \leq C$. Finally, applying It\^o's formula to $|\lambda|^2$ and $|q^N|^2$ gives the BMO estimates for their martingale integrands.
\end{proof}

\begin{Lemma}[Strong local Riccati consistency]
\label{lem:K_uniform}
Under Assumption~\ref{ass:reg_consistency}, there exists a deterministic constant $C>0$, independent of $N$, such that
\[
\|K^N-K\|_{S^\infty_{\F^0}(L^2(I))}
+
\|Z^{K,N}-Z^K\|_{\mathrm{BMO}(L^2(I))}
\leq
C\delta_N^K.
\]
In particular,
\[
\E\left[
\sup_{t\in[0,T]}
\|K_t^N-K_t\|_{L^2(I)}^2
+
\int_0^T
\|Z_t^{K,N}-Z_t^K\|_{L^2(I)}^2\,dt
\right]
\leq
C(\delta_N^K)^2.
\]

Moreover,
\[
\begin{aligned}
&
\operatorname*{ess\,sup}_{(s,\omega)}\|(O^N_s)^{-1}-O_s^{-1}\|_{(L^2(I))}
+
\operatorname*{ess\,sup}_{(s,\omega)}\|U^N_s-U_s\|_{(L^2(I))}
+
\operatorname*{ess\,sup}_{(s,\omega)}\|\widehat U^N_s-U_s\|_{(L^2(I))}
\leq
C\delta_N^K.
\end{aligned}
\]
\end{Lemma}

\begin{proof}
Set $\Delta K^N:=K^N-K$,and $\Delta Z^{K,N}:=Z^{K,N}-Z^K$. 

\noindent The exact embedded equation gives $dK_t^N = -\Big[ \mathfrak F^{K,\mathrm{loc},N}(t,K_t^N) + \mathfrak d_t^{K,N} \Big]dt + Z_t^{K,N}\,dW_t^0$, with $K_T^N = H^N + h_N\operatorname{diag}_NG^{H,N}$. On the deterministic bounded ball containing $K^N$ and $K$, the local Riccati driver is uniformly Lipschitz. By the standard mean-value linearization, there exists an $\mathbb F^0$-progressively measurable
bounded linear operator $\mathcal L_t^N$ on $L^2(I;\mathbb S^d)$ such that $\mathfrak F^{K,\mathrm{loc},N}(t,K_t^N) - \mathfrak F^{K,\mathrm{loc},N}(t,K_t) = \mathcal L_t^N\Delta K_t^N$, and $\|\mathcal L_t^N\|_{\mathcal L(L^2(I))} \leq C$.\\ 
Define the coefficient-consistency remainder by $\rho_t^{K,N} := \mathfrak F^{K,\mathrm{loc},N}(t,K_t) - \mathcal F^K(t,K_t)$. The pathwise coefficient-consistency assumptions and the boundedness of $K$ yield $\operatorname {ess\,sup}_{(t,\omega)} \|\rho_t^{K,N}\|_{L^2(I)} \leq C\delta_N^K$. Lemma~\ref{lem:diagonal_remainder_bounds} gives $\operatorname {ess\,sup}_{(t,\omega)} \|\mathfrak d_t^{K,N}\|_{L^2(I)} \leq C\sqrt{h_N} \leq C\delta_N^K$. Moreover, $\|\Delta K_T^N\|_{L^2(I)} \leq \|H^N-H\|_{L^2(I)} + \left\| h_N\operatorname{diag}_NG^{H,N} \right\|_{L^2(I)} \leq C\delta_N^K$, $\P\text{-a.s.}$. 

\noindent Consequently, $(\Delta K^N,\Delta Z^{K,N})$ satisfies the linear
Hilbert-valued BSDE $d\Delta K_t^N = -\Big[ \mathcal L_t^N\Delta K_t^N + \rho_t^{K,N} + \mathfrak d_t^{K,N} \Big]dt + \Delta Z_t^{K,N}\,dW_t^0$, with a terminal condition bounded by $C\delta_N^K$. The standard conditional estimate for linear BSDEs with bounded operator coefficient therefore gives
\[
\begin{aligned}
\|\Delta K_t^N\|_{L^2(I)}
&\leq
C
\E\left[
\|\Delta K_T^N\|_{L^2(I)}
+
\int_t^T
\big(
\|\rho_s^{K,N}\|_{L^2(I)}
+
\|\mathfrak d_s^{K,N}\|_{L^2(I)}
\big)ds
\;\middle|\;
\Fc_t^0
\right]
\leq
C\delta_N^K.
\end{aligned}
\]

\noindent The right-hand side is deterministic. Hence $\|K^N-K\|_{S^\infty(L^2(I))} \leq C\delta_N^K$. Applying It\^o's formula to $\|\Delta K_t^N\|_{L^2(I)}^2$ between an arbitrary stopping time $\tau$ and $T$, and using the preceding bound, yields $\E\left[ \int_\tau^T \|\Delta Z_s^{K,N}\|_{L^2(I)}^2\,ds \;\middle|\; \Fc_\tau^0\right] \leq C(\delta_N^K)^2$. Thus $\|Z^{K,N}-Z^K\|_{\mathrm{BMO (L^2(I))}} \leq C\delta_N^K$.

\noindent For the gains, the definitions and the boundedness of the coefficients gives 
\[
\|O_t^N-O_t\|_{L^2(I)}
\leq
C\left(
\|R_t^N-R_t\|_{L^2(I)}
+
\|F_t^N-F_t\|_{L^2(I)}
+
\|K_t^N-K_t\|_{L^2(I)}
\right).
\]
Hence $\operatorname*{ess\,sup}_{(s,\omega)}\|O^N_s-O_s\|_{L^2(I)} \leq C\delta_N^K$. The resolvent identity and uniform coercivity give the same estimate for the inverses. The estimate for $U^N-U$ follows by expanding $(C^N)^\top K^N-C^\top K$ and $(F^N)^\top K^NE^N-F^\top KE$. 

\noindent Finally, $\widehat U_i^N-U_i^N = h_N(F_i^N)^\top K_i^NG_{ii}^{E,N}$, and $\|\widehat U^N-U^N\|_{L^2(I)}^2 \leq Ch_N$. Combining these estimates proves the result.
\end{proof}

\begin{Remark}[Why $S^2(L^\infty)$ is sufficient for $Y$]
\label{rem:Y_product_closure}
The deterministic $S^\infty(L^2)$ estimate of Lemma~\ref{lem:K_uniform} removes the potential correlation problem between the limiting process $Y$ and the local Riccati differences. Indeed,
\[
\begin{aligned}
\E\int_0^T
\|(K_t^N-K_t)Y_t\|_{L^2(I)}^2dt
&\leq
\|K^N-K\|_{S^\infty(L^2(I))}^2
\E\int_0^T\|Y_t\|_{L^\infty(I)}^2dt
\leq
C(\delta_N^K)^2.
\end{aligned}
\]
The same argument applies to the differences of $U$ and $O^{-1}$. On the other hand, terms in which an interaction-kernel difference acts on $Y$ are controlled using the deterministic $S^\infty(L^2(I))$-bound on $Y$ from Lemma~\ref{lem:remaining_bounds}. Thus no deterministic $L^\infty(I)$-bound on $Y$, no higher probabilistic moment, and no localization argument are required.
\end{Remark}

\begin{Remark}
\label{rem:K_uniform_block_compatible}
On a block-compatible grid, the interface part of the local coefficient error improves from $\sqrt{h_N}$ to $h_N$. However, the exact diagonal remainder remains present. Hence the complete estimate of Lemma~\ref{lem:K_uniform} does not automatically improve to $O(h_N)$ without an additional non-concentration estimate on the discrete diagonal contributions.
\end{Remark}

\subsection{Stability of Hilbert-valued BSREs and convergence of the Riccati systems}
\label{subsec:stability}

We now bring together the a priori bounds of Subsection~\ref{subsec:apriori} and prove the main convergence theorem. The argument rests on two ingredients. First, an abstract stability estimate for Hilbert-valued backward equations reduces convergence to a direct estimate of the driver difference along the two solutions.
Second, Lemma~\ref{lem:R1} provides such estimates for the exact embedded discrete drivers, including the diagonal correction in the local Riccati equation and the off-diagonal projection in the interaction Riccati equation. The triangular structure of the backward system then yields convergence successively for $K$, $\bar K$, $Y$, the scalar component and yields the main convergence Theorem~\ref{thm:convergence}.

\medskip

\noindent We start with the abstract stability lemma.

\begin{Lemma}[Stability of Hilbert-valued BSREs along solutions]
\label{lem:stability}
For each $N \geq 1$, let $E$ be a separable Hilbert space and  $(X^N, Z^N), (X, Z) \in S^2_{\F^0}([0,T]; E) \times H^2_{\mathbb F^0}([0,T]; \mathcal L_2(\R^{d_0}; E))$ satisfy the backward equations
\[
X_t^N = \xi^N + \int_t^T F^N(s, X_s^N)\,ds - \int_t^T Z_s^N\,dW^0_s,
\qquad
X_t = \xi + \int_t^T F(s, X_s)\,ds - \int_t^T Z_s\,dW^0_s.
\]
Assume that there exists a constant $L > 0$, independent of $N$, and a non-negative progressively measurable process $\rho^N \in H^2_{\Fc^0}([0,T]; \R)$ such that the Lipschitz stability along the limiting solution and the discretization error gives, 
\[
\|F^N(t,X_t^N)-F(t,X_t)\|_E
\le L\|X_t^N-X_t\|_E+\rho_t^N,
\qquad dt\otimes d\mathbb P\text{-a.e.}
\]
Then there exists a constant $C_L > 0$, depending only on $L$ and $T$, such that
\[
\E\!\left[ \sup_{t \in [0,T]} \|X_t^N - X_t\|^2_E + \int_0^T \|Z_t^N - Z_t\|^2_{\mathcal L_2}\,dt \right]
\leq C_L\left( \E\|\xi^N - \xi\|^2_E + \E\int_0^T |\rho_t^N|^2\,dt \right).
\]
\end{Lemma}

\begin{proof}
Set $\Delta X := X^N - X$, $\Delta Z := Z^N - Z$, $\Delta \xi := \xi^N - \xi$, and $\Delta F_t := F^N(t, X_t^N) - F(t, X_t)$. By assumption with the triangular inequality,
\begin{equation}\label{eq:driver_decomp}
\|\Delta F_t\|_E \leq L\,\|\Delta X_t\|_E + \rho_t^N, \qquad dt \otimes d\P\text{-a.e.}
\end{equation}

\noindent Applying It\^o's formula to $\|\Delta X_t\|^2_E$ on a separable Hilbert space (see, e.g., \cite[Theorem~4.32]{DaPratoZabczyk2014} or \cite[Proposition~1.122]{FabbriGozziSwiech2017}) yields, for $t \in [0,T]$, $\|\Delta X_t\|^2_E + \int_t^T \|\Delta Z_s\|^2_{\mathcal L_2}\,ds
= \|\Delta \xi\|^2_E + 2\int_t^T \langle \Delta X_s, \Delta F_s\rangle_E\,ds - 2\int_t^T \langle \Delta X_s, \Delta Z_s\,dW^0_s\rangle_E$.
By Cauchy--Schwarz and Young's inequality combined with \eqref{eq:driver_decomp} on the drift part, taking expectation and using that the stochastic integral has zero mean,
\[
\E\|\Delta X_t\|^2_E + \E\int_t^T \|\Delta Z_s\|^2_{\mathcal L_2}\,ds
\leq \E\|\Delta \xi\|^2_E + \E\int_0^T |\rho_s^N|^2\,ds + (2L+1)\,\E\int_t^T \|\Delta X_s\|^2_E\,ds.
\]
The backward Grönwall lemma then yields, for all $t \in [0,T]$,
\begin{equation}\label{eq:stab_L2}
\E\|\Delta X_t\|^2_E + \E\int_0^T \|\Delta Z_s\|^2_{\mathcal L_2}\,ds
\leq C_L\left( \E\|\Delta \xi\|^2_E + \E\int_0^T |\rho_s^N|^2\,ds \right).
\end{equation}

\noindent Returning to the It\^o expansion, taking the supremum in $t \in [0,T]$, applying the BDG inequality to the martingale term, bounding $\sup_s \|\Delta X_s\|_E$ outside the integral and then applying Young's inequality with parameter $\eta = 1/4$ yields to $\E\sup_t \bigg|\int_t^T \langle \Delta X_s, \Delta Z_s\,dW^0_s\rangle_E\bigg| \leq \tfrac{1}{4}\,\E\sup_t \|\Delta X_t\|^2_E + C\,\E\int_0^T \|\Delta Z_s\|^2_{\mathcal L_2}\,ds$.
Combining this with the It\^o expansion and absorbing $\frac{1}{4}\,\E\sup_t \|\Delta X_t\|^2_E$ into the left-hand side gives
\[
\E\sup_t \|\Delta X_t\|^2_E
\leq C_L\bigg( \E\|\Delta \xi\|^2_E + \E\int_0^T |\rho_s^N|^2\,ds + \E\int_0^T \|\Delta X_s\|^2_E\,ds + \E\int_0^T \|\Delta Z_s\|^2_{\mathcal L_2}\,ds \bigg).
\]
The two integrals on the right are already controlled by \eqref{eq:stab_L2}, and the conclusion follows.
\end{proof}

\begin{Remark}
\label{rem:stability_along_solutions}
Lemma~\ref{lem:stability} only requires a direct driver-difference estimate along the two processes being compared. No global Lipschitz property of the Riccati drivers is asserted. In the applications below, the constant multiplying the backward difference is deterministic and uniform in $N$, as a consequence of the operator and $S^\infty$-bounds established in Subsection~\ref{subsec:apriori}. All coefficient, kernel, diagonal-band, and mixed common-noise errors are collected in the process $\rho^N$.
\end{Remark}

\medskip

\noindent \noindent For later comparison, we write $\mathcal F^{\bar K}(t,k,\bar k)$ for the expression defining the driver in \eqref{eq:Kbar_BSRE_continuum}, with every occurrence of $K_t$ in $O$, $U$, $V$ and $\Psi$ replaced by $k$. Thus, $\mathcal F(t,\bar k) = \mathcal F^{\bar K}(t,K_t,\bar k)$. 

\smallskip
\noindent For the interaction component, we introduce the projected limiting processes  $\bar K_t^{[N]}:=\Pi_N^\circ\bar K_t$ and $Z_t^{\bar K,[N]}:=\Pi_N^\circ Z_t^{\bar K}$. Applying the operator $\Pi_N^\circ$ to the continuum equation gives $d\bar K_t^{[N]} = -\Pi_N^\circ \mathcal F^{\bar K}(t,K_t,\bar K_t)\,dt + Z_t^{\bar K,[N]}\,dW_t^0$, $\bar K_T^{[N]} = \Pi_N^\circ G^H$. 

\smallskip
\noindent We also introduce the exact embedded local driver $\mathfrak F_t^{K,\mathrm{tot},N} := \mathfrak F^{K,\mathrm{loc},N}(t,K_t^N) + \mathfrak d_t^{K,N}(K_t^N,\bar K_t^N)$.

\smallskip
\noindent Furthermore, if $u\in I_i^N$, we denote by $\mathfrak F_t^{Y,N}(u) := h_N^{-1}\big[\mathscr F_t^{p,N}\big]_i$ the exact rescaled driver of the embedded process $Y^N$, and by $\mathcal F_t^Y$ the driver inside the brackets in \eqref{eq:Y_BSDE_continuum}. Finally, $\ell_t^N:=\mathscr F_t^{q,N}$, whereas $\ell_t$ denotes the already integrated scalar driver defined after \eqref{eq:lambda_BSDE_continuum}.

\medskip
\noindent We now state the decomposition lemma. It is the technical core of the section: each discrete driver is shown to coincide with the corresponding continuum driver evaluated at the discrete objects, up to a discretization remainder controlled by the aggregate error rate $r_N$.

\begin{Lemma}[Decomposition and stability of the discrete drivers]
\label{lem:R1}
Under Assumption~\ref{ass:reg_consistency}, there exist non-negative progressively measurable processes $\rho^{K,N}$, $\rho^{\bar K,N}$, $\rho^{Y,N}$, $\rho^{q,N}$, and a deterministic constant $C>0$, independent of $N$, such that
\begin{equation}
\label{eq:R1_remainder_bound}
\E\int_0^T
\left(
|\rho_t^{K,N}|^2
+
|\rho_t^{\bar K,N}|^2
+
|\rho_t^{Y,N}|^2
+
|\rho_t^{q,N}|^2
\right)dt
\leq
Cr_N^2.
\end{equation}
Moreover, the following estimates hold $dt\otimes d\P$-a.e.

\begin{enumerate}
\item[\textup{(i)}]
\emph{Exact local Riccati driver}
\[
\left\|
\mathfrak F_t^{K,\mathrm{tot},N}
-
\mathcal F^K(t,K_t)
\right\|_{L^2(I)}
\leq
C\|K_t^N-K_t\|_{L^2(I)}
+
\rho_t^{K,N}.
\]

\item[\textup{(ii)}]
\emph{Projected interaction Riccati driver}
\begin{align*}
&
\left\|
\Pi_N^\circ
\mathfrak F^{\bar K,N}
(t,K_t^N,\bar K_t^N)
-
\Pi_N^\circ
\mathcal F^{\bar K}
(t,K_t,\bar K_t)
\right\|_{L^2(I^2)}
\leq
C\|\bar K_t^N-\bar K_t^{[N]}\|_{L^2(I^2)}
+
C\|K_t^N-K_t\|_{L^2(I)}
+
\rho_t^{\bar K,N}.
\end{align*}

\item[\textup{(iii)}]
\emph{Linear driver.}
\begin{align*}
\|\mathfrak F_t^{Y,N}-\mathcal F_t^Y\|_{L^2(I)}
\leq{}&
C\|Y_t^N-Y_t\|_{L^2(I)}
+
C\|K_t^N-K_t\|_{L^2(I)}
+
C\|Z_t^{K,N}-Z_t^K\|_{L^2(I)}
+
C\|\bar K_t^N-\bar K_t\|_{L^2(I^2)}\\
&
+C\|Z_t^{\bar K,N}-Z_t^{\bar K}\|_{L^2(I^2)}
+
\rho_t^{Y,N}.
\end{align*}

\item[\textup{(iv)}]
\emph{Integrated scalar driver}
\begin{align*}
|\ell_t^N-\ell_t|
\leq{}&
C\|K_t^N-K_t\|_{L^2(I)}
+
C\|\bar K_t^N-\bar K_t\|_{L^2(I^2)}
+
C\|Y_t^N-Y_t\|_{L^2(I)}
+
C\|Z_t^{Y,N}-Z_t^Y\|_{L^2(I)}
+
\rho_t^{q,N}.
\end{align*}
\end{enumerate}
\end{Lemma}

\begin{proof}
The proof is given in Appendix~\ref{app:R1}.
\end{proof}

\noindent We can now state and prove the main result of this section.

\begin{Theorem}[Discrete-to-continuum convergence of the Riccati systems]
\label{thm:convergence}
Under Assumption~\ref{ass:reg_consistency}, there exists a constant $C > 0$, independent of $N$, such that the following convergence estimates hold:
\begin{align}
\label{eq:T1}
\E\!\left[ \sup_{t \in [0,T]} \|K_t^N - K_t\|^2_{L^2(I)} + \int_0^T \|Z_t^{K,N} - Z_t^K\|^2_{L^2(I)}\,dt \right]
&\leq C\,r_N^2,\\
\label{eq:T2}
\E\!\left[ \sup_{t \in [0,T]} \|\bar K_t^N - \bar K_t\|^2_{L^2(I^2)} + \int_0^T \|Z_t^{\bar K,N} - Z_t^{\bar K}\|^2_{L^2(I^2)}\,dt \right]
&\leq C\,r_N^2,\\
\label{eq:T3}
\E\!\left[ \sup_{t \in [0,T]} \|Y_t^N - Y_t\|^2_{L^2(I)} + \int_0^T \|Z_t^{Y,N} - Z_t^Y\|^2_{L^2(I)}\,dt \right]
&\leq C\,r_N^2,\\
\label{eq:T4}
\E\!\left[ \sup_{t \in [0,T]} \bigg|q_t^N - \lambda_t\bigg|^2 + \int_0^T|Z_t^{q,N}-Z_t^\lambda|^2dt \right]
&\leq C\,r_N^2.
\end{align}
In particular, $(K^N, \bar K^N, Y^N, q^N)$ converges to $(K, \bar K, Y, \lambda)$ in the natural $S^2$-topologies as $N \to \infty$, with quantitative rate $r_N$.
\end{Theorem}

\begin{proof}
The proof proceeds in four steps, estimating successively $K^N-K$, $\bar K^N-\bar K$, $Y^N-Y$, and $q^N-\lambda$, following the triangular structure of the limiting system. The exact finite equation for $K^N$ still contains the diagonal remainder depending on $\bar K^N$, this remainder is controlled by the uniform estimates established above.

\medskip
\noindent
\textit{Step 1: convergence of the local Riccati component:} Lemma~\ref{lem:K_uniform} already gives the stronger estimate $\|K^N - K\|_{S^\infty_{\F^0}(L^2(I))} + \|Z^{K,N}-Z^K\|_{\mathrm{BMO}(L^2(I))} \leq C\delta_N^K$. Since $\delta_N^K\leq r_N$, this immediately implies \eqref{eq:T1}. Notice that this estimate includes both the exact diagonal driver $\mathfrak d^{K,N}$ and the terminal correction $h_N\operatorname{diag}_NG^{H,N}$.

\medskip
\noindent
\textit{Step 2: convergence of the interaction Riccati component:} We first compare $\bar K^N$ with the projected limiting process $\bar K^{[N]}=\Pi_N^\circ\bar K$. Their equations have respective drivers $\Pi_N^\circ \mathfrak F^{\bar K,N}(t,K_t^N,\bar K_t^N)$ and $\Pi_N^\circ \mathcal F^{\bar K}(t,K_t,\bar K_t)$. By Lemma~\ref{lem:R1}\textup{(ii)},
\begin{align*}
&
\left\|
\Pi_N^\circ
\mathfrak F^{\bar K,N}(t,K_t^N,\bar K_t^N)
-
\Pi_N^\circ
\mathcal F^{\bar K}(t,K_t,\bar K_t)
\right\|_{L^2(I^2)}
\leq
C\|\bar K_t^N-\bar K_t^{[N]}\|_{L^2(I^2)}
+
\widetilde\rho_t^{\bar K,N},
\end{align*}
where $\widetilde\rho_t^{\bar K,N} := C\|K_t^N-K_t\|_{L^2(I)} + \rho_t^{\bar K,N}$. Step~1 and \eqref{eq:R1_remainder_bound} imply $\E\int_0^T |\widetilde\rho_t^{\bar K,N}|^2dt \le Cr_N^2$. Moreover, $\|\bar K_T^N-\bar K_T^{[N]}\|_{L^2(I^2)}
= \|\Pi_N^\circ(G^{H,N}-G^H)\|_{L^2(I^2)} \leq \varepsilon_N^{\mathrm{ker}}$.
Lemma~\ref{lem:stability}, applied in $L^2(I^2)$, therefore gives
\begin{align}
\label{eq:projected_Kbar_convergence}
\E\Bigg[
\sup_{t\in[0,T]}
\|\bar K_t^N-\bar K_t^{[N]}\|_{L^2(I^2)}^2
+
\int_0^T
\|Z_t^{\bar K,N}-Z_t^{\bar K,[N]}\|_{L^2(I^2)}^2dt
\Bigg]
\leq
Cr_N^2.
\end{align}

\noindent Since $\bar K^N$ and $Z^{\bar K,N}$ disappear on $\mathcal D_N$, $\bar K^N-\bar K= \bar K^N-\bar K^{[N]} - \mathbf1_{\mathcal D_N}\bar K$, and similarly $Z^{\bar K,N}-Z^{\bar K} = Z^{\bar K,N}-Z^{\bar K,[N]} - \mathbf1_{\mathcal D_N}Z^{\bar K}$. Combining \eqref{eq:projected_Kbar_convergence} with Lemma~\ref{lem:diagonal_band_modulus} proves \eqref{eq:T2}.

\medskip
\noindent
\textit{Step 3: convergence of the linear component:} The terminal conditions satisfy $Y_T^N=Y_T=0$. By Lemma~\ref{lem:R1}\textup{(iii)}, \[\|\mathfrak F_t^{Y,N} - \mathcal F_t^Y\|_{L^2(I)} \leq C\|Y_t^N-Y_t\|_{L^2(I)} + \widetilde\rho_t^{Y,N},\] where
\begin{align*}
\widetilde\rho_t^{Y,N}
:={}
C\|K_t^N-K_t\|_{L^2(I)}
+
C\|Z_t^{K,N}-Z_t^K\|_{L^2(I)}
+
C\|\bar K_t^N-\bar K_t\|_{L^2(I^2)}
+
C\|Z_t^{\bar K,N}-Z_t^{\bar K}\|_{L^2(I^2)}
+
\rho_t^{Y,N}.
\end{align*}
Steps~1 and~2 and \eqref{eq:R1_remainder_bound} imply $\E\int_0^T|\widetilde\rho_t^{Y,N}|^2dt \leq Cr_N^2$. Applying Lemma~\ref{lem:stability} in $L^2(I;\mathbb R^d)$ gives \eqref{eq:T3}.

\medskip
\noindent
\textit{Step 4: convergence of the scalar component.} The conditional representations of the two scalar equations are $q_t^N = \E\left[ \int_t^T\ell_s^N\,ds \;\middle|\; \Fc_t^0 \right]$, and $\lambda_t = \E\left[ \int_t^T\ell_s\,ds \;\middle|\; \Fc_t^0 \right]$. Hence, setting $\Xi^N := \int_0^T|\ell_s^N-\ell_s|\,ds$, we have $|q_t^N-\lambda_t| \leq \E[\Xi^N\mid\Fc_t^0]$. Doob's $L^2$ inequality and Cauchy--Schwarz give \[\E\sup_{t\in[0,T]}|q_t^N-\lambda_t|^2 \leq 4\E|\Xi^N|^2\leq 4T\E\int_0^T|\ell_s^N-\ell_s|^2ds\] By Lemma~\ref{lem:R1} \textup{(iv)} and Steps~1--3, $\E\int_0^T|\ell_s^N-\ell_s|^2ds \leq Cr_N^2$. This proves \eqref{eq:T4}.
\end{proof}

\begin{Corollary}[Conditional algebraic convergence rates]
\label{cor:explicit_rates}
Let $(a_N)_{N\geq1}$ be a deterministic sequence converging to zero. If $\eta_N
+
\sqrt{h_N}
+
\varepsilon_N^{\mathrm{ker}}
+
\varepsilon_N^\Sigma
+
d_N^{\mathrm{diag}}
\leq
Ca_N$, then all the estimates in
\eqref{eq:T1}--\eqref{eq:T4} hold with right-hand side
$Ca_N^2$. In particular, for a uniform grid $h_N$, if $\eta_N
+
\varepsilon_N^{\mathrm{ker}}
+
\varepsilon_N^\Sigma
+
d_N^{\mathrm{diag}}
\leq
CN^{-1/2}$,
then the complete backward system converges with non-squared rate
$O(N^{-1/2})$.

\smallskip
\noindent An $O(N^{-1})$ non-squared rate is not a consequence of the aggregate rate $r_N$ as presently defined, since $r_N$ contains the generic contribution $\sqrt{h_N}$. Such an improvement requires rerunning the preceding estimates with refined $O(h_N)$ bounds for every occurrence of the interface and exact diagonal errors, together with together with $\eta_N+\varepsilon_N^{\mathrm{ker}} +\varepsilon_N^\Sigma+d_N^{\mathrm{diag}}=O(N^{-1})$. Under these additional targeted estimates, one may introduce the refined rate obtained by replacing the generic $\sqrt{h_N}$ contribution by $h_N$, and the conclusions of Theorem~\ref{thm:convergence} then hold with non-squared rate $O(N^{-1})$.
\end{Corollary}

\begin{Remark}
Block compatibility removes the cells crossing the interfaces and
therefore improves the corresponding local projection error. It does
not remove the diagonal band $\mathcal D_N$, nor the exact discrete
diagonal corrections. Algebraic estimates on these terms may be added
as targeted stronger assumptions when an explicit $O(N^{-1})$ result
is desired.
\end{Remark}

\section{Convergence of feedback coefficients and stability of feedback evaluations}
\label{sec:feedback_convergence}

Theorem~\ref{thm:convergence} establishes the convergence of the discrete  Riccati objects towards their continuum counterparts. We now transfer this  convergence to the level of the optimal control laws. This is carried out 
in two steps. We first show, in Lemma~\ref{lem:gain_convergence}, that the local gains $O^N$, $(O^N)^{-1}$ and $\widehat U^N$ converge in the stronger $L^\infty(\Omega\times[0,T];L^2(I))$ topology, whereas the nonlocal and affine gains $V^N$ and $\Gamma^N$ converge in expected essential supremum norm. We then prove, in Theorem~\ref{thm:feedback_convergence}, a stability estimate for the associated feedback maps: whenever the discrete and continuum state  fields are close, the resulting controls are close. The convergence of the  optimal controls themselves stated in  Corollary~\ref{cor:feedback_evaluation} is then obtained under certain condition. Actual convergence of the optimal controls requires, in addition, a comparison of the finite closed-loop state with conditionally independent cellwise copies of the continuum optimal state; this is carried out in Section~\ref{sec:state_convergence}.

\smallskip
\noindent Throughout this section, we work under Assumption~\ref{ass:reg_consistency} 
and we freely use the a priori bounds and convergence estimates of 
Sections~\ref{sec:problem_discrete}--\ref{sec:convergence}.

\subsection{Embedded feedback maps and convergence of the feedback coefficients}

\noindent Recall that the exact off-diagonal discrete feedback kernel is $V_t^N(u,v) := \sum_{i\neq j} V_{ij,t}^N \mathbf 1_{I_i^N}(u) \mathbf 1_{I_j^N}(v)$. Equivalently, $V_t^N = \Pi_N^\circ \left[ (C_t^N)^\top\bar K_t^N + (F_t^N)^\top K_t^NG_t^{E,N} \right]$. In particular, $V_t^N$ vanishes on $\mathcal D_N$. For $x\in {\mathsf H}$, define
\begin{equation}
\label{eq:finite_feedback_map}
\Phi_t^N(x)
:=
-(O_t^N)^{-1}
\left(
\widehat U_t^Nx
+
T_{V_t^N}x
+
\Gamma_t^N
\right).
\end{equation}
\noindent If $x$ is a step function and $u\in I_i^N$, then $(T_{V_t^N}x)(u) = h_N\sum_{j\neq i}V_{ij,t}^Nx_j$, and therefore $\Phi_t^N(x)(u) = -(O_{i,t}^N)^{-1} \big( \widehat U_{i,t}^Nx_i +$\\  
$ h_N\sum_{j\neq i}V_{ij,t}^Nx_j + \Gamma_{i,t}^N\big)$. Thus \eqref{eq:finite_feedback_map} is exactly the step-function embedding of the centralized finite-dimensional optimal feedback. For $x,m\in \mathsf H$, define the two-argument continuum feedback map by $\Phi_t(x,m)
:=
-O_t^{-1}
\left(
U_tx
+
T_{V_t}m
+
\Gamma_t
\right)$. The continuum optimal control is obtained by evaluating this map at the
individual optimal state and its conditional mean: $\widehat\alpha_t^*
=
\Phi_t(X_t^*,\bar X_t^*)$.

\begin{Lemma}[Convergence of the feedback coefficients]
\label{lem:gain_convergence}
Under Assumption~\ref{ass:reg_consistency}, there exists a deterministic
constant $C>0$, independent of $N$, such that
\begin{align}
\label{eq:local_gain_convergence}
&
\operatorname*{ess\,sup}_{(s,\omega)}\|O^N_s-O_s\|_{L^2(I)}
+
\operatorname*{ess\,sup}_{(s,\omega)}\|(O^N_s)^{-1}-O^{-1}_s\|_{L^2(I)}
+
\operatorname*{ess\,sup}_{(s,\omega)}\|\widehat U^N_s-U_s\|_{L^2(I)}
\leq
C\delta_N^K,
\end{align}
and
\begin{equation}
\label{eq:nonlocal_gain_convergence}
\E\left[
\operatorname*{ess\,sup}_{t\in[0,T]}
\left(
\|V_t^N-V_t\|_{L^2(I^2)}^2
+
\|\Gamma_t^N-\Gamma_t\|_{L^2(I)}^2
\right)
\right]
\leq
Cr_N^2.
\end{equation}
\end{Lemma}

\begin{proof}
The estimates for $O^N-O$, $(O^N)^{-1}-O^{-1}$ and $\widehat U^N-U$ follow from Lemma~\ref{lem:K_uniform}, the resolvent identity and the uniform coercivity of $O^N$ and $O$. For the interaction gain, set $\widetilde V_t^N
:=
(C_t^N)^\top\bar K_t^N
+
(F_t^N)^\top K_t^NG_t^{E,N}$ and $V_t^N=\Pi_N^\circ\widetilde V_t^N$. Since $V_t
=
C_t^\top\bar K_t
+
F_t^\top K_tG_t^E$, we have $V_t^N-V_t = \Pi_N^\circ(\widetilde V_t^N-V_t) - \mathbf 1_{\mathcal D_N}V_t$. Moreover,
\begin{align*}
\widetilde V_t^N-V_t
={}&
(C_t^N)^\top(\bar K_t^N-\bar K_t)
+
(C_t^N-C_t)^\top\bar K_t
+
(F_t^N)^\top K_t^N(G_t^{E,N}-G_t^E)
+
\big(
(F_t^N)^\top K_t^N-F_t^\top K_t
\big)G_t^E.
\end{align*}
The first term is controlled by
$\|\bar K_t^N-\bar K_t\|_{L^2(I^2)}$. For the second term,
Assumption~\ref{ass:reg_consistency}\textup{(A-row/col)} yields
\[
\|(C_t^N-C_t)^\top\bar K_t\|_{L^2(I^2)}^2
\leq
\|C_t^N-C_t\|_{L^2(I)}^2
\operatorname*{ess\,sup}_{u\in I}
\int_I|\bar K_t(u,v)|^2\,dv.
\]
The third term is controlled by the pointwise bound on $K^N$ and $\varepsilon_N^{\mathrm{ker}}$, while the fourth is controlled by $\|K^N-K\|_{L^2(I)}$, the local coefficient errors and the uniform boundedness of the limiting kernel $G^E$.

\smallskip
\noindent Finally, \[\|\mathbf 1_{\mathcal D_N}V_t\|_{L^2(I^2)}
\leq
C\|\mathbf 1_{\mathcal D_N}\bar K_t\|_{L^2(I^2)}
+
C\sqrt{h_N}\]
Combining these estimates with \eqref{eq:T1}, \eqref{eq:T2} and
Lemma~\ref{lem:diagonal_band_modulus} gives $\E\operatorname*{ess\,sup}_{t\in[0,T]}
\|V_t^N-V_t\|_{L^2(I^2)}^2
\leq
Cr_N^2$.
\noindent For the affine gain,
\begin{align*}
\Gamma_t^N-\Gamma_t
={}&
(C_t^N)^\top(Y_t^N-Y_t)
+
(C_t^N-C_t)^\top Y_t
+
\left[
(F_t^N)^\top K_t^ND_t^N
-
F_t^\top K_tD_t
\right]
+
(R_t^N\iota_t^N-R_t\iota_t).
\end{align*}
The first term is controlled by \eqref{eq:T3}. For the second term,
$\|(C_t^N-C_t)^\top Y_t\|_{L^2(I)}
\leq
\|C_t^N-C_t\|_{L^2(I)}
\|Y_t\|_{L^\infty(I)}$. The coefficient error is bounded deterministically by $C\delta_N^K$, while $\E\sup_t\|Y_t\|_{L^\infty(I)}^2\leq C$ by Assumption~\ref{ass:reg_consistency}\textup{(A-row/col)}. Hence this term contributes at most $C(\delta_N^K)^2$ after expectation. The two remaining differences are controlled by direct product expansions, the boundedness of the local coefficients, Lemma~\ref{lem:K_uniform} and Assumption~\ref{ass:reg_consistency}\textup{(A-coeff)}. Therefore,
$\E\operatorname*{ess\,sup}_{t\in[0,T]}
\|\Gamma_t^N-\Gamma_t\|_{L^2(I)}^2
\leq
Cr_N^2$.
\end{proof}

\subsection{Uniform operator bounds}
\label{subsec:operator_bounds}

The stability estimate for the feedback maps proved in the next subsection 
requires the integral operators $T_{V^N}$ and $T_V$ to be uniformly bounded from $\mathsf H$ to $\mathsf U$, uniformly in $N$. For the limiting operator $T_V$, the required operator bound follows from the continuum solvability theorem $T_{\bar K}\in L^\infty(\Omega;C([0,T];\mathcal L(L^2(I))))$. For the discrete operator $T_{V^N}$, we have the analogous bound with $T_{\bar{K}^N}$. We show that the required bound is a direct consequence of the  finite-dimensional LQ structure: the full discrete quadratic Riccati  operator is bounded by a value-function argument, and the interaction part is recovered by subtracting the local part. This is the content of the following lemma.

\begin{Lemma}[Uniform operator bounds for the feedback kernels]
\label{lem:operator_bounds}
There exists a deterministic constant $C>0$, independent of $N$, such
that
\[
\sup_{N\ge1}\operatorname*{ess\,sup}_{(t,\omega)}
\|T_{V_t^N}\|_{\mathcal L(\mathsf H;\mathsf U)}
+
\operatorname*{ess\,sup}_{(t,\omega)}
\|T_{V_t}\|_{\mathcal L(\mathsf H;\mathsf U)}
\le C.
\]
\end{Lemma}

\begin{proof}
Since $\bar K^N$ vanishes on $\mathcal D_N$, $T_{V_t^N}
=
M_{(C_t^N)^\top}T_{\bar K_t^N}
+
M_{(F_t^N)^\top K_t^N}
T_{\Pi_N^\circ G_t^{E,N}}$. The multiplication operators are uniformly bounded by the pointwise bounds on $C^N,F^N,K^N$. Moreover, Lemma~\ref{lem:Kbar_bound} gives $\sup_{N,t}
\|T_{\bar K_t^N}\|_{\mathcal L(\mathsf H)}
\leq C$, and 

\noindent $\|T_{\Pi_N^\circ G_t^{E,N}}\|_{\mathcal L(\mathsf H)}
\leq
\|\Pi_N^\circ G_t^{E,N}\|_{L^2(I^2)}
\leq C$. This proves the discrete estimate. The continuum estimate follows in the same way from $T_{V_t}
=
M_{C_t^\top}T_{\bar K_t}
+
M_{F_t^\top K_t}T_{G_t^E}$ and the continuum operator bound in Lemma~\ref{lem:Kbar_bound}.
\end{proof}

\subsection{Stability and conditional convergence of feedback evaluations}
\label{subsec:feedback_maps}

We now turn to the feedback maps themselves. For $x\in L^2(I;\mathbb R^d)$, we have the finite-$N$ and continuum feedback maps \eqref{eq:finite_feedback_map} with the two-argument continuum map $\Phi_t(x,m)(u)$. When evaluated along the corresponding optimal closed-loop states, these maps give the optimal controls: $\widehat\alpha_t^{N,*}=\Phi_t^N(X_t^{N,*})$ and $\widehat\alpha_t^*=\Phi_t(X_t^*,\bar X_t^*)$. The result below is a conditional stability statement; it becomes convergence of the optimal controls only after the state and empirical-interaction errors have been estimated.

\noindent For random fields $X$ and $m$, set $\mathcal M_t(X,m) := U_tX_t+T_{V_t}m_t+\Gamma_t$ and we introduce the conditional non-concentration process
\[
\theta_t^{X,m}
:=
\operatorname*{ess\,sup}_{u\in I}
\E\left[
|X_t(u)|^2
+
|\mathcal M_t(X,m)(u)|^2
\;\middle|\;
\Fc_t^0
\right],
\quad
\textit{and}
\quad
\Theta_{X,m}
:=
\operatorname{ess\,sup}_{t\in[0,T]}\theta_t^{X,m}.
\]
Note that $\theta_t^{X,m}$ is $\mathcal F_t^0$-measurable for each $t$. We use jointly measurable versions of these conditional moments; the essential supremum in time is taken with respect to $dt$.

\begin{Theorem}[Stability and conditional convergence of the feedback evaluations]
\label{thm:feedback_convergence}
Assume the hypotheses of Theorem~\ref{thm:convergence}. Let
$X^N,X$ be square-integrable adapted state fields on a common probability
space and let $m$ be an $\mathbb F^0$-adapted $\mathsf H$-valued process. Assume \[\sup_N
\E\sup_{t\in[0,T]}\|X_t^N\|_{\mathsf H}^2
+
\E\sup_{t\in[0,T]}
\left(
\|X_t\|_H^2+\|m_t\|_{\mathsf H}^2
\right)
<\infty \quad \textit{and} \quad \E[\Theta_{X,m}]<\infty\]
Then there exists a constant $C>0$, independent of $N$, such that
\begin{align}
\label{eq:feedback_stability_two_arguments}
&
\E\int_0^T
\|
\Phi_t^N(X_t^N)-\Phi_t(X_t,m_t)
\|_\mathsf U^2dt
\leq
C\E\int_0^T
\|X_t^N-X_t\|_{\mathsf H}^2dt
+
C\E\int_0^T
\|T_{V_t^N}X_t^N-T_{V_t}m_t\|_\mathsf U^2dt
+
Cr_N^2.
\end{align}
The constant $C$ may depend on
$\E[\Theta_{X,m}]$, but not on $N$.

\smallskip
\noindent In particular, if $\E\int_0^T \|X_t^N-X_t\|_{\mathsf H}^2dt \longrightarrow0$ and $\E\int_0^T
\|T_{V_t^N}X_t^N-T_{V_t}m_t\|_\mathsf U^2dt
\longrightarrow0$, then
\[
\Phi^N(X^N)
\longrightarrow
\Phi(X,m)
\quad\text{in}\quad
L^2(\Omega\times[0,T];\mathsf U).
\]
\end{Theorem}

\begin{proof}
Write $\Phi_t^N(X_t^N)-\Phi_t(X_t,m_t)
=
A_t^N+B_t^N$, where
\[
A_t^N
:={}
-(O_t^N)^{-1}
\Big[
\widehat U_t^N(X_t^N-X_t)
+
T_{V_t^N}X_t^N-T_{V_t}m_t
\Big],
\quad
B_t^N:=
-\big((O_t^N)^{-1}-O_t^{-1}\big)
\mathcal M_t(X,m)
-
(O_t^N)^{-1}
\Big[
(\widehat U_t^N-U_t)X_t
+
(\Gamma_t^N-\Gamma_t)
\Big].
\]

\medskip
\noindent\textit{Step 1: Stability with respect to the state and the
interaction aggregate.} By uniform coercivity, $\operatorname*{ess\,sup}_{u\in I} |(O_t^N(u))^{-1}|_{\mathrm F}\le \sqrt m\,c_R^{-1}$, $dt\otimes d\mathbb P\text{-a.e.}$. Moreover, $\widehat U^N$ is uniformly bounded in $L^\infty(I)$. Consequently,
$\|A_t^N\|_\mathsf U^2
\leq
C\|X_t^N-X_t\|_{\mathsf H}^2
+
C\|T_{V_t^N}X_t^N-T_{V_t}m_t\|_\mathsf U^2$. After integration and expectation,
\begin{align}
\label{eq:feedback_A_estimate}
\E\int_0^T\|A_t^N\|_\mathsf U^2dt
\leq{}&
C\E\int_0^T\|X_t^N-X_t\|_{\mathsf H}^2dt
+
C\E\int_0^T
\|T_{V_t^N}X_t^N-T_{V_t}m_t\|_\mathsf U^2dt.
\end{align}

\medskip
\noindent\textit{Step 2: convergence of the local feedback gains.} Since $(O_t^N)^{-1}-O_t^{-1}$ is $\Fc_t^0$-measurable, conditioning on $\Fc_t^0$ gives
\begin{align*}
\E\left[
\left\|
\big((O_t^N)^{-1}-O_t^{-1}\big)
\mathcal M_t(X,m)
\right\|_\mathsf U^2
\right]
&
\leq
\E\left[
\int_I
\left|
(O_t^N(u))^{-1}-O_t(u)^{-1}
\right|^2
\E\left[
|\mathcal M_t(X,m)(u)|^2
\;\middle|\;
\Fc_t^0
\right]du
\right]
\\
&
\leq
\E\left[
\theta_t^{X,m}
\left\|
(O_t^N)^{-1}-O_t^{-1}
\right\|_{L^2(I)}^2
\right].
\end{align*}
Lemma~\ref{lem:gain_convergence} gives the deterministic essential-supremum estimate $\operatorname*{ess\,sup}_{(t,\omega)}
\left\|
(O_t^N)^{-1}-O_t^{-1}
\right\|_{L^2(I)}^2
\leq
C(\delta_N^K)^2$. Therefore,
\begin{align}
\label{eq:feedback_inverse_gain_product}
&
\E\int_0^T
\left\|
\big((O_t^N)^{-1}-O_t^{-1}\big)
\mathcal M_t(X,m)
\right\|_\mathsf U^2dt
\leq
C(\delta_N^K)^2
\E\int_0^T\theta_t^{X,m}dt
\leq
CT(\delta_N^K)^2
\E[\Theta_{X,m}]
\leq
Cr_N^2.
\end{align}

\noindent The same conditional argument applies to the local gain difference. Indeed,
\begin{align*}
&
\E\left[
\|(\widehat U_t^N-U_t)X_t\|_\mathsf U^2
\right]
\leq
\E\left[
\int_I
|\widehat U_t^N(u)-U_t(u)|^2
\E\left[
|X_t(u)|^2
\;\middle|\;
\Fc_t^0
\right]du
\right]
\leq
\E\left[
\theta_t^{X,m}
\|\widehat U_t^N-U_t\|_{L^2(I)}^2
\right].
\end{align*}
Using again Lemma~\ref{lem:gain_convergence}, we obtain
\begin{equation}
\label{eq:feedback_local_gain_product}
\E\int_0^T
\|(\widehat U_t^N-U_t)X_t\|_\mathsf U^2dt
\leq
Cr_N^2.
\end{equation} Finally, the affine gain difference does not multiply a state field. Uniform coercivity and Lemma~\ref{lem:gain_convergence} give
\begin{align}
\label{eq:feedback_affine_gain}
&
\E\int_0^T
\left\|
(O_t^N)^{-1}
(\Gamma_t^N-\Gamma_t)
\right\|_\mathsf U^2dt
\leq
C\E\int_0^T
\|\Gamma_t^N-\Gamma_t\|_\mathsf U^2dt
\leq
Cr_N^2.
\end{align}
Combining
\eqref{eq:feedback_inverse_gain_product},
\eqref{eq:feedback_local_gain_product} and
\eqref{eq:feedback_affine_gain}, we obtain $\E\int_0^T\|B_t^N\|_\mathsf U^2dt \leq Cr_N^2$. Together with \eqref{eq:feedback_A_estimate}, this proves \eqref{eq:feedback_stability_two_arguments}. The final convergence statement follows immediately from $r_N\to0$.
\end{proof}

\begin{Remark}[Conditional non-concentration and common noise]
\label{rem:non_concentration_common_noise}
The condition $\E[\Theta_{X,m}]<\infty$ is a conditional non-concentration condition in the label variable. It does not require the conditional second moments of the state to be bounded by a deterministic constant. The process $\theta_t^{X,m}$ is allowed to depend on the realization of the common noise. For example, in the scalar model $dX_t(u)=\sigma_0(u)\,dW_t^0$ with deterministic initial condition $\xi\in L^\infty(I)$, $\E\left[
|X_t(u)|^2
\;\middle|\;
\Fc_t^0
\right]
=
|\xi(u)+\sigma_0(u)W_t^0|^2$. This quantity is not uniformly bounded by a deterministic constant, but its supremum in time has finite moments under the standing boundedness assumptions. In Theorem~\ref{thm:feedback_convergence}, no localization is needed.
Indeed, Lemma~\ref{lem:gain_convergence} bounds $(O^N)^{-1}-O^{-1}$ and $\widehat U^N-U$ deterministically in $L^\infty(\Omega\times[0,T];L^2(I))$. Hence their products with the random state fields are controlled directly by conditioning on $\Fc_t^0$.

\noindent A localization argument may nevertheless be required later when estimating the interaction error $T_{V_t^N}X_t^N-T_{V_t}m_t$, because the available estimate for the interaction-gain error is $\mathbb E\left[\operatorname*{ess\,sup}_{t\in[0,T]} \|V_t^N- V_t\|_{L^2(I^2)}^2\right]\le Cr_N^2$, whereas a deterministic bound of order $r_N$ is not available. Any resulting degradation of the final convergence rate therefore comes from the empirical interaction comparison, not from the local feedback gains.
\end{Remark}

\begin{Corollary}[Quantitative feedback evaluation]
\label{cor:feedback_evaluation}
Under the assumptions of Theorem~\ref{thm:feedback_convergence}, set 

\[a_N^2 := \E\int_0^T \|X_t^N-X_t\|_{\mathsf H}^2dt, \qquad b_N^2 := \E\int_0^T \|T_{V_t^N}X_t^N-T_{V_t}m_t\|_{\mathsf U}^2dt \] Then
\[
\E\int_0^T
\|
\Phi_t^N(X_t^N)-\Phi_t(X_t,m_t)
\|_\mathsf U^2dt
\leq
C\left(
a_N^2+b_N^2+r_N^2
\right).
\]
Consequently, if $a_N\to0$ and $b_N\to0$, then
\[
\Phi^N(X^N)
\longrightarrow
\Phi(X,m)
\quad\text{in}\quad
L^2(\Omega\times[0,T];\mathsf U).
\]
More precisely, if $a_N+b_N\leq Cs_N$ for some deterministic sequence $s_N\to0$, then the feedback evaluations converge with non-squared rate $O(s_N+r_N)$.
\end{Corollary}

\begin{proof}
The estimate follows directly from
\eqref{eq:feedback_stability_two_arguments}. The convergence statement
and the rate follow by taking square roots.
\end{proof}

\begin{Remark}[The remaining empirical interaction error]
\label{rem:empirical_feedback_error}
The term $T_{V_t^N}X_t^N-T_{V_t}m_t$ has intentionally been left explicit in Theorem~\ref{thm:feedback_convergence}. For the optimal systems, $m_t=\bar X_t^*$, and this term becomes $T_{V_t^N}X_t^{N,*}-T_{V_t}\bar X_t^*$. It cannot in general be bounded solely by $\|X_t^{N,*}-X_t^*\|_{L^2(I)}$, because $X_t^*$ is an individual random state whereas $\bar X_t^*$ is its conditional mean. Its estimate requires a conditional empirical-fluctuation argument, followed by the kernel and gain consistency estimates. This argument is carried out in Section~\ref{sec:state_convergence}.
\end{Remark}

\begin{Remark}[Centralized and decentralized feedback structures]
\label{rem:decentralized_link}
The exact finite-dimensional feedback $\Phi_t^N(x) = -(O_t^N)^{-1} \left( \widehat U_t^Nx + T_{V_t^N}x + \Gamma_t^N \right)$ is centralized. Indeed, if $u\in I_i^N$, then $(T_{V_t^N}x)(u) = h_N\sum_{j\neq i}V_{ij,t}^Nx_j$, so that the control of agent $i$ depends on the complete finite state  vector $(x_1,\ldots,x_N)$. This is consistent with the finite social planner problem, whose admissible centralized controls are adapted to the global finite-system filtration.

\smallskip
\noindent By contrast, the limiting feedback has the representative-agent form
$\Phi_t(x,m)(u)
=
-O_t(u)^{-1}
\left(
U_t(u)x(u)
+
T_{V_t}m(u)
+
\Gamma_t(u)
\right)$. When $m=\bar X_t^*$, it uses only the individual state, the label and an
$\Fc_t^0$-adapted conditional aggregate. It therefore suggests the decentralized lifted feedback $\Phi_t^{\mathrm{dec}}(u,x)
:=
-O_t(u)^{-1}
\left(
U_t(u)x
+
T_{V_t}\bar X_t^*(u)
+
\Gamma_t(u)
\right)$. 

\noindent For the finite population, Section~\ref{sec:decentralized_lift} constructs a decentralized control by taking cell averages of the limiting feedback coefficients. The resulting control uses each agent's own state and the common noise information. Its admissibility and asymptotic optimality are proved in that section. In particular, the comparison $V^{N,\mathrm{cent}}\le V^{N,\mathrm{dec}} \le J^N(\alpha^{N,\mathrm{lift}})$ relates the lifted strategy to both finite-population optimization problems.
\end{Remark}

\section{Convergence of the optimal trajectories and value functions}
\label{sec:state_convergence}

Theorem~\ref{thm:feedback_convergence} establishes the stability of the  feedback maps conditionally on the proximity of the state fields. In this  section we close the loop. A direct $L^2(I)$ comparison between the finite state vector and a continuum random field would require an uncountable family of independent idiosyncratic Brownian motions, which is not part of the representative-agent model. Instead, we use a cellwise coupling: for every finite agent, we construct a copy of the continuum optimal state whose label is uniformly distributed in the corresponding grid cell and which is driven by the same private Brownian motion as that agent. This gives the correct mean-square formulation of the conditional propagation-of-chaos estimate and preserves the exact cell-average identities for step kernels.

\smallskip
\noindent Throughout this section, we work under Assumption~\ref{ass:reg_consistency} 
and the standing assumptions of the previous sections. In particular, we 
recall that the coefficients and kernels, both discrete and continuum, are 
$\F^0$-adapted, so that they may be pulled out of conditional expectations 
given $\Fc_t^0$.

\subsection{Cellwise continuum copies and empirical fluctuations}

To lighten notation, throughout this section we set $X^*:=\widehat X$, $\bar X^*:=\widehat{\bar X}$, $\widehat\alpha^*:=\widehat\alpha$, for the continuum optimizer, and $X^{N,*}:=\widehat{\mathbf X}^{\,N}$, $\widehat\alpha^{N,*}:= \widehat{\boldsymbol\alpha}^{\,N}$, for the finite centralized optimizer. 

\noindent For $f\in \mathsf H$, recall the orthogonal cell-average projection, $(\mathsf P_N f)(u)$ and for $\mathbf x^N, \mathbf y^N\in\mathbb H_N$, set $\langle \mathbf x^N,\mathbf y^N\rangle_N:=h_N\sum_i\langle x^i,y^i\rangle$.

\begin{Assumption}[Cellwise coupling of the initial conditions]
\label{ass:initial_cellwise_coupling}
For every $N$, the probability space can be enlarged by independent labels $U_1^N,\ldots,U_N^N$, independent of the common noise, with $\mathcal L(U_i^N)(du) = h_N^{-1}\mathbf1_{I_i^N}(u)\,du$. Using independent private randomizers and the Brownian motions $W^{1,N},\ldots,W^{N,N}$, we construct cellwise copies $(\widetilde\xi_i^N,\widetilde X^{i,N,*}, \widetilde\alpha^{i,N,*})$ of the continuum optimal system. Conditional on $\Fc_T^0$, the full trajectories of the cellwise copies are independent. Moreover, by non-anticipativity, for each $t\in[0,T]$ the stopped variables $\big( \widetilde\xi_i^N, \widetilde X_{\cdot\wedge t}^{i,N,*}, \widetilde\alpha_{\cdot\wedge t}^{i,N,*} \big)$, $1\leq i\leq N$, are conditionally independent given $\Fc_t^0$. The $i$-th copy has the conditional law of the continuum optimal system when the label is uniformly distributed on $I_i^N$. The finite and continuum-copy initial conditions are coupled so that
\[
\varepsilon_{\xi,N}^2
:=
\E\left[
|\xi^N-\widetilde\xi^N|_N^2
\right]
\longrightarrow0,
\quad
\sup_{N\geq1}
\E\left[
|\xi^N|_N^2+|\widetilde\xi^N|_N^2
\right]
<\infty.
\]
\end{Assumption}

\noindent The preceding coupling assumption is compatible with a measurable representation of the continuum closed-loop equation. More explicitly, the $i$-th copy is driven by $(U_i^N,W^{i,N},W^0)$ and satisfies
\begin{align*}
d\widetilde X_t^{i,N,*}
={}&
\Big[
A_t(U_i^N)+B_t(U_i^N)\widetilde X_t^{i,N,*}
+(T_{G_t^B}\bar X_t^*)(U_i^N)
+C_t(U_i^N)\widetilde\alpha_t^{i,N,*}
\Big]dt
\\
&+
\Big[
D_t(U_i^N)+E_t(U_i^N)\widetilde X_t^{i,N,*}
+(T_{G_t^E}\bar X_t^*)(U_i^N)
+F_t(U_i^N)\widetilde\alpha_t^{i,N,*}
\Big]dW_t^{i,N}
+
\Sigma_t^0(U_i^N)dW_t^0,
\end{align*}
where
\[
\widetilde\alpha_t^{i,N,*}
=
-O_t(U_i^N)^{-1}
\Big[
U_t(U_i^N)\widetilde X_t^{i,N,*}
+(T_{V_t}\bar X_t^*)(U_i^N)
+\Gamma_t(U_i^N)
\Big].
\]
Set
\[
\bar x_{i,t}^N
:=
\E[\widetilde X_t^{i,N,*}\mid\Fc_t^0]
=
h_N^{-1}\int_{I_i^N}\bar X_t^*(u)\,du.
\]

\noindent We use the representative-state counterpart of the conditional moment process introduced in Section~\ref{subsec:feedback_maps}. Set
\[
\begin{aligned}
\theta_t^*
&:=\operatorname*{ess\,sup}_{u\in I}
\mathbb E\!\left[
|X_t^*|^2+
|\mathcal M_t(X^*,\bar X^*)(u)|^2
\mid\mathcal F_t^0,U=u\right],\\
\Theta_* 
&:=\max\Big\{
\operatorname*{ess\,sup}_{t\in[0,T]}\theta_t^*;\quad
\operatorname*{ess\,sup}_{u\in I}\mathbb E[|\xi|^2\mid U=u];\quad
\operatorname*{ess\,sup}_{u\in I}
\mathbb E[|X_T^*|^2\mid\mathcal F_T^0,U=u]
\Big\}.
\end{aligned}
\]
where
$\mathcal M_t(X^*,\bar X^*)(u) =U_t(u)X_t^*+(T_{V_t}\bar X_t^*)(u)+\Gamma_t(u)$. We choose jointly measurable versions of the conditional moments. The time essential supremum uses Lebesgue measure; the two endpoint terms retain the state-moment bounds needed for the initial and terminal costs. Uniform coercivity gives the corresponding conditional second-moment bound for $\widehat\alpha_t^*$, for $dt\otimes d\mathbb P$-almost every $(t,\omega)$.

\begin{Lemma}[Conditional empirical fluctuation]
\label{lem:conditional_empirical_fluctuation}
Let $L_t^N$ be step kernel and define $(\mathbb L_t^Nz^N)^i := h_N\sum_{j=1}^NL_{ij,t}^Nz^j$, $L_t^N$ being $\Fc_t^0$-measurable. Then
$\E\left[
\left|
\mathbb L_t^N
(\widetilde X_t^{N,*}-\bar x_t^N)
\right|_N^2
\,\middle|\,
\Fc_t^0
\right]
\leq
h_N\theta_t^*\|L_t^N\|_{L^2(I^2)}^2$, where $\theta_t^*$ is defined above and
$\theta_t^*\le\Theta_*$ for $dt\otimes d\mathbb P$-almost every
$(t,\omega)$. Moreover, for $u\in I_i^N$,
\begin{equation}
\label{eq:cell_average_identity}
h_N\sum_{j=1}^NL_{ij,t}^N\bar x_{j,t}^N
=
(T_{L_t^N}\bar X_t^*)(u).
\end{equation}
\end{Lemma}

\begin{proof}
Conditional on $\Fc_t^0$, the random variables $\widetilde X_t^{j,N,*}-\bar x_{j,t}^N$ are independent and centered. Consequently, all cross terms with different indices vanish, and
\[
\E\left[
\left|
\mathbb L_t^N
(\widetilde X_t^{N,*}-\bar x_t^N)
\right|_N^2
\,\middle|\,
\Fc_t^0
\right]
=
h_N^3\sum_{i,j=1}^N
\E\left[
\left|
L_{ij,t}^N
(\widetilde X_t^{j,N,*}-\bar x_{j,t}^N)
\right|^2
\,\middle|\,
\Fc_t^0
\right]
\leq
h_N^3\theta_t^*\sum_{i,j=1}^N|L_{ij,t}^N|^2
=
h_N\theta_t^*\|L_t^N\|_{L^2(I^2)}^2
\]
\noindent Finally, $h_N\sum_jL_{ij,t}^N\bar x_{j,t}^N = \sum_j\int_{I_j^N}L_t^N(u,v)\bar X_t^*(v)\,dv = (T_{L_t^N}\bar X_t^*)(u)$ with the definition of $\bar x_{j,t}^N$ and the fact that $L_t^N$ is constant on each cell gives the result.
\end{proof}

\subsection{Convergence of the centralized optimal state and control}

\begin{Theorem}[Cellwise convergence of the optimal closed-loop states]
\label{thm:state_convergence}
Assume the hypotheses of Theorem~\ref{thm:convergence} and
Assumption~\ref{ass:initial_cellwise_coupling}, and suppose that
$\E[\Theta_*]<\infty$. Then there exists a deterministic constant $C>0$,
independent of $N$ and $M$, such that, for every $M\geq1$,
\begin{align}
\label{eq:state_convergence_localized}
\E\left[
\sup_{t\in[0,T]}
|X_t^{N,*}-\widetilde X_t^{N,*}|_N^2
\right]
\leq
C\Big(
\varepsilon_{\xi,N}^2
+Mr_N^2
+\E[\Theta_*\mathbf1_{\{\Theta_*>M\}}]
\Big).
\end{align}
In particular, $X^{N,*}-\widetilde X^{N,*}\to0$ in empirical $S^2$ as
soon as $\varepsilon_{\xi,N}\to0$. If $\Theta_*\in L^2(\Omega)$, then
\begin{equation}
\label{eq:state_convergence_L2_rate}
\E\left[
\sup_{t\in[0,T]}
|X_t^{N,*}-\widetilde X_t^{N,*}|_N^2
\right]
\leq
C(\varepsilon_{\xi,N}^2+r_N).
\end{equation}
If $\Theta_*$ is essentially bounded by a deterministic constant, the sharper rate estimate $C(\varepsilon_{\xi,N}^2+r_N^2)$ is recovered.

\end{Theorem}

\begin{proof}
Set
$\Delta X_t^N:=X_t^{N,*}-\widetilde X_t^{N,*}$ and
$\Delta\alpha_t^N:=\widehat\alpha_t^{N,*}
-\widetilde\alpha_t^{N,*}$, and define
$\phi_N(t):=\E[\sup_{s\leq t}|\Delta X_s^N|_N^2]$.

\medskip
\noindent\textit{Step 1: stability of the state equations.} Subtracting the two synchronously coupled equations and applying the finite-dimensional BDG and Young inequalities in the empirical norm gives
\begin{align}
\label{eq:sampled_state_stability}
\phi_N(t)
\leq
C\varepsilon_{\xi,N}^2
+C\E\int_0^t
\big(
|\Delta b_s^N|_N^2
+|\Delta\sigma_s^N|_{N,\mathrm{HS}}^2
+|\Delta\sigma_s^{0,N}|_{N,\mathrm{HS}}^2
\big)ds.
\end{align}
The local terms are decomposed like that : $B_t^{i,N}X_t^{i,N,*}  -B_t(U_i^N)\widetilde X_t^{i,N,*} = B_t^{i,N}\Delta X_t^{i,N} +\big(B_t^{i,N}-B_t(U_i^N)\big) \widetilde X_t^{i,N,*}$, with analogous decompositions for $C,E,F$. Since $U_i^N$ is uniform on $I_i^N$, conditioning on $\Fc_t^0$ yields, for any such local coefficient
$\beta$ and for $Z=X^*$ or $\widehat\alpha^*$, \[\E\left[
h_N\sum_i
|\beta_t^{i,N}-\beta_t(U_i^N)|^2
|\widetilde Z_t^{i,N}|^2
\right]
\leq
C\E\left[
\theta_t^*
\|\beta_t^N-\beta_t\|_{L^2(I)}^2
\right]
\leq
Cr_N^2\E[\Theta_*]\]

\noindent Here the last inequality uses the deterministic pathwise consistency bound from Assumption~\ref{ass:reg_consistency}. The additive errors $A^N-A$ and $D^N-D$ are simpler. Assumption~\ref{ass:reg_consistency}\textup{(A-Sigma)} also gives \[\E\int_0^T
h_N\sum_i
|\Sigma_t^{0,i,N}-\Sigma_t^0(U_i^N)|^2dt
\leq
T(\varepsilon_N^\Sigma)^2\].

\noindent For $G=G^B$ or $G^E$, the interaction difference is decomposed as
\begin{align}
\label{eq:sampled_interaction_decomposition}
\mathsf G_t^NX_t^{N,*}
-(T_{G_t}\bar X_t^*)(U^N)
={}
\mathsf G_t^N\Delta X_t^N
+\mathsf G_t^N(\widetilde X_t^{N,*}-\bar x_t^N)
+
\big(T_{G_t^N-G_t}\bar X_t^*\big)(U^N),
\end{align}
where \eqref{eq:cell_average_identity} was used. The first term is bounded by the uniform operator norm of $\mathsf G_t^N$, the second contributes $Ch_N\E[\Theta_*]$ by Lemma~\ref{lem:conditional_empirical_fluctuation}, and the last one is bounded by $\E\left[
\|T_{G_t^N-G_t}\bar X_t^*\|_{\mathsf H}^2
\right]
\leq
(\varepsilon_N^{\mathrm{ker}})^2
\E[\Theta_*]$. Since $h_N\leq r_N^2$, collecting the drift and diffusion estimates in
\eqref{eq:sampled_state_stability} gives
\begin{align}
\label{eq:state_before_feedback}
\phi_N(t)
\leq
C\varepsilon_{\xi,N}^2
+C\int_0^t\phi_N(s)\,ds
+C\E\int_0^t|\Delta\alpha_s^N|_N^2ds
+Cr_N^2.
\end{align}

\medskip
\noindent\textit{Step 2: sampled feedback comparison.}
The proof of Theorem~\ref{thm:feedback_convergence}, applied cellwise, gives the same local-gain estimates. The only new term is the empirical interaction in the feedback. It has the exact decomposition
\begin{align}
\label{eq:sampled_feedback_aggregate}
\mathbb V_t^NX_t^{N,*}
-(T_{V_t}\bar X_t^*)(U^N)
={}&
\mathbb V_t^N\Delta X_t^N
+\mathbb V_t^N(\widetilde X_t^{N,*}-\bar x_t^N)
+
\big(T_{V_t^N-V_t}\bar X_t^*\big)(U^N).
\end{align}
The first term is controlled by Lemma~\ref{lem:operator_bounds}. For the second term, Lemma~\ref{lem:Kbar_bound}(b), the local coefficient bounds, and Assumption~\ref{ass:reg_consistency} give \[\sup_{N\ge1}\operatorname {ess\,sup}_{(t,\omega)} \|V_t^N\|_{L^2(I^2)} \le C\sup_{N\ge1}\operatorname {ess\,sup}_{(t,\omega)} \big(\|\bar K_t^N\|_{L^2(I^2)} +\|G_t^{E,N}\|_{L^2(I^2)}\big) \le C\].

\noindent Hence Lemma~\ref{lem:conditional_empirical_fluctuation} yields \[\mathbb E\left| \mathbb V_t^N(\widetilde X_t^{N,*}-\bar x_t^N) \right|_N^2 \le h_N\mathbb E\!\left[\theta_t^* \|V_t^N\|_{L^2(I^2)}^2\right] \le Ch_N\mathbb E[\Theta_*]\quad \text{for a.e. }t\]. 

\noindent For the last term, on $\{\Theta_*\leq M\}$ by Lemma\ref{lem:Kbar_bound}(b) and \eqref{eq:nonlocal_gain_convergence} we have, 

\[\E\int_0^T \|T_{V_t^N-V_t}\bar X_t^*\|_\mathsf U^2 \mathbf1_{\{\Theta_*\leq M\}}dt \leq M\E\int_0^T\|V_t^N-V_t\|_{L^2(I^2)}^2dt \leq CMr_N^2\]

\noindent On $\{\Theta_*>M\}$, the uniform operator bounds on $T_{V^N}$ and $T_V$ give $\|T_{V_t^N-V_t}\bar X_t^*\|_\mathsf U^2
\leq
C\|\bar X_t^*\|_{\mathsf H}^2
\leq
C\Theta_*$, and therefore the complementary contribution is bounded by $CT\E[\Theta_*\mathbf1_{\{\Theta_*>M\}}]$.

\smallskip
\noindent The local-gain differences satisfy the deterministic bound \eqref{eq:local_gain_convergence} so their products with the sampled continuum state are controlled by conditioning, exactly as in Theorem~\ref{thm:feedback_convergence}. The affine gain
difference is controlled directly by \eqref{eq:nonlocal_gain_convergence}. We conclude that, for every
$t\in[0,T]$,
\begin{align}
\label{eq:sampled_control_estimate}
\E\int_0^t|\Delta\alpha_s^N|_N^2ds
\leq
C\int_0^t\phi_N(s)\,ds
+C\Big(
Mr_N^2
+\E[\Theta_*\mathbf1_{\{\Theta_*>M\}}]
\Big).
\end{align}

\medskip
\noindent\textit{Step 3: Gronwall argument.}
Substituting \eqref{eq:sampled_control_estimate} into \eqref{eq:state_before_feedback} and applying Gronwall's lemma proves \eqref{eq:state_convergence_localized}. For qualitative convergence, choose any $M_N\uparrow\infty$ such that $M_Nr_N^2\to0$ and use the uniform integrability of the single variable $\Theta_*$. If $\Theta_*\in L^2$, choose $M=r_N^{-1}$ and use $\E[\Theta_*\mathbf1_{\{\Theta_*>M\}}] \leq M^{-1}\E[\Theta_*^2]$. This gives \eqref{eq:state_convergence_L2_rate}. If $\Theta_*$ is essentially bounded, take $M$ larger than its deterministic bound to obtain the sharper rate estimate.
\end{proof}

\begin{Corollary}[Convergence of the centralized optimal controls]
\label{cor:optimal_controls}
Under the assumptions of Theorem~\ref{thm:state_convergence},
\begin{align}
\label{eq:optimal_control_convergence_localized}
\E\int_0^T
|\widehat\alpha_t^{N,*}-\widetilde\alpha_t^{N,*}|_N^2dt
\leq
C\Big(
\varepsilon_{\xi,N}^2
+Mr_N^2
+\E[\Theta_*\mathbf1_{\{\Theta_*>M\}}]
\Big).
\end{align}
Consequently the control difference converges to zero in empirical
$L^2(\Omega\times[0,T])$. Under $\Theta_*\in L^2$, the right-hand side is
$C(\varepsilon_{\xi,N}^2+r_N)$; under a deterministic bound on $\Theta_*$,
it is $C(\varepsilon_{\xi,N}^2+r_N^2)$.
\end{Corollary}

\begin{proof}
Use \eqref{eq:sampled_control_estimate} with $t=T$ and insert
\eqref{eq:state_convergence_localized}.
\end{proof}

\begin{Remark}[Moment assumptions and algebraic rates]
\label{rem:state_rate_moments}
The localization loss is caused only by the product of the random kernel difference $V^N-V$ with the random conditional aggregate. If $\Theta_*\in L^p(\Omega)$ for some $p>1$, choosing $M\asymp r_N^{-2/p}$ gives a squared state-and-control error of order $r_N^{2(p-1)/p}$, apart from the initial coupling error. Hence the non-squared rate is $r_N^{(p-1)/p}$. The case $p=2$ gives the rate $r_N^{1/2}$, whereas an essential bound recovers the full rate $r_N$. This explicitly separates the deterministic discretization rate from the probabilistic moment loss.
\end{Remark}

\begin{Remark}[Meaning of the state convergence]
The preceding estimate does not identify a finite vector with an uncountable continuum state field. Each reference pair $(U_i^N,\widetilde X^{i,N,*},\widetilde\alpha^{i,N,*})$ has the continuum optimal law restricted to the $i$-th cell. Thus Theorem~\ref{thm:state_convergence} and Corollary~\ref{cor:optimal_controls} are a quantitative, cellwise conditional propagation-of-chaos statement for the centralized optimum. This interpretation is consistent with the classical coupling approach to propagation of chaos; see, for instance, \cite{Sznitman1991,Meleard1996}.
\end{Remark}

\subsection{Convergence of the value functions at the initial time}

\begin{Theorem}[Convergence of the centralized values]
\label{thm:value_convergence}
Under the assumptions of Theorem~\ref{thm:state_convergence}, there exists a constant $C>0$, independent of $N$, such that
\begin{equation}
\label{eq:value_convergence_rate}
\E\left[
|V_0^N(\xi^N)-V_0(\xi)|
\right]
\leq
C(\varepsilon_{\xi,N}+r_N).
\end{equation}
In particular, $V_0^N(\xi^N)\to V_0(\xi)$ whenever $\varepsilon_{\xi,N}\to0$.
\end{Theorem}

\begin{proof}
Using the exact block rescaling and the conditional value representation, at time zero
\begin{align*}
V_0^N(\xi^N)
={}&
\E\Big[
\langle\xi^N,M_{K_0^N}\xi^N\rangle_N
+\langle\xi^N,T_{\bar K_0^N}\xi^N\rangle_N
+2\langle Y_0^N,\xi^N\rangle_N
\Big]
+q_0^N.
\end{align*}
The pointwise bound on $K^N$, the operator bound on $T_{\bar K^N}$, the $S^\infty(\mathsf H)$-bound on $Y^N$, and Cauchy--Schwarz imply
\begin{align}
\label{eq:value_initial_coupling}
\left|
V_0^N(\xi^N)-V_0^N(\widetilde\xi^N)
\right|
\leq
C\varepsilon_{\xi,N}.
\end{align}

\noindent It remains to compare the value evaluated at the cellwise continuum samples with the continuum value. Since $\mathcal F_0^0$ is trivial, the initial Riccati coefficients are deterministic up to null sets. Moreover, $c_\xi:= \operatorname*{ess\,sup}_{u\in I} \mathbb E[|\xi|^2\mid U=u] \le \mathbb E[\Theta_*]<\infty$. The cellwise uniformity of $U_i^N$ therefore gives
\[
\begin{aligned}
&\left|
\mathbb E\langle\widetilde\xi^N,
M_{K_0^N}\widetilde\xi^N\rangle_N
-
\mathbb E\langle\xi,K_0(U)\xi\rangle
\right|\le
c_\xi\|K_0^N-K_0\|_{L^1(I)}
\le c_\xi\|K_0^N-K_0\|_{L^2(I)}
\le Cr_N.
\end{aligned}
\]
For the interaction part, conditional independence and the fact that $\bar K_0^N$ vanishes on the diagonal blocks give the exact identity
\begin{align*}
\E\left[
\langle\widetilde\xi^N,
T_{\bar K_0^N}\widetilde\xi^N\rangle_N
\,\middle|\,\Fc_0^0
\right]
=
\langle\mathsf P_N\bar\xi,
T_{\bar K_0^N}\mathsf P_N\bar\xi\rangle_{\mathsf H}
=
\langle\bar\xi,
T_{\bar K_0^N}\bar\xi\rangle_{\mathsf H}.
\end{align*}
The second equality follows because a step kernel depends only on cell
averages and its image is a step function. Hence
\[
\left|
\E\langle\widetilde\xi^N,
T_{\bar K_0^N}\widetilde\xi^N\rangle_N
-
\langle\bar\xi,T_{\bar K_0}\bar\xi\rangle_{\mathsf H}
\right|
\leq
C\|\bar K_0^N-\bar K_0\|_{L^2(I^2)}
\leq Cr_N.
\]
Similarly, $\left|
\E\langle Y_0^N,\widetilde\xi^N\rangle_N
-
\E\langle Y_0(U),\xi\rangle
\right|
\leq
C\|Y_0^N-Y_0\|_{\mathsf H}
\leq Cr_N$, and \eqref{eq:T4} gives $\E|q_0^N-\lambda_0|\leq Cr_N$. Combining these estimates with \eqref{eq:value_initial_coupling} proves \eqref{eq:value_convergence_rate}.
\end{proof}

\begin{Remark}
The value estimate is stated at time zero because the augmented Brownian sigma-field $\Fc_0^0$ is $\P$-trivial; the backward coefficients at time zero can therefore be separated from the initial conditional moments without an additional localization. A uniform-in-$t$ value estimate for random initial conditions can be obtained under a uniform conditional moment assumption, but it is not needed for convergence of the initial centralized social values.
\end{Remark}

\begin{Theorem}[Centralized discrete-to-continuum convergence]
\label{thm:centralized_complete_convergence}
Under Assumption~\ref{ass:reg_consistency}, Assumption~\ref{ass:initial_cellwise_coupling}, and $\E[\Theta_*]<\infty$, the backward Riccati objects converge as in Theorem~\ref{thm:convergence}, the optimal centralized state and control converge to their cellwise continuum copies as in Theorem~\ref{thm:state_convergence} and Corollary~\ref{cor:optimal_controls}, and the initial centralized values converge as in Theorem~\ref{thm:value_convergence}.
\end{Theorem}

\begin{proof}
This is the combination of Theorems~\ref{thm:convergence}, \ref{thm:state_convergence}, and \ref{thm:value_convergence}, together with Corollary~\ref{cor:optimal_controls}.
\end{proof}

\section{Decentralized lifted feedback and asymptotic optimality} \label{sec:decentralized_lift}

The finite-dimensional optimizer studied above is centralized, since the control of each agent depends on the complete finite state vector. We now show that the representative-agent feedback nevertheless induces a decentralized finite-population strategy whose social cost is asymptotically optimal. The only additional error is the cell-projection modulus of the affine part of the limiting feedback. 

\noindent For the continuum optimal system, set $L_t(u):=O_t(u)^{-1}U_t(u)$ and $\chi_t(u):=O_t(u)^{-1} \big((T_{V_t}\bar X_t^*)(u)+\Gamma_t(u)\big)$.
Thus
\[
\widehat\alpha_t^* =-L_t(U)X_t^*-\chi_t(U)
\]
Recall the orthogonal cell-average projection $\mathsf P_N$ from Section~\ref{sec:state_convergence}, and define $L_t^{[N]}:=\mathsf P_NL_t$ and  $\chi_t^{[N]}:=\mathsf P_N\chi_t$. If $u\in I_i^N$, we write $L_{i,t}^{[N]}:=L_t^{[N]}(u)$ and $\chi_{i,t}^{[N]}:=\chi_t^{[N]}(u)$.

\begin{Definition}[Decentralized information class]
\label{def:decentralized_controls}
For a finite control $\alpha^N$ and its associated state $X^N$, let
\[
\mathcal F_t^{i,N,\mathrm{obs}}
:=
\sigma\big(
\xi^{i,N},W_s^{i,N},W_s^0,X_s^{i,N}:0\leq s\leq t
\big)^{\P}.
\]
We denote by $\mathcal A_0^{N,\mathrm{dec}}$ the controls $\alpha^N=(\alpha^{1,N},\ldots,\alpha^{N,N})\in\mathcal A_0^N$ such that each $\alpha^{i,N}$ is progressively measurable with respect to $\mathbb F^{i,N,\mathrm{obs}}$. The decentralized social value is
\[
V_0^{N,\mathrm{dec}}(\xi^N)
:=
\inf_{\alpha^N\in\mathcal A_0^{N,\mathrm{dec}}}
J^N(0,\xi^N,\alpha^N).
\]
We write \(V_0^{N,\mathrm{cent}}(\xi^N):=V_0^N(\xi^N)\) for the centralized
value.
\end{Definition}

\noindent The inclusion $\mathcal A_0^{N,\mathrm{dec}}\subset\mathcal A_0^N$ immediately gives $V_0^{N,\mathrm{cent}}(\xi^N) \leq V_0^{N,\mathrm{dec}}(\xi^N)$.

\begin{Lemma}[Blockwise stability of the local Riccati field]
\label{lem:blockwise_K_stability}
Under Assumption~\ref{ass:reg_consistency}, there exists a deterministic constant $C>0$ such that, for $u,v$ in the same block $J_a$, $\operatorname {ess\,sup}_{(t,\omega)} |K_t(u)-K_t(v)| +\|Z^K(u)-Z^K(v)\|_{\mathrm{BMO}} \le C|u-v|$. Consequently $O^{-1}$, $U$, and $L:=O^{-1}U$ are blockwise Lipschitz in $u$, uniformly in $(t,\omega)$.
\end{Lemma}

\begin{proof}
Subtract the two local Riccati BSDEs. On the deterministic bounded ball containing $K$, the difference of the drivers is the sum of a bounded linear operator applied to $K(u)-K(v)$ and a source bounded by $C|u-v|$; the terminal difference has the same bound. The conditional estimate for this linear BSDE gives the essential-supremum estimate for $K(u)-K(v)$. Ito's formula between an arbitrary stopping time and $T$ then gives the BMO estimate for $Z^K(u)-Z^K(v)$. The claims for $O^{-1}$, $U$, and $L$ follow from their definitions and the resolvent identity.
\end{proof}

\begin{Lemma}[Projection of the limiting feedback]
\label{lem:decentralized_projection}
Under Assumption~\ref{ass:reg_consistency}, $\operatorname*{ess\,sup}_{(t,\omega)}\|L^{[N]}_t-L_t\|_{L^2(I;\R^{m\times d})} \leq C\sqrt{h_N}$. If the grid is block-compatible, the right-hand side can be replaced by $Ch_N$. Moreover, define
\[
\big(d_N^{\mathrm{aff}}\big)^2
:=
\E\int_0^T
\|\chi_t^{[N]}-\chi_t\|_{L^2(I)}^2dt,
\quad
\textit{then}
\quad
d_N^{\mathrm{aff}}\to0
\]
\end{Lemma}

\begin{proof}
By the blockwise stability lemma above, the standard parameter-stability estimate for the local Riccati BSDE gives $\operatorname*{ess\,sup}_{(t,\omega)} |K_t(u)-K_t(v)| \leq C|u-v|$, $u,v\in J_m$. The definitions of $O$ and $U$, the blockwise Lipschitz regularity of the local coefficients, and the resolvent identity show that $L=O^{-1}U$ satisfies the same estimate. On every cell contained in a single block, the cell-average error is therefore bounded by $Ch_N$. On the union $\mathcal B_N$ of cells crossing the block interfaces, we use only the uniform pointwise bound on $L$; since $|\mathcal B_N|\leq C h_N$, this part contributes $C\sqrt{h_N}$ in $L^2(I)$. If the grid is block-compatible, $\mathcal B_N$ is empty.

\noindent The operator bounds on $T_V$ and $O^{-1}$, together with the moment bounds on $\bar X^*$ and $\Gamma$, imply $\E\int_0^T\|\chi_t\|_{L^2(I)}^2dt<\infty$. Since the orthogonal projections $\mathsf P_N$ are contractions and converge strongly to the identity on $L^2(I)$, $\|\mathsf P_N\chi_t-\chi_t\|_{L^2(I)}\to0$ for $dt\otimes d\P$ a.e. $(t,\omega)$. The bound $\|\mathsf P_N\chi_t-\chi_t\|_{L^2(I)}^2
\leq4\|\chi_t\|_{L^2(I)}^2$ and dominated convergence give $d_N^{\mathrm{aff}}\to0$.
\end{proof}

\noindent We define the lifted decentralized feedback by
\begin{equation}
\label{eq:decentralized_lifted_control}
\alpha_{t}^{i,N,\mathrm{lift}}
:=
-L_{i,t}^{[N]}X_t^{i,N,\mathrm{lift}}
-\chi_{i,t}^{[N]},
\qquad 1\leq i\leq N,
\end{equation}
where $X^{N,\mathrm{lift}}$ solves the original finite state equation with this control and initial condition $\xi^N$.

\begin{Proposition}[Admissibility of the lifted feedback]
\label{prop:lifted_admissibility}
Under the standing assumptions, the closed-loop finite system associated with \eqref{eq:decentralized_lifted_control} has a unique solution in empirical $S^2$, and $\alpha^{N,\mathrm{lift}}\in\mathcal A_0^{N,\mathrm{dec}}$.
\end{Proposition}

\begin{proof}
The processes \(L_{i}^{[N]}\) and \(\chi_i^{[N]}\) are
\(\mathbb F^0\)-progressively measurable. The slope \(L^{[N]}\) is
uniformly bounded, and
\[
h_N\sum_{i=1}^N
\E\int_0^T|\chi_{i,t}^{[N]}|^2dt
=
\E\int_0^T\|\mathsf P_N\chi_t\|_{L^2(I)}^2dt
\leq
\E\int_0^T\|\chi_t\|_{L^2(I)}^2dt.
\]
The closed-loop equation is therefore a linear finite-dimensional SDE with square integrable affine term and uniformly bounded linear coefficients. The usual SDE estimate gives existence, uniqueness, and the required empirical $S^2$-bound. Formula \eqref{eq:decentralized_lifted_control} uses only the common-noise coefficients and the current individual state, so it is progressively measurable with respect to $\mathbb F^{i,N,\mathrm{obs}}$.
\end{proof}

\begin{Theorem}[Cellwise convergence of the lifted decentralized system]
\label{thm:decentralized_state_convergence}
Assume the hypotheses of Theorem~\ref{thm:state_convergence}. Then
\begin{align}
\E\left[
\sup_{t\in[0,T]}
|X_t^{N,\mathrm{lift}}-\widetilde X_t^{N,*}|_N^2
\right]
+
\E\int_0^T
|\alpha_t^{N,\mathrm{lift}}-\widetilde\alpha_t^{N,*}|_N^2dt
\leq
C\left(
\varepsilon_{\xi,N}^2+r_N^2+
\big(d_N^{\mathrm{aff}}\big)^2
\right).
\label{eq:decentralized_state_control_rate}
\end{align}
\end{Theorem}

\begin{proof}
Set
\(\Delta X^N:=X^{N,\mathrm{lift}}-\widetilde X^{N,*}\). The synchronous
state-equation estimate used in the proof of
Theorem~\ref{thm:state_convergence}, together with the interaction
decomposition \eqref{eq:sampled_interaction_decomposition}, gives
\begin{align*}
\E\left[
\sup_{s\leq t}|\Delta X_s^N|_N^2
\right]
\leq{}&
C\varepsilon_{\xi,N}^2
+C\int_0^t
\E\left[
\sup_{q\leq s}|\Delta X_q^N|_N^2
\right]ds
+C\E\int_0^t
|\alpha_s^{N,\mathrm{lift}}-
\widetilde\alpha_s^{N,*}|_N^2ds
+Cr_N^2.
\end{align*}
Here the conditional empirical fluctuation is of order $h_N$, and all coefficient, kernel, and common-noise errors are contained in $r_N$. For $u\in I_i^N$, \[\alpha_t^{i,N,\mathrm{lift}}
-\widetilde\alpha_t^{i,N,*}
={}
-L_{i,t}^{[N]}
\big(
X_t^{i,N,\mathrm{lift}}-\widetilde X_t^{i,N,*}
\big)
-
\big(L_t^{[N]}(U_i^N)-L_t(U_i^N)\big)
\widetilde X_t^{i,N,*}
-
\big(\chi_t^{[N]}(U_i^N)-\chi_t(U_i^N)\big)\] 
The first term is controlled by the uniform bound on $L^{[N]}$. For the second, conditioning on $\mathcal F_t^0$ and using the conditional moment variable $\Theta_*$ yields
\begin{align*}
\E\left[
h_N\sum_i
|L_t^{[N]}(U_i^N)-L_t(U_i^N)|^2
|\widetilde X_t^{i,N,*}|^2
\right]
\leq
\E\left[
\theta_t^*
\|L_t^{[N]}-L_t\|_{L^2(I)}^2
\right]
\leq
Ch_N\E[\Theta_*]
\leq Cr_N^2
\end{align*}
Since the labels are uniform on their cells, the last term satisfies $\E\left[
h_N\sum_i
|\chi_t^{[N]}(U_i^N)-\chi_t(U_i^N)|^2
\right]
=
\E\|\chi_t^{[N]}-\chi_t\|_{L^2(I)}^2$.
Consequently,
\begin{align*}
\E\int_0^t
|\alpha_s^{N,\mathrm{lift}}-
\widetilde\alpha_s^{N,*}|_N^2ds
\leq{}&
C\int_0^t
\E\left[
\sup_{q\leq s}|\Delta X_q^N|_N^2
\right]ds
+C\left(r_N^2+\big(d_N^{\mathrm{aff}}\big)^2\right).
\end{align*}
Substitution into the state estimate and Gronwall's lemma prove both terms
in \eqref{eq:decentralized_state_control_rate}.
\end{proof}

\begin{Theorem}[Cost consistency of the lifted feedback]
\label{thm:decentralized_cost_consistency}
Under the assumptions of Theorem~\ref{thm:decentralized_state_convergence},
\[
\left|
J^N(0,\xi^N,\alpha^{N,\mathrm{lift}})
-V_0(\xi)
\right|
\leq
C\left(
\varepsilon_{\xi,N}+r_N+d_N^{\mathrm{aff}}
\right).
\]
\end{Theorem}

\begin{proof}
The local running, terminal, and control costs are compared by adding and subtracting the corresponding quantities evaluated at the cellwise copies. The uniform coefficient bounds, Cauchy--Schwarz, and Theorem~\ref{thm:decentralized_state_convergence} control the state and control differences by $C(\varepsilon_{\xi,N}+r_N+d_N^{\mathrm{aff}})$. The local coefficient errors are bounded in the same way as in the state comparison, using the conditional moment bound $\E[\Theta_*]<\infty$.

\noindent We detail the only genuinely nonlocal term. Set $\vartheta_t :=\operatorname*{ess\,sup}_{u\in I} \mathbb E[|X_t^*|^2\mid\mathcal F_t^0,U=u]$. By the definition of $\Theta_*$, we have $\vartheta_t\le\Theta_*$ for $dt\otimes d\mathbb P$-almost every $(t,\omega)$, and $\vartheta_T\le\Theta_*$ almost surely. Let $G^N$ denote either the running or terminal cost kernel, and let $\bar x^N$ be the cell average of the corresponding limiting conditional mean. Conditional independence gives
\begin{align*}
&
\E\left[
h_N^2\sum_{i\neq j}
\left\langle
\widetilde X_t^{i,N,*},
G_{ij,t}^N\widetilde X_t^{j,N,*}
\right\rangle
\,\middle|\,\mathcal F_t^0
\right]
=
\left\langle
\mathsf P_N\bar X_t^*,
T_{\Pi_N^\circ G_t^N}
\mathsf P_N\bar X_t^*
\right\rangle_{\mathsf H}
=
\left\langle
\bar X_t^*,
T_{\Pi_N^\circ G_t^N}\bar X_t^*
\right\rangle_{\mathsf H}.
\end{align*}
Moreover, the boundedness of the limiting kernel gives
\[
\|\Pi_N^\circ G_t^N-G_t\|_{L^2(I^2)}
\leq
\|G_t^N-G_t\|_{L^2(I^2)}
+\|\mathbf1_{\mathcal D_N}G_t\|_{L^2(I^2)}
\leq
C\big(\varepsilon_N^{\mathrm{ker}}+\sqrt{h_N}\big).
\]
The diagonal part of the finite sum is negligible as well. Indeed, conditional on $\mathcal F_t^0$,
\[
h_N^2\sum_i
|G_{ii,t}^N|
\E\left[
|\widetilde X_t^{i,N,*}|^2
\,\middle|\,\mathcal F_t^0
\right]
\leq
C\vartheta_t\sqrt{h_N}\|G_t^N\|_{L^2(I^2)}.
\]
After integration and expectation, these two errors are bounded by $Cr_N$. Replacing the cellwise copies by $X^{N,\mathrm{lift}}$ is controlled by the uniform operator bounds on the finite cost kernels and the state estimate. The terminal interaction term is estimated in the same way, using
$\vartheta_T\le\Theta_*$. Collecting the local, control, and interaction estimates proves the claim.
\end{proof}

\begin{Corollary}[Asymptotic optimality gap for the centralized problem]
\label{cor:centralized_lifted_optimality_gap}
Under the assumptions of Theorem~\ref{thm:decentralized_cost_consistency},
\begin{equation}
\label{eq:centralized_lifted_optimality_gap}
0
\leq
J^N(0,\xi^N,\alpha^{N,\mathrm{lift}})
-V_0^{N,\mathrm{cent}}(\xi^N)
\leq
C\left(
\varepsilon_{\xi,N}+r_N+d_N^{\mathrm{aff}}
\right).
\end{equation}
In particular, the lifted limiting feedback is asymptotically optimal for
the original centralized finite-population problem.
\end{Corollary}

\begin{proof}
The lower bound follows from the optimality of the centralized finite
feedback. For the upper bound, insert the continuum value between the two
finite costs: $J^N(0,\xi^N,\alpha^{N,\mathrm{lift}})
-V_0^{N,\mathrm{cent}}(\xi^N)
\leq
\left|
J^N(0,\xi^N,\alpha^{N,\mathrm{lift}})-V_0(\xi)
\right|
+
\left|
V_0^{N,\mathrm{cent}}(\xi^N)-V_0(\xi)
\right|$. The first term is controlled by Theorem~\ref{thm:decentralized_cost_consistency}, and the second one by Theorem~\ref{thm:value_convergence}.
\end{proof}

\begin{Theorem}[Asymptotic optimality under decentralized information]
\label{thm:decentralized_asymptotic_optimality}
Under the preceding assumptions,
\[
V_0^{N,\mathrm{dec}}(\xi^N)
\longrightarrow
V_0(\xi),
\qquad
J^N(0,\xi^N,\alpha^{N,\mathrm{lift}})
-V_0^{N,\mathrm{dec}}(\xi^N)
\longrightarrow0.
\]
More precisely,
\begin{align*}
&
\left|
V_0^{N,\mathrm{dec}}(\xi^N)-V_0(\xi)
\right|
+
\left|
J^N(0,\xi^N,\alpha^{N,\mathrm{lift}})
-V_0^{N,\mathrm{dec}}(\xi^N)
\right|
\leq
C\left(
\varepsilon_{\xi,N}+r_N+d_N^{\mathrm{aff}}
\right).
\end{align*}
\end{Theorem}

\begin{proof}
Since the centralized information class is larger and the lifted feedback is decentralized,
\[
V_0^{N,\mathrm{cent}}(\xi^N)
\leq
V_0^{N,\mathrm{dec}}(\xi^N)
\leq
J^N(0,\xi^N,\alpha^{N,\mathrm{lift}}).
\]
The left-hand term converges to $V_0(\xi)$ by Theorem~\ref{thm:value_convergence}, with error $C(\varepsilon_{\xi,N}+r_N)$. The right-hand term satisfies the estimate
of Theorem~\ref{thm:decentralized_cost_consistency}. The two conclusions follow from the squeeze argument, and the displayed quantitative estimate follows by subtracting the two extreme quantities.
\end{proof}

\begin{Remark}[Algebraic rates]
\label{rem:decentralized_algebraic_rate}
The modulus $d_N^{\mathrm{aff}}$ is qualitative under the minimal $L^2(I)$-regularity of the affine feedback field $\chi$. If, for some $\beta>0$, one assumes the targeted estimate $d_N^{\mathrm{aff}}\leq CN^{-\beta}$, then the decentralized value and optimality-gap estimates inherit the non-squared rate $O\big( \varepsilon_{\xi,N}+r_N+N^{-\beta} \big)$. 

\noindent For example, on a block-compatible grid, the bound $\E\int_0^T \sum_{m=1}^M \|\partial_u\chi_t\|_{L^2(J_m)}^2dt <\infty$ implies \(d_N^{\mathrm{aff}}=O(h_N)\) by the cellwise Poincar\'e inequality. Thus the stronger hypothesis needed for an algebraic decentralized rate is isolated from the assumptions of the centralized convergence theorem.
\end{Remark}

\section{Conclusion}
We derived the exact finite-population stochastic Riccati system for a heterogeneous linear--quadratic social planner problem with common noise and established its quantitative convergence to the continuum system. The analysis retains the diagonal corrections, the off-diagonal band, and the common-noise martingale integrands. The backward estimates were propagated to the optimal feedbacks, cellwise coupled trajectories, controls, and initial social values. We also constructed a decentralized finite-population lift of the limiting feedback and proved its asymptotic optimality. The error decomposition identifies precisely which additional regularity assumptions yield algebraic rates. Natural extensions include controlled common-noise volatility, sparse interaction regimes, and numerical approximation of the resulting operator-valued stochastic Riccati equations. A further direction is to incorporate self-exciting jump shocks, motivated by their use in models of weather derivatives and aggregate losses; see \cite{AhmedEyjolfsson2026} for a numerical pricing approach in this setting. Extending the present finite-population analysis to such models would require a separate treatment of the jump dynamics.

\appendix

\section{Proof of the measurable representation result}\label{app:P0}
\begin{proof}[Proof of Proposition~\ref{prop:measurable_rpz}]
Set $\mathsf Z=(U,\xi,W,W^0)$, with values in the standard Borel space $\mathsf E :=I\times\mathbb R^d\times C([0,T];\mathbb R) \times C([0,T];\mathbb R^{d_0})$. By the definition of admissibility, the control has a jointly Borel, non-anticipative representative. The $\mathbb F^0$-progressive coefficients and the frozen field $m$ likewise have jointly Borel, non-anticipative representatives on the canonical common-noise path space. Standard Picard iteration for the frozen linear SDE can be implemented using simple predictable approximations of the stochastic integrals. Every iterate is then a Borel function of $\mathsf Z$; its $S^2$-limit admits an almost surely convergent subsequence, yielding a Borel map $x^{m,\alpha}:\mathsf E\longrightarrow C([t,T];\mathbb R^d)$ such that $X^{m,\alpha}=x^{m,\alpha}(\mathsf Z)$ almost surely. Pathwise uniqueness identifies this limit with the unique strong solution. Repeating the factorization with the stopped input at time $s$ gives the non-anticipative maps $x_s^{m,\alpha}$. This is the standard measurable-representation argument; see also \cite[Theorem~3.5]{DeFeoMekkaoui2025}.
\end{proof}

\section{Proof of Theorem \ref{thm:wellposed_limit_state}}\label{app:P1}

Let $m\in\mathcal H_t= H_{\mathbb F^0}^2([t,T];\mathsf H)$, the space of field $\mathbb F^0$-progressive with squared norme $\E\int_t^T\|m_s\|_{\mathsf H}^2ds<\infty$. We consider the frozen equation
\[
\left\{
\begin{aligned}
dX_s^m
&=
\Big[
A_s(U)+B_s(U)X_s^m+(T_{G_s^B}m_s)(U)+C_s(U)\alpha_s
\Big]\,ds
\\
&\quad+
\Big[
D_s(U)+E_s(U)X_s^m+(T_{G_s^E}m_s)(U)+F_s(U)\alpha_s
\Big]\,dW_s
+\Sigma_s^0(U)\,dW_s^0,
\qquad s\in[t,T],
\\
X_t^m&=\xi.
\end{aligned}
\right.
\tag{A.1}
\]
By Assumption \ref{ass:forward_coeff}, this is a linear SDE with square-integrable affine terms and uniformly bounded linear coefficients. Hence it admits a unique strong solution $X^m\in S^2([t,T];\R^d)$.

\noindent We now define $\Phi(m)_s(u) := \E[X_s^m\mid \Fc_s^0,U=u]$, $(s,u)\in[t,T]\times I$.
Since $X^m = x^{m,\alpha}(U,\xi,W,W^0)$ by Proposition 
\ref{prop:measurable_rpz}, the map $(u,\omega)\mapsto X_s^m(\omega)$ 
is jointly measurable. The conditional expectation 
$\Phi(m)_s(u) = \E[X_s^m\mid\Fc_s^0,U=u]$ then admits a jointly 
measurable version by the disintegration theorem on Polish spaces. Moreover, by Jensen's inequality,
\[
\E\left[\int_I|\Phi(m)_s(u)|^2\,du\right]
\le
\E[|X_s^m|^2],
\]
and consequently $\Phi(m)\in\mathcal H_t$. Thus $\Phi$ is a well-defined map from $\mathcal H_t$ into itself.

\noindent We now prove that $\Phi$ is a contraction on small time intervals.
Let $m,n\in\mathcal H_t$, and set $\Delta X_s:=X_s^m-X_s^n$, $\Delta m_s:=m_s-n_s$.
Then
\[
\left\{
\begin{aligned}
d\Delta X_s
&=
\Big[
B_s(U)\Delta X_s+(T_{G_s^B}\Delta m_s)(U)
\Big]\,ds
\\
&\quad+
\Big[
E_s(U)\Delta X_s+(T_{G_s^E}\Delta m_s)(U)
\Big]\,dW_s,
\\
\Delta X_t&=0.
\end{aligned}
\right.
\]
Using standard estimates for linear SDEs, together with the uniform boundedness of $B,E$ and of the operators $T_{G^B}$ and $T_{G^E}$, there exists a constant $C_T>0$, independent of $m,n$, such that, for all $r\in[t,T]$,
\[
\E|\Delta X_r|^2
\le
C_T\int_t^r \E|\Delta X_s|^2\,ds
+
C_T\int_t^r
\E\left[\int_I |\Delta m_s(u)|^2\,du\right]ds.
\]
By Gronwall's lemma, $\E|\Delta X_r|^2 \le C_T\int_t^r \E\left[\int_I |\Delta m_s(u)|^2\,du\right]ds$.
Integrating in $r$ over $[t,T]$ gives $\E\left[\int_t^T |\Delta X_r|^2\,dr\right] \le C_T(T-t)\|m-n\|_{\mathcal H_t}^2.$
On the other hand, by Jensen's inequality, $\|\Phi(m)-\Phi(n)\|_{\mathcal H_t}^2
\le \E\left[\int_t^T |\Delta X_s|^2\,ds\right]$.
Combining this estimate with the one above, we obtain $\|\Phi(m)-\Phi(n)\|_{\mathcal H_t}^2 \le C_T(T-t)\|m-n\|_{\mathcal H_t}^2.$
Hence $\Phi$ is a contraction whenever $T-t$ is small enough.

\noindent By Banach's fixed-point theorem, on every interval of length $\delta>0$ small enough, there exists a unique fixed point $m^\ast\in\mathcal H_t$ of $\Phi$. If $X:=X^{m^\ast}$, then $m_s^\ast(u)=\E[X_s\mid\Fc_s^0,U=u]$, so $X$ solves the full state equation \eqref{eq:state_equations}. Local uniqueness follows from uniqueness of the fixed point and pathwise uniqueness for the frozen SDE.

\noindent The global result on $[t,T]$ follows by iterating the previous argument over finitely many subintervals of length at most $\delta$. It remains to prove the a priori estimate. By applying standard SDE estimates to \eqref{eq:state_equations}, using the boundedness of $B,E,C,F$, the operator bounds on $T_{G^B}$ and $T_{G^E}$, Jensen's inequality for the conditional mean field, and the Burkholder--Davis--Gundy (BDG) inequality, we obtain for all $s\in[t,T]$
\[
\begin{aligned}
\E\left[\sup_{t\le q\le s}|X_q^\alpha|^2\mid \Fc_t^0\right]
\le
&C
\left(
\E\left[ |\xi|^2 +\int_t^T\Big( |\alpha_s|^2 +\|A_s\|_{\mathsf H}^2 +\|D_s\|_{\mathsf H}^2 +\|\Sigma_s^0\|_{L^2(I;\mathbb R^{d\times d_0})}^2 \Big)ds \,\middle|\,\mathcal F_t^0 \right]
\right)
+
C\int_t^s
\E\left[\sup_{t\le q\le r}|X_q^\alpha|^2\mid \Fc_t^0\right]dr.
\end{aligned}
\]
Gronwall's lemma yields
\[
\E\left[ \sup_{t\le s\le T}|X_s^\alpha|^2 \,\middle|\,\mathcal F_t^0 \right] \le C_T\E\left[ |\xi|^2 +\int_t^T\Big( |\alpha_s|^2 +\|A_s\|_{\mathsf H}^2 +\|D_s\|_{\mathsf H}^2 +\|\Sigma_s^0\|_{L^2(I;\mathbb R^{d\times d_0})}^2 \Big)ds \,\middle|\,\mathcal F_t^0 \right], \qquad \P\text{-a.s.}
\]
Finally, the existence of a jointly measurable version of $\bar X_s^\alpha(u)=\E[X_s^\alpha\mid\Fc_s^0,U=u]$
follows from the measurable representation in Proposition \ref{prop:measurable_rpz} and the fixed-point construction above. Hence the conditional marked law $\mu_s^{0,\alpha}:=\Lc((U,X_s^\alpha)\mid\Fc_s^0)$ is well defined as a $\Pc_2^\lambda(I\times\R^d)$-valued $\F^0$-adapted process. CQFD

\section{Proof of Theorem \ref{thm:wellposed_finite_state}}\label{app:P2}

Fix $N\ge1$, $t\in[0,T]$, $\xi^N\in\mathcal I_t^N$, and $\alpha^N\in\Ac_t^N$.
We work on the Banach space $\mathbb S_t^N=S_{\mathbb F^N}^2([t,T];\mathbb H_N)$. Let $Y^N\in\mathbb S_t^N$. We define $X^N=\Psi^N(Y^N)$ by
\[
\begin{aligned}
X_s^{i,N}
:=
\xi^{i,N}
&+
\int_t^s
\Big[
A_r^{i,N}
+
B_r^{i,N}Y_r^{i,N}
+
\frac1N\sum_{j=1}^N G_{B,r}^N(u_i^N,u_j^N)Y_r^{j,N}
+
C_r^{i,N}\alpha_r^{i,N}
\Big]dr
\\
&+
\int_t^s
\Big[
D_r^{i,N}
+
E_r^{i,N}Y_r^{i,N}
+
\frac1N\sum_{j=1}^N G_{E,r}^N(u_i^N,u_j^N)Y_r^{j,N}
+
F_r^{i,N}\alpha_r^{i,N}
\Big]dW_r^{i,N}
+
\int_t^s \Sigma_{0,r}^{i,N}\,dW_r^0,
\qquad i=1,\dots,N
\end{aligned}
\]
The stochastic integrals are well defined by Assumption \ref{ass:finite_forward_coeff}. Using Cauchy--Schwarz, the BDG inequality, and the uniform bounds on the coefficients, there exists a constant $C_T>0$, independent of $N$, such that
\[
\begin{aligned}
\E\left[\sup_{t\le q\le s}|X_q^N|_N^2\right]
\le
C_T\Bigg(
&\E|\xi^N|_N^2
+
\E\int_t^s |A_r^N|_N^2\,dr
+
\E\int_t^s |D_r^N|_N^2\,dr
\\
&+
\E\int_t^s \|\Sigma_{r}^{0,N}\|_{N,\mathrm{HS}}^2\,dr
+
\E\int_t^s |\alpha_r^N|_N^2\,dr 
+
\int_t^s
\E\left[\sup_{t\le \ell\le r}|Y_\ell^N|_N^2\right]dr
\Bigg)
\end{aligned}
\]
Therefore $\Psi^N(Y^N)\in \mathbb S_t^N$, so $\Psi^N$ is well defined.

\noindent We now prove that $\Psi^N$ is a contraction on sufficiently small intervals. Let $Y^N,Z^N\in\mathbb S_t^N$, and set $X^N:=\Psi^N(Y^N)$, $\widetilde X^N:=\Psi^N(Z^N)$, $\Delta X^N:=X^N-\widetilde X^N$, $\Delta Y^N:=Y^N-Z^N$.
Then, for each $i=1,\dots,N$,
\[
\begin{aligned}
\Delta X_s^{i,N}
=
&\int_t^s
\Big[
B_r^{i,N}\Delta Y_r^{i,N}
+
\frac1N\sum_{j=1}^N G_{B,r}^N(u_i^N,u_j^N)\Delta Y_r^{j,N}
\Big]dr
+
\int_t^s
\Big[
E_r^{i,N}\Delta Y_r^{i,N}
+
\frac1N\sum_{j=1}^N G_{E,r}^N(u_i^N,u_j^N)\Delta Y_r^{j,N}
\Big]dW_r^{i,N}.
\end{aligned}
\]
Using again Cauchy--Schwarz, BDG, and the uniform bounds on the local and interaction operators, we get, for all $s\in[t,T]$,
\[
\E\left[\sup_{t\le q\le s}|\Delta X_q^N|_N^2\right]
\le
C_T
\int_t^s
\E\left[\sup_{t\le \ell\le r}|\Delta Y_\ell^N|_N^2\right]dr.
\]
Hence, on an interval $[t,t+\delta]$ with $C_T\delta<1$, $\Psi^N$ is a contraction on $\mathbb S_t^N$. By Banach's fixed-point theorem, there exists a unique local solution. Repeating the argument on finitely many subintervals yields a unique global solution on $[t,T]$.

\noindent Uniqueness follows from the same estimate. Indeed, if $X^N$ and $\widetilde X^N$ are two solutions of \eqref{eq:finite_N_state} with the same initial condition and control, then $\Delta X^N$ satisfies
\[
\E\left[\sup_{t\le q\le s}|\Delta X_q^N|_N^2\right]
\le
C_T\int_t^s
\E\left[\sup_{t\le \ell\le r}|\Delta X_\ell^N|_N^2\right]\,dr.
\]
By Gronwall's lemma, $\E\left[\sup_{t\le s\le T}|\Delta X_s^N|_N^2\right]=0$, and therefore $X^N=\widetilde X^N$ up to indistinguishability.

\noindent It remains to prove the a priori estimate. Applying the previous estimate to the fixed point $X^N=\Psi^N(X^N)$ gives, with Gronwall's lemma, 
\[
\begin{aligned}
\E\left[\sup_{t\le s\le T}|X_s^N|_N^2\right]
\le
C_T
\Bigg(
&\E|\xi^N|_N^2
+
\E\int_t^T |A_s^N|_N^2\,ds
+
\E\int_t^T |D_s^N|_N^2\,ds
+
\E\int_t^T \|\Sigma_{s}^{0,N}\|_{N,\mathrm{HS}}^2\,ds
+
\E\int_t^T |\alpha_s^N|_N^2\,ds
\Bigg).
\end{aligned}
\]

\noindent Finally, the measurable representation of $X^N$ follows from the Picard construction above: each Picard iterate is a Borel functional of $(\xi^N,W^{1,N},\dots,W^{N,N},W^0)$, and the limit in $S^2_{\F^N}([t,T];\H_N)$ admits a subsequence converging almost surely. Hence the solution can be chosen as a measurable functional of the input data. 

\section{Proof of Proposition \ref{prop:fundamental_relation_continuum}}\label{app:fundamental_relation_continuum}

\begin{proof}
Fix $t\in[0,T]$, $\xi\in\Ic_t$, and an admissible control $\alpha\in\Ac_t$. 
By the measurable representation result, we may work with a jointly measurable family $(X^u,\alpha^u)_{u\in I}$ such that, conditionally on the label $U=u$, the representative state and control are denoted by $X^u$ and $\alpha^u$. 
More precisely, $X_s=X_s^U$, $\alpha_s=\alpha_s^U$, $\P$-a.s., and $\bar X_s(u)=\E[X_s^u\mid \Fc_s^0]$, $\bar\alpha_s(u)=\E[\alpha_s^u\mid \Fc_s^0]$.
To keep the notation light, in the computations below we fix $u\in I$ and write, by abuse of notation, $X_s=X_s^u$, $\alpha_s=\alpha_s^u$, while keeping the dependence on $u$ in the coefficients $K_s(u)$, $\bar K_s(u,v)$ and $Y_s(u)$. The scalar process $\lambda$ will be added only after integration with respect to the label variable. We also define $M_s^B(u):=\int_I G_s^B(u,v)\bar X_s(v)\,dv$ and $M_s^E(u):=\int_I G_s^E(u,v)\bar X_s(v)\,dv$.
Then the state equation can be written as
\[
dX_s
=
b_s^\alpha(u)\,ds+\sigma_s^\alpha(u)\,dW_s+\Sigma_s^0(u)\,dW_s^0,
\]
where $b_s^\alpha(u) := A_s(u)+B_s(u)X_s^u+M_s^B(u)+C_s(u)\alpha_s^u$, and $\sigma_s^\alpha(u):=D_s(u)+E_s(u)X_s^u+M_s^E(u)+F_s(u)\alpha_s^u$.
Moreover, taking conditional expectation with respect to $\Fc_s^0$ and $U=u$, the conditional mean field satisfies $d\bar X_s(u)=\bar b_s^\alpha(u)\,ds+\Sigma_s^0(u)\,dW_s^0$,
with $\bar b_s^\alpha(u):= A_s(u)+B_s(u)\bar X_s(u)+M_s^B(u)+C_s(u)\bar\alpha_s(u)$, and
$\bar\alpha_s(u):=\E[\alpha_s\mid \Fc_s^0,U=u]$.
For $u\in I$, we define the random field (in the remaining computations, we suppress the superscript $u$ on $X^u$ and $\alpha^u$) :
\begin{align}
\mathcal V_s(u)
:=
&\;
\langle X_s,K_s(u)X_s\rangle
+
\left\langle
\bar X_s(u),
\int_I \bar K_s(u,v)\bar X_s(v)\,dv
\right\rangle
+
2\langle Y_s(u),X_s\rangle.
\label{eq:app_random_field_V}
\end{align}
We also introduce
\begin{align}
\mathcal S_s^\alpha(u)
:=
\mathcal V_s(u)
+
\int_t^s
\Big(
&\langle X_r,Q_r(u)X_r\rangle
+
\left\langle
\bar X_r(u),
\int_I G_r^Q(u,v)\bar X_r(v)\,dv
\right\rangle +
\langle \alpha_r+\iota_r(u),R_r(u)(\alpha_r+\iota_r(u))\rangle
\Big)\,dr.
\label{eq:app_S_alpha}
\end{align}

\noindent Applying Itô's formula to $\mathcal S_s^\alpha(u)$ gives $d\mathcal S_s^\alpha(u)=D_s^{u,\alpha}\,ds+dM_s^{u,\alpha}$, where $M^{u,\alpha}$ is a local martingale and where the drift is given by
\begin{align}
D_s^{u,\alpha}
=
&\;
\langle b_s^\alpha(u),K_s(u)X_s\rangle
+
\langle X_s,K_s(u)b_s^\alpha(u)\rangle
-
\langle X_s,\dot K_s(u)X_s\rangle
+
\langle \sigma_s^\alpha(u),K_s(u)\sigma_s^\alpha(u)\rangle
+
\langle \Sigma_s^0(u),K_s(u)\Sigma_s^0(u)\rangle_{\rm F}
\\
&
+2\langle \Sigma_s^0(u),Z_s^K(u)X_s\rangle_{\rm F}
\notag+
\left\langle
\bar b_s^\alpha(u),
\int_I \bar K_s(u,v)\bar X_s(v)\,dv
\right\rangle
+
\left\langle
\bar X_s(u),
\int_I \bar K_s(u,v)\bar b_s^\alpha(v)\,dv
\right\rangle
-
\left\langle
\bar X_s(u),
\int_I \dot{\bar K}_s(u,v)\bar X_s(v)\,dv
\right\rangle
\\
&
+
\int_I
\left\langle
\Sigma_s^0(u),
\bar K_s(u,v)\Sigma_s^0(v)
\right\rangle_{\rm F}\,dv
\notag+
\int_I
\left\langle
\Sigma_s^0(u),
Z_s^{\bar K}(u,v)\bar X_s(v)
\right\rangle_{\rm F}\,dv
+
\int_I
\left\langle
\bar X_s(u),
Z_s^{\bar K}(u,v)\Sigma_s^0(v)
\right\rangle_{\rm F}\,dv
+
2\langle -\dot Y_s(u),X_s\rangle
\\
&
+2\langle Y_s(u),b_s^\alpha(u)\rangle
+
2\langle Z_s^Y(u),\Sigma_s^0(u)\rangle_{\rm F}
\notag+
\langle X_s,Q_s(u)X_s\rangle
+
\left\langle
\bar X_s(u),
\int_I G_s^Q(u,v)\bar X_s(v)\,dv
\right\rangle
+
\langle \alpha_s+\iota_s(u),R_s(u)(\alpha_s+\iota_s(u))\rangle .
\end{align}

\noindent Here $\dot K,\dot{\bar K}$ and $\dot Y$ denote the drift components appearing in the corresponding backward equations, namely\\ 
$dK_s=-\dot K_s\,ds+Z_s^K\,dW_s^0$, $d\bar K_s=-\dot{\bar K}_s\,ds+Z_s^{\bar K}\,dW_s^0$, $dY_s=-\dot Y_s\,ds+Z_s^Y\,dW_s^0$.

\noindent Let $\mathfrak c_s^\alpha$ denote the running-cost integrand in \eqref{eq:conditional_cost}. Set $m_s=\bar X_s^\alpha$ and $\mathcal R_s^\alpha :=U_s(U)X_s^\alpha+(T_{V_s}m_s)(U)+\Gamma_s(U)$, with $\alpha_s^{\star,\alpha}=-O_s(U)^{-1}\mathcal R_s^\alpha$. The algebraic square completion is
\begin{align}
\langle\alpha_s,O_s(U)\alpha_s\rangle
 +2\langle\alpha_s,\mathcal R_s^\alpha\rangle
={}&\langle\alpha_s-\alpha_s^{\star,\alpha},
 O_s(U)(\alpha_s-\alpha_s^{\star,\alpha})\rangle
 -\langle\mathcal R_s^\alpha,O_s(U)^{-1}\mathcal R_s^\alpha\rangle.
\label{eq:app_square_completion}
\end{align}
The identification of these control terms in the preceding drift calculation is made after conditioning on $\mathcal F_s^0$ and integrating the label. In particular, whenever a multiplier is $\mathcal F_s^0\vee\sigma(U)$-measurable, we use $\mathbb E[X_s^\alpha\mid\mathcal F_s^0,U=u]=m_s(u)$ and
$\mathbb E[\alpha_s\mid\mathcal F_s^0,U=u]=\bar\alpha_s(u)$. Fubini's theorem and the kernel symmetries then allow the terms
in $\bar\alpha_s$ and $m_s$ to be collected. Substitution of the
three backward equations cancels all state-dependent non-square
terms. The remaining scalar term is $\ell_s$, cancelled by the
drift $-\ell_s$ of $\lambda$. To justify the expectation step, use the scalar semimartingale
\[
\begin{aligned}
\mathcal T_s^\alpha:={}&
 (X_s^\alpha)^\top K_s(U)X_s^\alpha
 +2Y_s(U)^\top X_s^\alpha
 +\langle m_s,T_{\bar K_s}m_s\rangle_{\mathsf H}
 +\lambda_s+\int_t^s\mathfrak c_q^\alpha\,dq.
\end{aligned}
\]
Applying the same Ito calculation, with the interaction quadratic term integrated in its two labels, gives $d\mathcal T_s^\alpha=\mathfrak D_s^\alpha ds+dM_s^\alpha$, where M is initially a local martingale and
\begin{equation}
\mathbb E[\mathfrak D_s^\alpha\mid\mathcal F_s^0]
 =\mathbb E\left[
 \langle\alpha_s-\alpha_s^{\star,\alpha},
 O_s(U)(\alpha_s-\alpha_s^{\star,\alpha})\rangle
 \,\middle|\,\mathcal F_s^0\right].
\label{eq:conditional_drift_cancellation}
\end{equation}
All drift terms are integrable. Indeed, K and its driver are bounded, $T_{\bar K}$ and $T_{\mathcal F(s,\bar K_s)}$ have bounded operator norms, and the terms involving $Z^K,Z^{\bar K},Z^Y$ are integrable by Cauchy--Schwarz and the boundedness of $\Sigma^0$. Also $m\in S^2_{\mathbb F^0}(\mathsf H)$: its conditional mean equation has drift $A+Bm+T_{G^B}m+C\bar\alpha$ and common-noise coefficient $\Sigma^0$, with $\bar\alpha\in H^2(\mathsf U)$. If $\varphi^Y$ denotes the full driver of Y, then $\varphi^Y\in H^2(\mathsf H)$ and the pointwise conditional representation and Doob's inequality give $\int_I\mathbb E\sup_{s\le T}|Y_s(u)|^2du \le C_T\mathbb E\int_0^T\|\varphi_s^Y\|_{\mathsf H}^2ds<\infty$. Thus $Y(U)$ and $X^\alpha$ are in $S^2$, and their product is of class (D). The bounds on K and $T_{\bar K}$ show the same property for the quadratic terms. The scalar component is of class (D), also in the enlarged filtration: common-noise martingales remain martingales after adjoining the independent private inputs. Finally the accumulated cost has integrable total variation. Consequently $\mathcal T^\alpha$ is of class (D). Since its drift has integrable total variation, M is a uniformly integrable martingale.

We may therefore take conditional expectations in the integrated
identity. Using the tower property yields
\begin{align}
\mathbb E[\mathcal T_T^\alpha-\mathcal T_t^\alpha
 \mid\mathcal F_t^0]
=\mathbb E\left[\int_t^T
 \langle\alpha_s-\alpha_s^{\star,\alpha},
 O_s(U)(\alpha_s-\alpha_s^{\star,\alpha})\rangle ds
 \,\middle|\,\mathcal F_t^0\right].
\label{eq:app_integrated_identity}
\end{align}
The terminal conditions and disintegration in U give $\mathbb E[\mathcal T_T^\alpha\mid\mathcal F_t^0]=J(t,\xi,\alpha)$, whereas
$\mathbb E[\mathcal T_t^\alpha\mid\mathcal F_t^0]=\mathcal V_t(\xi)$. This proves the fundamental relation.
\noindent Combining the previous identities yields the desired fundamental relation.
\end{proof}

\section{Proof of Proposition \ref{prop:fundamental_relation_discrete}}
\label{app:fundamental_relation_discrete}

\begin{proof}
Fix $t\in[0,T]$, $\xi^N\in\Ic_t^N$, and an admissible control 
$\alpha^N\in\Ac_t^N$. Throughout this proof, we work with the global 
formulation \eqref{eq:discrete_global_state}--\eqref{eq:discrete_cost_global} 
and write $\mathbf X_s := \mathbf X_s^N$, 
$\boldsymbol\alpha_s := \boldsymbol\alpha_s^N$ for brevity. We further 
introduce the shorthand notations
\[
\mathbf b_s := \mathbf A_s^N + \mathbf B_s^N\mathbf X_s + \mathbf C_s^N\boldsymbol\alpha_s,
\qquad
\boldsymbol\sigma_s^r := \mathbf D_s^{r,N} + \mathbf E_s^{r,N}\mathbf X_s 
+ \mathbf F_s^{r,N}\boldsymbol\alpha_s,
\qquad
\boldsymbol\sigma_s^0 := \boldsymbol\Sigma_s^{0,N},
\]
so that the global state equation reads $d\mathbf X_s = \mathbf b_s\,ds + \sum_{r=1}^N \boldsymbol\sigma_s^r\,dW_s^{r,N}  + \boldsymbol\sigma_s^0\,dW_s^0$.
\medskip
\noindent
\textbf{Step 1: The auxiliary process.}
Consider the process
\begin{align}
\mathcal S_s^{N,\alpha}
&:=
(\mathbf X_s)^\top P_s^N \mathbf X_s
+ 2(p_s^N)^\top \mathbf X_s
+ q_s^N
+ \int_t^s
\Big[
(\mathbf X_r)^\top \mathbf Q_r^N \mathbf X_r
+ (\boldsymbol\alpha_r + \boldsymbol\iota^N_r)^\top \mathbf R_r^N (\boldsymbol\alpha_r + \boldsymbol\iota^N_r)
\Big]\,dr.
\label{eq:app_S_discrete}
\end{align}
By the terminal conditions $P_T^N = \mathbf H^N$, $p_T^N = 0$, $q_T^N = 0$ 
and the definition of the cost \eqref{eq:discrete_cost_global},
\begin{equation}\label{eq:app_terminal_S}
\E\left[\mathcal S_T^{N,\alpha} \,\big|\, \Fc_t^0\right]
= J^N(t,\xi^N,\alpha^N).
\end{equation}
On the other hand,
\begin{equation}
\label{eq:app_initial_S}
\E\left[
\mathcal S_t^{N,\alpha}
\;\middle|\;
\Fc_t^0
\right]
=
\E\left[
(\xi^N)^\top P_t^N\xi^N
+
2(p_t^N)^\top\xi^N
\;\middle|\;
\Fc_t^0
\right]
+
q_t^N
=
\mathcal V_t^N(\xi^N).
\end{equation}
Therefore, to establish the fundamental relation 
\eqref{eq:fundamental_relation_discrete}, it suffices to compute 
$\mathbb{E}[\mathcal S_T^{N,\alpha} - \mathcal S_t^{N,\alpha} \mid \Fc_t^0]$.

\medskip
\noindent
\textbf{Step 2: Itô's formula on $\mathcal S^{N,\alpha}$.}
We apply Itô's formula to each of the three building blocks of 
$\mathcal S^{N,\alpha}$.

\smallskip
\noindent
\textit{Quadratic term.}
By the product rule applied to $(\mathbf X_s)^\top P_s^N \mathbf X_s$, 
the dynamics \eqref{eq:discrete_backward_dynamics} of $P^N$, and the 
state equation,
\begin{align}
d\big[(\mathbf X_s)^\top P_s^N \mathbf X_s\big]
=
&\Big[
2(\mathbf b_s)^\top P_s^N \mathbf X_s
- (\mathbf X_s)^\top \mathscr F_s^{P,N} \mathbf X_s
+ \sum_{r=1}^N (\boldsymbol\sigma_s^r)^\top P_s^N \boldsymbol\sigma_s^r
+ \langle \boldsymbol\sigma_s^{0}, P_s^N\boldsymbol\sigma_s^{0}\rangle_{\mathrm F} +2\langle \mathbf X_s, Z_s^{P,N}\boldsymbol\sigma_s^0\rangle
\Big]\,ds
+ d\mathcal M_s^{(1)},
\label{eq:app_dXPX}
\end{align}
where $\mathcal M^{(1)}$ is a local martingale and where we used: The drift contribution $2(\mathbf b_s)^\top P_s^N \mathbf X_s$ from $d\mathbf X_s$; the drift contribution $-(\mathbf X_s)^\top \mathscr F_s^{P,N} \mathbf X_s$ from $dP_s^N$; the quadratic variation contributions from the idiosyncratic noises  $dW^{r,N}$, namely $\sum_r (\boldsymbol\sigma_s^r)^\top P_s^N \boldsymbol\sigma_s^r$; the quadratic variation contribution from the common noise $dW^0$,  namely $\langle\boldsymbol\sigma_s^0,P_s^N\boldsymbol\sigma_s^0\rangle_{\rm F}$; the cross-variation between $dP_s^N$ and the common-noise component  of $d\mathbf X_s$, namely $2(\mathbf X_s)^\top Z_s^{P,N} \boldsymbol\sigma_s^0$.

\smallskip
\noindent
\textit{Linear term.}
Similarly, applying Itô's formula to $2(p_s^N)^\top \mathbf X_s$,
\begin{equation}\label{eq:app_dpX}
d\big[2(p_s^N)^\top \mathbf X_s\big]
=
\Big[
-2(\mathscr F_s^{p,N})^\top \mathbf X_s
+ 2(p_s^N)^\top \mathbf b_s
+ 2\langle Z_s^{p,N},\boldsymbol\Sigma_s^{0,N}\rangle_{\rm F}
\Big]\,ds
+ d\mathcal M_s^{(2)},
\end{equation}
where $\mathcal M^{(2)}$ is a local martingale.

\smallskip
\noindent
\textit{Constant term.}
By \eqref{eq:discrete_backward_dynamics},
\begin{equation}\label{eq:app_dq}
dq_s^N = -\mathscr F_s^{q,N}\,ds + Z_s^{q,N}\,dW_s^0.
\end{equation}

\medskip
\noindent
\textbf{Step 3: Total drift of $\mathcal S^{N,\alpha}$.}
Combining \eqref{eq:app_dXPX}--\eqref{eq:app_dq} with the running cost in 
\eqref{eq:app_S_discrete}, the drift of $\mathcal S_s^{N,\alpha}$ is
\begin{align}
\mathcal D_s^{N,\alpha}
:=
&\;2(\mathbf b_s)^\top P_s^N \mathbf X_s
- (\mathbf X_s)^\top \mathscr F_s^{P,N} \mathbf X_s
+ \sum_{r=1}^N (\boldsymbol\sigma_s^r)^\top P_s^N \boldsymbol\sigma_s^r
+ \langle\boldsymbol\sigma_s^0,P_s^N\boldsymbol\sigma_s^0\rangle_{\rm F}
+ 2(\mathbf X_s)^\top Z_s^{P,N} \boldsymbol\sigma_s^0
\notag\\
&-2(\mathscr F_s^{p,N})^\top \mathbf X_s
+ 2(p_s^N)^\top \mathbf b_s
+ 2\langle Z_s^{p,N},\boldsymbol\Sigma_s^{0,N}\rangle_{\rm F}
- \mathscr F_s^{q,N}
+ (\mathbf X_s)^\top \mathbf Q_s^N \mathbf X_s
+ (\boldsymbol\alpha_s + \boldsymbol\iota_s^N)^\top \mathbf R_s^N (\boldsymbol\alpha_s + \boldsymbol\iota_s^N).
\label{eq:app_drift_total}
\end{align}

\medskip
\noindent
\textbf{Step 4: Expansion in $\mathbf X_s$ and $\boldsymbol\alpha_s$.}
We now expand the products involving $\mathbf b_s$ and $\boldsymbol\sigma_s^r$ 
in terms of $\mathbf X_s$, $\boldsymbol\alpha_s$, and the additive terms.

\smallskip
\noindent
\textit{Drift contribution :} $2(\mathbf b_s)^\top P_s^N \mathbf X_s
=
2(\mathbf A_s^N)^\top P_s^N \mathbf X_s
+ 2(\mathbf X_s)^\top (\mathbf B_s^N)^\top P_s^N \mathbf X_s
+ 2(\boldsymbol\alpha_s)^\top (\mathbf C_s^N)^\top P_s^N \mathbf X_s$.

\smallskip
\noindent
\textit{Idiosyncratic contribution :} $\sum_{r=1}^N (\boldsymbol\sigma_s^r)^\top P_s^N \boldsymbol\sigma_s^r
=
\sum_{r=1}^N
\big(
\mathbf D_s^{r,N} + \mathbf E_s^{r,N}\mathbf X_s + \mathbf F_s^{r,N}\boldsymbol\alpha_s
\big)^\top
P_s^N
\big(
\mathbf D_s^{r,N} + \mathbf E_s^{r,N}\mathbf X_s + \mathbf F_s^{r,N}\boldsymbol\alpha_s
\big)$.

\bigskip
\noindent Expanding the square and collecting terms by their dependence in 
$(\mathbf X_s,\boldsymbol\alpha_s)$, we have : A quadratic term in $\mathbf X_s$:  $(\mathbf X_s)^\top \big[\sum_r (\mathbf E_s^{r,N})^\top P_s^N \mathbf E_s^{r,N}\big] \mathbf X_s$; a quadratic term in $\boldsymbol\alpha_s$:  $(\boldsymbol\alpha_s)^\top \big[\sum_r (\mathbf F_s^{r,N})^\top P_s^N \mathbf F_s^{r,N}\big] \boldsymbol\alpha_s$; a cross term $\mathbf X_s\boldsymbol\alpha_s$:\\
$2(\boldsymbol\alpha_s)^\top \big[\sum_r (\mathbf F_s^{r,N})^\top P_s^N \mathbf E_s^{r,N}\big] \mathbf X_s$; linear terms in $\mathbf X_s$: 
$2(\mathbf X_s)^\top \sum_r (\mathbf E_s^{r,N})^\top P_s^N \mathbf D_s^{r,N}$; linear terms in $\boldsymbol\alpha_s$: 
$2(\boldsymbol\alpha_s)^\top \sum_r (\mathbf F_s^{r,N})^\top P_s^N \mathbf D_s^{r,N}$; and constant term: $\sum_r (\mathbf D_s^{r,N})^\top P_s^N \mathbf D_s^{r,N}$.

\bigskip
\noindent
\textit{Linear term contribution :} $2(p_s^N)^\top \mathbf b_s
=
2(p_s^N)^\top \mathbf A_s^N
+ 2(\mathbf X_s)^\top (\mathbf B_s^N)^\top p_s^N
+ 2(\boldsymbol\alpha_s)^\top (\mathbf C_s^N)^\top p_s^N$.

\smallskip
\noindent
\textit{Running control cost :} $(\boldsymbol\alpha_s + \boldsymbol\iota_s^N)^\top \mathbf R_s^N (\boldsymbol\alpha_s + \boldsymbol\iota_s^N)
=
(\boldsymbol\alpha_s)^\top \mathbf R_s^N \boldsymbol\alpha_s
+ 2(\boldsymbol\alpha_s)^\top \mathbf R_s^N \boldsymbol\iota_s^N
+ (\boldsymbol\iota_s^N)^\top \mathbf R_s^N \boldsymbol\iota_s^N$.

\medskip
\noindent
\textbf{Step 5: Collecting the $\boldsymbol\alpha_s$-dependent terms and 
square completion.}
Combining the $\boldsymbol\alpha_s$-dependent contributions from Step 4, 
the part of $\mathcal D_s^{N,\alpha}$ depending on $\boldsymbol\alpha_s$ is
\begin{align}
\chi_s^N(\boldsymbol\alpha_s)
:=
&\;(\boldsymbol\alpha_s)^\top \mathcal O_s^N \boldsymbol\alpha_s
+ 2(\boldsymbol\alpha_s)^\top \big( \mathcal U_s^N \mathbf X_s + \boldsymbol\Gamma_s^N \big),
\label{eq:app_chi_discrete}
\end{align}
where $\mathcal O_s^N$, $\mathcal U_s^N$, $\boldsymbol\Gamma_s^N$ are 
defined in \eqref{eq:def_O_discrete}--\eqref{eq:def_Gamma_discrete}. 
By Remark \ref{rem:O_inverse_discrete}, $\mathcal O_s^N$ is invertible 
and positive-definite. Completing the square:
\begin{align}
\chi_s^N(\boldsymbol\alpha_s)
=
&\;\Big(\boldsymbol\alpha_s - \widehat{\boldsymbol\alpha}_s^{N,\alpha}\Big)^\top
\mathcal O_s^N
\Big(\boldsymbol\alpha_s - \widehat{\boldsymbol\alpha}_s^{N,\alpha}\Big)
-\big(\mathcal U_s^N \mathbf X_s + \boldsymbol\Gamma_s^N\big)^\top
(\mathcal O_s^N)^{-1}
\big(\mathcal U_s^N \mathbf X_s + \boldsymbol\Gamma_s^N\big),
\label{eq:app_square_completion_discrete}
\end{align}
where 
$\widehat{\boldsymbol\alpha}_s^{N,\alpha} 
:= -(\mathcal O_s^N)^{-1}(\mathcal U_s^N \mathbf X_s + \boldsymbol\Gamma_s^N)$.

\medskip
\noindent
\textbf{Step 6: Identification of the backward drivers.}
We now collect the remaining terms in $\mathcal D_s^{N,\alpha}$ (those not depending on $\boldsymbol\alpha_s$) and group them by their dependence in $\mathbf X_s$.

\medskip
\noindent
\textit{Quadratic in $\mathbf X_s$ :} $(\mathbf X_s)^\top
\Big[
\mathbf Q_s^N
+ (\mathbf B_s^N)^\top P_s^N + P_s^N \mathbf B_s^N
+ \sum_{r=1}^N (\mathbf E_s^{r,N})^\top P_s^N \mathbf E_s^{r,N}
- (\mathcal U_s^N)^\top (\mathcal O_s^N)^{-1} \mathcal U_s^N
- \mathscr F_s^{P,N}
\Big]
\mathbf X_s$,\\

\noindent where the $-(\mathcal U_s^N)^\top (\mathcal O_s^N)^{-1} \mathcal U_s^N$ 
term comes from the square-completion correction in 
\eqref{eq:app_square_completion_discrete}. Since $P^N$ solves 
\eqref{eq:discrete_Riccati_P}, the expression in brackets is identically 
zero.

\medskip
\noindent
\textit{Linear in $\mathbf X_s$ :} $2(\mathbf X_s)^\top
\Big[
P_s^N \mathbf A_s^N
+ Z_s^{P,N} \boldsymbol\sigma_s^0
+ (\mathbf B_s^N)^\top p_s^N
+ \sum_{r=1}^N (\mathbf E_s^{r,N})^\top P_s^N \mathbf D_s^{r,N}
- (\mathcal U_s^N)^\top (\mathcal O_s^N)^{-1} \boldsymbol\Gamma_s^N
- \mathscr F_s^{p,N}
\Big]$,

\noindent where $\boldsymbol\sigma_s^0 = \boldsymbol\Sigma_s^{0,N}$. Since $p^N$ 
solves \eqref{eq:discrete_linear_p}, the expression in brackets is 
identically zero.

\medskip
\noindent
\textit{Constant :} $2(\mathbf A_s^N)^\top p_s^N
+ \sum_{r=1}^N (\mathbf D_s^{r,N})^\top P_s^N \mathbf D_s^{r,N}
+ \langle \boldsymbol\Sigma_s^{0,N}, P_s^N\boldsymbol\Sigma_s^{0,N}\rangle_{\mathrm F} +2\langle Z_s^{p,N},\boldsymbol\Sigma_s^{0,N}\rangle_{\mathrm F}
+ (\boldsymbol\iota_s^N)^\top \mathbf R_s^N \boldsymbol\iota_s^N
- (\boldsymbol\Gamma_s^N)^\top (\mathcal O_s^N)^{-1} \boldsymbol\Gamma_s^N
- \mathscr F_s^{q,N}$

\noindent Since $q^N$ solves \eqref{eq:discrete_scalar_q}, this expression is 
identically zero.

\medskip
\noindent
\textbf{Step 7: Conclusion.}
Combining Steps 5 and 6, the total drift of $\mathcal S^{N,\alpha}$ 
reduces to
\begin{equation}\label{eq:app_final_drift}
\mathcal D_s^{N,\alpha}
=
\big(\boldsymbol\alpha_s - \widehat{\boldsymbol\alpha}_s^{N,\alpha}\big)^\top
\mathcal O_s^N
\big(\boldsymbol\alpha_s - \widehat{\boldsymbol\alpha}_s^{N,\alpha}\big).
\end{equation}
Integrating from $t$ to $T$, taking conditional expectation with respect 
to $\Fc_t^0$, and using the local martingale parts from 
\eqref{eq:app_dXPX}, \eqref{eq:app_dpX}, \eqref{eq:app_dq}, with the fact that process $\mathcal S^{N,\alpha}$ is of class $(D)$: $P^N$ is bounded, $p^N$ and $\mathbf X$ belong to $S^2$, $q^N$ is of class $(D)$, and the running cost has integrable total variation. The completed-square drift is integrable. Hence the total local martingale is uniformly integrable, and taking conditional expectation gives
\[
\E\left[\mathcal S_T^{N,\alpha} - \mathcal S_t^{N,\alpha} 
\,\big|\, \Fc_t^0\right]
=
\E\left[\int_t^T 
\big(\boldsymbol\alpha_s - \widehat{\boldsymbol\alpha}_s^{N,\alpha}\big)^\top
\mathcal O_s^N
\big(\boldsymbol\alpha_s - \widehat{\boldsymbol\alpha}_s^{N,\alpha}\big)\,ds
\,\Big|\, \Fc_t^0\right].
\]
Combining with \eqref{eq:app_terminal_S} and \eqref{eq:app_initial_S} 
yields the fundamental relation \eqref{eq:fundamental_relation_discrete}.
\end{proof}

\section{Proof of Lemma~\ref{lem:R1}}
\label{app:R1}

For notational simplicity, we suppress the dependence on $(t,\omega)$. We use $T_G$ for the integral operator associated with a kernel $G$, and $M_a$ for multiplication by $a$. We repeatedly use $\|T_G\|_{\mathrm{op}} \leq \|T_G\|_{\mathrm{HS}} = \|G\|_{L^2(I^2)}$. 

\noindent Set \[a_t^N := \sum_{\Phi\in\{A,B,C,D,E,F,\iota,R,Q\}} \|\Phi_t^N-\Phi_t\|_{L^2(I)}, \qquad g_t^N := \sum_{G\in\{G^B,G^E,G^Q\}} \|G_t^N-G_t\|_{L^2(I^2)}, \qquad b_t^N := \|\mathbf1_{\mathcal D_N}\bar K_t\|_{L^2(I^2)}\] 
Assumption~\ref{ass:reg_consistency} gives $\operatorname {ess\,sup}_{(t,\omega)}a_t^N \leq C\delta_N^K$, $\operatorname {ess\,sup}_{(t,\omega)}g_t^N \leq C\varepsilon_N^{\mathrm{ker}}$ and by Lemma~\ref{lem:diagonal_band_modulus}, $\E\sup_{t\in[0,T]}|b_t^N|^2 \leq (d_N^{\mathrm{diag}})^2$.

\subsection{Local Riccati driver}

Write
\begin{align*}
\mathfrak F_t^{K,\mathrm{tot},N}
-
\mathcal F^K(t,K_t)
={}&
\Big[
\mathfrak F^{K,\mathrm{loc},N}(t,K_t^N)
-
\mathfrak F^{K,\mathrm{loc},N}(t,K_t)
\Big]
+
\Big[
\mathfrak F^{K,\mathrm{loc},N}(t,K_t)
-
\mathcal F^K(t,K_t)
\Big]
+
\mathfrak d_t^{K,N}.
\end{align*}
On the deterministic bounded set containing $K^N$ and $K$, the local Riccati driver is uniformly Lipschitz. Hence $\left\| \mathfrak F^{K,\mathrm{loc},N}(t,K_t^N) - \mathfrak F^{K,\mathrm{loc},N}(t,K_t) \right\|_{L^2(I)} \leq C\|K_t^N-K_t\|_{L^2(I)}$. The second bracket is controlled by $Ca_t^N$, using the resolvent and the uniform coercivity of $O^N$ and $O$. Finally, Lemma~\ref{lem:diagonal_remainder_bounds} gives $\operatorname*{ess\,sup}_{(t,\omega)} \|\mathfrak d_t^{K,N}\|_{L^2(I)} \leq C\sqrt{h_N}$. Thus item \textup{(i)} holds with $\rho_t^{K,N} := Ca_t^N+\|\mathfrak d_t^{K,N}\|_{L^2(I)}$, and $\E\int_0^T|\rho_t^{K,N}|^2dt \leq Cr_N^2$.

\subsection{Projected interaction Riccati driver}

Set $\Delta\bar K_t^{[N]} := \bar K_t^N-\bar K_t^{[N]}$. Since $\bar K^N$ vanishes on $\mathcal D_N$, $\|\bar K_t^N-\bar K_t\|_{L^2(I^2)} \leq \|\Delta\bar K_t^{[N]}\|_{L^2(I^2)} + b_t^N$. We first consider affine terms. For example, \[(B_t^N)^\top\bar K_t^N-B_t^\top\bar K_t
={}
(B_t^N)^\top(\bar K_t^N-\bar K_t)
+
(B_t^N-B_t)^\top\bar K_t\] 
The first term is bounded by $C\left( \|\Delta\bar K_t^{[N]}\|_{L^2(I^2)} + b_t^N \right)$. For the second one, the row estimate in
Assumption~\ref{ass:reg_consistency}\textup{(A-row/col)} gives
\[
\|(B_t^N-B_t)^\top\bar K_t\|_{L^2(I^2)}^2
\leq
\|B_t^N-B_t\|_{L^2(I)}^2
\operatorname*{ess\,sup}_{u\in I}
\int_I|\bar K_t(u,v)|^2dv.
\]
Terms acting in the second label variable are controlled by the
corresponding column estimate.For terms containing both a local Riccati component and an interaction kernel, we use the decomposition $K_t^NG_t^N-K_tG_t
=
K_t^N(G_t^N-G_t)
+
(K_t^N-K_t)G_t$. This ordering is important. The first term is controlled because $K^N$ is pointwise bounded, while the second is controlled because the limiting kernel $G$ is uniformly bounded. No pointwise bound on the
discrete kernel $G^N$ is needed. 

\noindent Kernel-composition terms are treated using, for example, 
\[
(T_{G_t^N})^\ast T_{\bar K_t^N}
-
(T_{G_t})^\ast T_{\bar K_t}
=
(T_{G_t^N}-T_{G_t})^\ast T_{\bar K_t^N}
+
(T_{G_t})^\ast
(T_{\bar K_t^N}-T_{\bar K_t})
\]
First term gives $\|
(T_{G_t^N}-T_{G_t})^\ast T_{\bar K_t^N}
\|_{\mathrm{HS}}
\leq
\|G_t^N-G_t\|_{L^2(I^2)}
\|T_{\bar K_t^N}\|_{\mathrm{op}}$, the second one is bounded by $C\left( \|\Delta\bar K_t^{[N]}\|_{L^2(I^2)} + b_t^N \right)$.

\noindent We next distinguish the full auxiliary kernel from the exact
off-diagonal feedback kernel. Set $\widetilde V_t^N
:=
(C_t^N)^\top\bar K_t^N
+
(F_t^N)^\top K_t^NG_t^{E,N}$. The exact embedded finite-dimensional gain is $V_t^N
:=
\Pi_N^\circ\widetilde V_t^N$, whereas
$V_t
=
C_t^\top\bar K_t
+
F_t^\top K_tG_t^E$. The full kernel is also the convenient object for comparing the backward drivers. For $i\ne j$, the exact off-diagonal driver is 
\[
\mathfrak F_{ij}^{\bar K,N}
=\Psi_{ij}^N
 -(U_i^N)^\top(O_i^N)^{-1}\widetilde V_{ij}^N
 -(\widetilde V_{ji}^N)^\top(O_j^N)^{-1}U_j^N
 -h_N\sum_{\ell=1}^N
 (\widetilde V_{\ell i}^N)^\top(O_\ell^N)^{-1}
 \widetilde V_{\ell j}^N
\]
Here $\Psi^N$ is \eqref{eq:Psi_continuum} evaluated at the embedded discrete coefficients, $K^N$ and $\bar K^N$, with integrals replaced by the exact cell sums. Indeed, the summands $\ell=i$ and $\ell=j$ combine with the two mixed terms because $\widehat U_i^N=U_i^N+h_N\widetilde V_{ii}^N$. Thus the exact embedded interaction driver is the off-diagonal projection of this full expression. 

\noindent Consequently,
$V_t^N-V_t
=
\Pi_N^\circ
\big(
\widetilde V_t^N-V_t
\big)
-
\mathbf 1_{\mathcal D_N}V_t$. The unprojected difference admits the decomposition
\begin{align*}
\widetilde V_t^N-V_t
={}&
(C_t^N)^\top(\bar K_t^N-\bar K_t)
+
(C_t^N-C_t)^\top\bar K_t
+
(F_t^N)^\top K_t^N
(G_t^{E,N}-G_t^E)
+
\big(
(F_t^N)^\top K_t^N-F_t^\top K_t
\big)G_t^E.
\end{align*}
The first term is bounded by $C\left(
\|\Delta\bar K_t^{[N]}\|_{L^2(I^2)}
+
b_t^N
\right)$. The second term is controlled by the row estimate in Assumption~\ref{ass:reg_consistency}\textup{(A-row/col)}. The third term is controlled by the pointwise bound on $K^N$ and the kernel consistency error, while the fourth one is controlled by the boundedness of the limiting kernel $G^E$. Moreover, $\mathbf 1_{\mathcal D_N}V_t = C_t^\top \mathbf 1_{\mathcal D_N}\bar K_t + F_t^\top K_t \mathbf 1_{\mathcal D_N}G_t^E$. Since $G^E$ is uniformly bounded and $|\mathcal D_N|=h_N$,
$\|
\mathbf 1_{\mathcal D_N}V_t
\|_{L^2(I^2)}
\leq
Cb_t^N+C\sqrt{h_N}$. We therefore obtain
\begin{align}
\label{eq:R1_V_difference}
\|V_t^N-V_t\|_{L^2(I^2)}
\leq{}&
C\|\Delta\bar K_t^{[N]}\|_{L^2(I^2)}
+
C\|K_t^N-K_t\|_{L^2(I)}
+
Cb_t^N
+
\rho_t^{V,N},
\end{align}
where $\E\int_0^T|\rho_t^{V,N}|^2dt
\leq
C\left(
(\delta_N^K)^2
+
(\varepsilon_N^{\mathrm{ker}})^2
+
h_N
\right)
\leq
Cr_N^2$.  The preceding decomposition also gives \[\|\widetilde V_t^N-V_t\|_{L^2(I^2)} \le C\|\Delta\bar K_t^{[N]}\|_{L^2(I^2)} +C b_t^N+C g_t^N+C\delta_N^K R_t\]
where $R_t=1+\|\bar K_t\|_{\rm row}+\|\bar K_t\|_{\rm col}$ and $\mathbb E\sup_t R_t^2<\infty$. The full auxiliary operators $T_{\widetilde V_t^N}$ are uniformly bounded, since $T_{\widetilde V_t^N}=M_{(C_t^N)^\top}T_{\bar K_t^N} +M_{(F_t^N)^\top K_t^N}T_{G_t^{E,N}}$. Set $M_t^N=M_{(O_t^N)^{-1}}$ and $M_t=M_{O_t^{-1}}$.
Then
\[
\begin{aligned}
T_{\widetilde V_t^N}^*M_t^NT_{\widetilde V_t^N}
 -T_{V_t}^*M_tT_{V_t}
={}&(T_{\widetilde V_t^N}-T_{V_t})^*M_t^NT_{\widetilde V_t^N} +T_{V_t}^*M_t^N(T_{\widetilde V_t^N}-T_{V_t}) +T_{V_t}^*(M_t^N-M_t)T_{V_t}.
\end{aligned}
\]
The first two terms are bounded in Hilbert--Schmidt norm by
$C\|\widetilde V_t^N-V_t\|_{L^2(I^2)}$. The last term is bounded by
$C\|(O_t^N)^{-1}-O_t^{-1}\|_{L^2(I)}\|V_t\|_{\rm row}$,
hence by $C\delta_N^K R_t$.
For the mixed term, write
\[
\begin{aligned}
M_{U_t^N}^*M_t^NT_{\widetilde V_t^N}-M_{U_t}^*M_tT_{V_t}
={}&M_{U_t^N}^*M_t^N(T_{\widetilde V_t^N}-T_{V_t})+(M_{U_t^N}^*M_t^N-M_{U_t}^*M_t)T_{V_t}.
\end{aligned}
\]
The same bounds apply; the other mixed term is its adjoint. Combining with the affine estimates and applying $\Pi_N^\circ$ gives item (ii) with $\rho_t^{\bar K,N}=C(b_t^N+g_t^N+\delta_N^K R_t)$ and $\mathbb E\int_0^T|\rho_t^{\bar K,N}|^2dt\le Cr_N^2$.

\subsection{Linear driver} 
Using the full auxiliary kernel, the exact rescaled linear driver is
\[
\begin{aligned}
\mathfrak F_t^{Y,N}={}&
 K_t^NA_t^N+T_{\bar K_t^N}A_t^N
 +Z_t^{K,N}\Sigma_t^{0,N}
 +T_{Z_t^{\bar K,N}}\Sigma_t^{0,N}
 +(B_t^N)^\top Y_t^N+(T_{G_t^{B,N}})^*Y_t^N
 +(E_t^N)^\top K_t^ND_t^N
 +(T_{G_t^{E,N}})^*(K_t^ND_t^N)\\
 &-(U_t^N)^\top(O_t^N)^{-1}\Gamma_t^N
 -(T_{\widetilde V_t^N})^*((O_t^N)^{-1}\Gamma_t^N).
\end{aligned}
\]
The diagonal contribution of the last operator combines with $U^N$ exactly as in the interaction Riccati driver. Thus no additional unidentified diagonal source is hidden in $f^N$. Write the two drivers as $\mathfrak F_t^{Y,N}
=
\mathcal A_t^NY_t^N+f_t^N$ and 
$\mathcal F_t^Y
=
\mathcal A_tY_t+f_t$. Then
\[
\mathfrak F_t^{Y,N}-\mathcal F_t^Y
=
\mathcal A_t^N(Y_t^N-Y_t)
+
(\mathcal A_t^N-\mathcal A_t)Y_t
+
(f_t^N-f_t).
\]
The uniform operator estimate gives $\|\mathcal A_t^N(Y_t^N-Y_t)\|_{L^2(I)}
\leq
C\|Y_t^N-Y_t\|_{L^2(I)}$. The pointwise products involving the limiting $Y$ are placed in the
remainder. For example, $\|(K_t^N-K_t)Y_t\|_{L^2(I)}
\leq
\|K_t^N-K_t\|_{L^2(I)}
\|Y_t\|_{L^\infty(I)}$.
Although the factor $\|Y_t\|_{L^\infty(I)}$ is random, Lemma~\ref{lem:K_uniform} gives the deterministic estimate $\operatorname*{ess\,sup}_{(t,\omega)}
\|K_t^N-K_t\|_{L^2(I)}
\leq
C\delta_N^K$. Therefore,
\begin{align}
\label{eq:R1_random_Y_product}
\E\int_0^T
\|(K_t^N-K_t)Y_t\|_{L^2(I)}^2dt
&\leq
C(\delta_N^K)^2
\E\int_0^T
\|Y_t\|_{L^\infty(I)}^2dt
\leq
C(\delta_N^K)^2.
\end{align}
The same argument applies to local coefficient errors and to the differences of $U$, $\widehat U$, and $O^{-1}$. Kernel terms acting on $Y$ are controlled in operator norm. For example, $\|(T_{G_t^{B,N}}-T_{G_t^B})^\ast Y_t\|_{L^2(I)} \leq \varepsilon_N^{\mathrm{ker}}\|Y_t\|_{L^2(I)}$. The deterministic $S^\infty(L^2)$-bound on $Y$ from
Lemma~\ref{lem:remaining_bounds} gives the required squared-integrable
remainder.

\noindent It remains to treat the common-noise terms in the source. Define $e_N^\Sigma(u)
:=
\operatorname*{ess\,sup}_{(t,\omega)}
|\Sigma_t^{0,N}(u)-\Sigma_t^0(u)|$. Then Assumption~\ref{ass:reg_consistency}\textup{(A-Sigma)} gives $\|e_N^\Sigma\|_{L^2(I)}
\leq
\varepsilon_N^\Sigma$.
Using the labelwise BMO estimate of Lemma~\ref{lem:ZK_BMO},
\begin{align}\label{eq:R1_ZK_Sigma}
&
\E\int_0^T
\|Z_t^K(\Sigma_t^{0,N}-\Sigma_t^0)\|_{L^2(I)}^2dt
\leq
\int_I
|e_N^\Sigma(u)|^2
\E\int_0^T|Z_t^K(u)|^2dt\,du
\leq
C(\varepsilon_N^\Sigma)^2.
\end{align}
Moreover, $\|(Z_t^{K,N}-Z_t^K)\Sigma_t^{0,N}\|_{L^2(I)} \leq C\|Z_t^{K,N}-Z_t^K\|_{L^2(I)}$. For the interaction martingale term,
\begin{align*}
&
T_{Z_t^{\bar K,N}}\Sigma_t^{0,N}
-
T_{Z_t^{\bar K}}\Sigma_t^0
=
T_{Z_t^{\bar K,N}-Z_t^{\bar K}}\Sigma_t^{0,N}
+
T_{Z_t^{\bar K}}
(\Sigma_t^{0,N}-\Sigma_t^0).
\end{align*}
Hence $\|
T_{Z_t^{\bar K,N}-Z_t^{\bar K}}\Sigma_t^{0,N}
\|_{L^2(I)}
\leq
C\|Z_t^{\bar K,N}-Z_t^{\bar K}\|_{L^2(I^2)}$, whereas $\|T_{Z_t^{\bar K}}
(\Sigma_t^{0,N}-\Sigma_t^0)
\|_{L^2(I)}
\leq
\|Z_t^{\bar K}\|_{L^2(I^2)}
\|\Sigma_t^{0,N}-\Sigma_t^0\|_{L^2(I)}$. Consequently, $\E\int_0^T
\|
T_{Z_t^{\bar K}}
(\Sigma_t^{0,N}-\Sigma_t^0)
\|_{L^2(I)}^2dt
\leq
C(\varepsilon_N^\Sigma)^2$. All the remaining source terms are combinations of the preceding local-multiplier, kernel-operator, and Hilbert--Schmidt estimates.  

\noindent Collecting them proves item \textup{(iii)} with $\E\int_0^T|\rho_t^{Y,N}|^2dt
\leq
Cr_N^2$.

\subsection{Integrated scalar driver}

Let $\Gamma_t^N
=
(C_t^N)^\top Y_t^N
+
(F_t^N)^\top K_t^ND_t^N
+
R_t^N\iota_t^N$ denote the embedded discrete affine gain, and let $\Gamma_t$ be its continuum counterpart. Direct expansion gives $\|\Gamma_t^N-\Gamma_t\|_{L^2(I)}
\leq
C\|Y_t^N-Y_t\|_{L^2(I)}
+
C\|K_t^N-K_t\|_{L^2(I)}
+
\rho_t^{\Gamma,N}$, where the products between local coefficient errors and $Y$ are estimated as in \eqref{eq:R1_random_Y_product}, so that $\E\int_0^T|\rho_t^{\Gamma,N}|^2dt
\leq
Cr_N^2$.

\noindent For the quadratic gain term, the trilinear decomposition gives
\begin{align*}
&
\left|
\int_I
(\Gamma_t^N)^\top(O_t^N)^{-1}\Gamma_t^N\,du
-
\int_I
\Gamma_t^\top O_t^{-1}\Gamma_t\,du
\right|
\leq
C\|\Gamma_t^N-\Gamma_t\|_{L^2(I)}
+
\|(O_t^N)^{-1}-O_t^{-1}\|_{L^2(I)}
\|\Gamma_t\|_{L^4(I)}^2.
\end{align*}
Since $\|\Gamma_t\|_{L^4(I)}^2 \leq \|\Gamma_t\|_{L^\infty(I)} \|\Gamma_t\|_{L^2(I)}$, the deterministic bound $\operatorname*{ess\,sup}_{(t,\omega)} \|\Gamma_t\|_{L^2(I)}\le C$, the $S^2(L^\infty)$-assumption on $Y$, and Lemma~\ref{lem:K_uniform} show that the last term is square-integrable with contribution bounded by $C(\delta_N^K)^2$. No fourth probabilistic moment of $Y$ is needed. The common-noise martingale term satisfies
\begin{align*}
&
\left|
\int_I
\left\langle
Z_t^Y,
\Sigma_t^{0,N}-\Sigma_t^0
\right\rangle_{\mathrm F}du
\right|
\leq
\|Z_t^Y\|_{L^2(I)}
\|\Sigma_t^{0,N}-\Sigma_t^0\|_{L^2(I)}
\leq
\varepsilon_N^\Sigma
\|Z_t^Y\|_{L^2(I)}.
\end{align*}
The term containing $Z_t^{Y,N}-Z_t^Y$ is bounded by $C\|Z_t^{Y,N}-Z_t^Y\|_{L^2(I)}$.  Finally, the double common-noise term is written in operator form: $\langle \Sigma_t^{0,N},
T_{\bar K_t^N}\Sigma_t^{0,N}
\rangle_{L^2(I)}
-
\langle
\Sigma_t^0,
T_{\bar K_t}\Sigma_t^0
\rangle_{L^2(I)}$. Adding and subtracting the two intermediate terms gives the bound $C\|\bar K_t^N-\bar K_t\|_{L^2(I^2)}
+
C\|\Sigma_t^{0,N}-\Sigma_t^0\|_{L^2(I)}$. The remaining terms $\langle D,KD\rangle_{\mathrm F}$, $\langle\Sigma^0,K\Sigma^0\rangle_{\mathrm F}$, $2A^\top Y$, $\iota^\top R\iota$ are handled by the same direct expansions, using boundedness of the
coefficients and Cauchy--Schwarz. We conclude that
\begin{align*}
|\ell_t^N-\ell_t|
\leq{}&
C\|K_t^N-K_t\|_{L^2(I)}
+
C\|\bar K_t^N-\bar K_t\|_{L^2(I^2)}
+
C\|Y_t^N-Y_t\|_{L^2(I)}
+
C\|Z_t^{Y,N}-Z_t^Y\|_{L^2(I)}
+
\rho_t^{q,N},
\end{align*}
with $\E\int_0^T|\rho_t^{q,N}|^2dt
\leq
Cr_N^2$. This proves item \textup{(iv)} and completes the proof of Lemma~\ref{lem:R1}.

\section*{Acknowledgements}
The author is deeply grateful to Samy Mekkaoui for numerous insightful discussions, valuable suggestions, and careful comments on earlier versions of this work, which greatly contributed to improving the paper. The author would also like to thank Imane Fakir for her comments and  unwavering support throughout the preparation of this work.

\bibliographystyle{plain}
\bibliography{SEbib}

\end{document}